\documentclass[11pt,oneside]{amsart}
\usepackage[utf8]{inputenc}
\usepackage[english,francais]{babel}
\usepackage[centertags]{amsmath}
\usepackage{amsmath}
\usepackage{amssymb}
\usepackage{mathabx}

\usepackage{mathrsfs}
\newcommand \A[1]{{\bf (#1)}}
\def\leftB{[\![}
\def\rightB{]\!]}

\usepackage{enumitem}
\usepackage{graphicx}

 \usepackage{dsfont}
 
\newcommand{\Tr}{{{\rm Tr}}}

\newcommand{\bb}{\mathbb}

\newcommand\R{\mathbb{R}}

\newcommand\N{\mathbb{N}}

\newcommand\F{\mathcal{F}}

\DeclareMathSymbol{\leqslant}{\mathrel}{AMSa}{"36}
\DeclareMathSymbol{\geqslant}{\mathrel}{AMSa}{"3E}

\def \N{{\bb N}}
\def \P{{\bb P}}

\def \R{{\bb R}}
\def \I{{\bb I}}

\def \E{{\bb E}}

\def\leftB{[\![}
\def\rightB{]\!]}

\def\G{\mathcal{G}}
\def\F{\mathcal{F}}

\def \ra{\rightarrow}

\newcounter{definition}[section]
\newcommand{\mysection}{\setcounter{equation}{0} \section}

\newtheorem{lemme}{Lemma}
\newtheorem{theo}{Theorem}
\newtheorem{remark}{Remark}

\newtheorem{prop}{Proposition}

\usepackage{geometry}
\usepackage{color}
 
\begin{document}

  \author{I. HONOR\'E} 
   \address{Universit\'e d'Evry Val d'Essonne}
   \email{igor.honore@univ-evry.fr} 
   \curraddr{Laboratoire de Math\'ematiques et Mod\'elisation d'\'Evry (LaMME), 23 Boulevard de France, 91037, Evry, France. }
   
   \author{S. MENOZZI}
   \address{Universit\'e d'\'Evry Val d'Essonne and Higher School of Economics}
   \email{stephane.menozzi@univ-evry.fr}
   \curraddr{Laboratoire de Math\'ematiques et Mod\'elisation d'\'Evry (LaMME), 23 Boulevard de France, 91037, \'Evry, France and Laboratory of Stochastic Analysis, HSE, Shabolovka 31, Moscow, Russian Federation.}
   
   \author{G. Pag\`es} 
   \address{Universit\'e Pierre et Marie Curie, Paris 6}
\email{gilles.pages@upmc.fr}
\curraddr{Laboratoire de Probabilit\'es et Mod\`eles Al\'eatoires. Campus de Jussieu, case 188, 4 pl. Jussieu 75252 Paris Cedex 5, France.}


   \title[Non-asymptotic Gaussian Estimates for Ergodic Approximations]{Non-asymptotic Gaussian Estimates for the Recursive Approximation of the Invariant Distribution of a Diffusion}


   \selectlanguage{english}
   \begin{abstract}
   We obtain non-asymptotic Gaussian concentration bounds for the difference between the invariant distribution $\nu $ of an ergodic Brownian diffusion process and the empirical distribution 
   of an approximating scheme with decreasing time step along a suitable class of  (smooth enough) test functions $f$ such that $f-\nu(f)$ is a coboundary of  the infinitesimal generator. We show that these bounds can still be improved when some suitable squared-norms of the diffusion coefficient also  lie in this class. We apply these estimates to design computable non-asymptotic confidence intervals for the approximating scheme. As a theoretical  application, we  finally derive non-asymptotic deviation bounds for the almost sure Central Limit Theorem.  
   \end{abstract}

   \keywords{Invariant distribution, diffusion processes, inhomogeneous Markov chains, non-asymptotic Gaussian concentration, almost sure Central Limit Theorem}

   \date{\today}

   \maketitle


\mysection{Introduction}


%
\subsection{Setting}
The aim of this article is to approach the invariant distribution of the solution of the diffusion equation:
\begin{equation}
\label{eq_diff}
\begin{split}
dY_t &= b(Y_t) dt + \sigma(Y_t) dW_t,\\
\end{split}
\end{equation}
where $(W_t)_{t\ge 0}$ is a Wiener process of dimension $r$ on a given filtered probability space $(\Omega,\G,(\G_t)_{t\ge 0},\P) $, $b : \R^d \rightarrow \R^d,$ and $ \sigma : \R^d \rightarrow \R^d \otimes \R^r $ are assumed to be Lipschitz continuous functions and to satisfy a  mean-reverting assumption \textcolor{black}{in the following sense}. If ${\mathcal A}$ denotes the infinitesimal generator  of the diffusion~\eqref{eq_diff}, there exists a  twice continuously differentiable Lyapunov function $V:\R^d\to (0,+\infty)$ such that $\lim_{|x|\to +\infty} V(x)=+\infty$ and ${\mathcal A} V\le \beta -\alpha V$ where $\beta\!\in \R$ and $\alpha >0$. Such a condition ensures the existence of an invariant distribution. We will also assume uniqueness of the invariant distribution, denoted from now by $\nu$. 
We refer to the monographs by \textcolor{black}{Khasminskii  \cite{khas:80} (see also its augmented second edition~\cite{khasminskii2011stochastic})}, 
or Villani~\cite{vill:09} \textcolor{black}{and to the survey paper \cite{page:01}}, for in-depth discussions on the conditions yielding such existence and uniqueness results.

\medskip
We introduce an approximation algorithm based on an Euler like discretization with decreasing time step,  which may use more general innovations than the Brownian increments. Namely, for the step sequence $(\gamma_k)_{k \geq 1} $ and $n\ge 0$, we define:
\begin{equation}\label{scheme}\tag{S}
X_{n+1}= X_n + \gamma_{n+1} b(X_n) + \sqrt{\gamma_{n+1}} \sigma(X_n) U_{n+1},
\end{equation}
where $X_0\in L^2(\Omega,\F_0,\P)$ and $(U_n)_{n \geq 1}$ is an i.i.d. sequence of centered random variables matching the moments of the Gaussian law on $\R^r $ up to order three, independent of $X_0$. 

We define the empirical (random) \textcolor{black}{occupation} measure of the scheme in the following way. For all $A\in {\mathcal B}(\R^d) $  (where ${\mathcal B}(\R^d)$ denotes the Borel $\sigma $-field on $\R^d$):

\begin{equation}\label{measure}
\nu_n(A) :=\nu_n(\omega,A) :=\frac{\sum_{k=1}^n \gamma_k \delta_{X_{k-1}(\omega)}(A)}{\sum_{k=1}^n \gamma_k}.
\end{equation}
The measure $\nu_n $ is here defined accordingly to the \textit{intrinsic} time scale of the scheme. \textcolor{black}{Namely, $\Gamma_n=\sum_{k=1}^n\gamma_k $ represents the current time associated with the Euler scheme \eqref{scheme}  after $n$ iterations}. Since we are interested in long time approximation, we consider steps $ (\gamma_k)_{k \geq 1}$ such that $\Gamma_n:=\sum_{k=1}^n \gamma_k \underset{n}{\rightarrow} + \infty $. We also assume $\gamma_k \underset{k}{\downarrow} 0 $. 
Observe that, for a bounded $\nu$-$a.s.$ continuous function $f$, it is proved   in~\cite{lamb:page:02} (see {\it e.g.}  Theorem~1), that:
\begin{equation}
\label{CV_ERGO}
\nu_n(f)=\frac{1}{ \Gamma_n}\sum_{k=1}^n \gamma_k f(X_{k-1}) \overset{a.s.}{\underset{n}{\longrightarrow}} \nu(f)=\int_{\R^d}f(x) \nu (dx),
\end{equation}
or equivalently that $\nu_n(\omega,\cdot) \overset{w}{\underset{n}{\longrightarrow}} \nu, \ \P(d\omega)-a.s.$ The above result can be seen as an inhomogeneous counterpart of \textit{stability results} discussed for homogeneous Markov chains in Duflo~\cite{dufl:90}. 
Intuitively, the decreasing steps make the approximation more and more accurate in long time and, therefore, the ergodic empirical mean of the scheme converges to the quantity of interest. Put it differently, there is no bias. This is a significant advantage w.r.t. a more \textit{naive} discretization method that would rely on a constant step scheme. Indeed, even if this latter approach gains in simplicity, taking $\gamma_k=h>0$ in~\eqref{scheme} would lead to replace the r.h.s. of~\eqref{CV_ERGO} by the quantity $\nu^h(f):= \int_{\R^d}f(x) \nu^h (dx)$, with $\nu^h $ denoting the invariant distribution of the Euler scheme with step $h$. In such a case, for the analysis to be complete, one needs to investigate the difference $\nu-\nu^h $ through the corresponding continuous and discrete Poisson problems. We refer to Talay \textit{et al.}~\cite{tala:tuba:90},~\cite{tala:02} for a precise presentation of this approach.

Now, once~\eqref{CV_ERGO} is available, the next question naturally concerns the rate of that convergence. This was originally
investigated by Lamberton and Pag\`es~\cite{lamb:page:02} for functions $f$ of the form  $f-\nu(f)={\mathcal A}\varphi $,
i.e. $f-\nu(f) $ is a \textit{coboundary} for ${\mathcal A}$. The specific reason for focusing on such a  class of functions is that an invariant distribution $\nu$ is characterized as a solution  in the distribution sense of the stationary Fokker-Planck equation
 ${\mathcal A}^*\nu=0 $ (where ${\mathcal A}^* $ stands for the adjoint of ${\mathcal A}$). Thus, for  smooth enough functions $\varphi$ (at least ${\mathcal C}^2(\R^d,\R)$), one has  $\nu({\mathcal A}\varphi)=\int_{\R^d}{\mathcal A}\varphi(x)\nu(dx)=0 $. The authors then investigate the weak convergence  of   
 $\nu_n(f)-\nu(f)$ once suitably renormalized. However, in these results, the assumptions   are made on the function $\varphi$ itself rather than on $f$.
 To overcome this limitation and exploit directly some assumptions on the function $f$ requires to solve 
the Poisson equation ${\mathcal A}\varphi=f-\nu(f) $. 
 This is precisely for this step that some structure conditions are needed,  namely (hypo)ellipticity or confluence conditions. We refer for instance to the work of Pardoux and Veretennikov~\cite{pard:vere:01},  Rothschield and Stein~\cite{roth:stei:76} or Villani~\cite{vill:09} who discuss the solvability of the Poisson problem under some ellipticity or hypoellipticity assumptions. We also mention the work of Pag\`es and Panloup~\cite{page:panl:14} who exploit some confluence conditions allowing to handle for instance the case of an Ornstein-Uhlenbeck process with degenerate covariance matrix. \textcolor{black}{We refer to Sections~\ref{SEC_PRACTICAL_MAIN_RESULTS} and~\ref{PRACT_DEV_BD} for precise assumptions giving the uniqueness of the invariant distribution of~\eqref{eq_diff} and the expected smoothness properties for the associated Poisson problem}. 

In the current paper, our goal is  to establish \textcolor{black}{for this recursive procedure  a non-asymptotic Gaussian control for the deviations of the quantity $\nu_n(f) - \nu(f)$ for  {\em possibly unbounded} Lipschitz continuous functions $f$}. Such non-asymptotic bounds are crucial in many applicative fields. Indeed, for specific practical simulations, it is not always possible to run ergodic means for very large values of $n$. It will be direct to derive, as a by-product of our deviations estimates, some  \textit{computable} non-asymptotic confidence intervals. 
 A specific feature of such non-asymptotic deviation inequalities is that their accuracy depends again on the  status of the diffusion coefficient $\sigma$  with  respect to  the Poisson equation. Thus,  if  $\|\sigma\|^2-\nu(\|\sigma\|^2)={\mathcal A}\vartheta $ is a coboundary \textcolor{black}{(where $\|\cdot\|$ denotes a matrix norm)}, we manage to improve our analysis, to derive better concentration bounds in a certain deviation range 
 as well as some additional deviation regimes. Also, this additional study seems rather efficient to capture the numerical behavior of the empirical deviations. We refer to Section~\ref{SEC_COBORD} and~\ref{SEC_NUM} for details about these points. Eventually, our main deviation results allow to provide deviation inequalities  for plain Lipschitz  continuous sources $f$ in the ergodic approximation, by using  a suitable regularization procedure, \textcolor{black}{as established in Theorem~\ref{THEO_CTR_LIP}}. \textcolor{black}{As expected, dealing with this general class of functions requires more stringent  constraints on the time steps, that must be \textit{small} enough, and prevents from obtaining the fastest convergence rates (see again Theorem~\ref{THEO_CTR_LIP} and Section~\ref{REG_LIP_SOURCE})}.

\textcolor{black}{The main feature of the sequence~\eqref{CV_ERGO} of  weighted empirical measures is that it  targets the \textit{true} invariant distribution $\nu$ of the continuous time diffusion. The price to pay is the use of an Euler scheme with decreasing step which is a non-homogeneous Markov chain. This  induces new difficulties
  compared to the extensive literature on deviation inequalities for ergodic homogeneous Markov chains.
   In particular, our approximation procedure  produces some remainder terms that need  to be controlled accurately enough  in a non-asymptotic way} to produce tractable deviation inequalities asymptotically close to their counterparts for the diffusion itself. This a major difficulty compared to a $CLT$ where these remainder terms are simply requested to go to $0$ \textcolor{black}{fast enough}.
 
 As mentioned above and like for the $CLT$ (see~\cite{bhat:82} for the diffusion or~\cite{lamb:page:02} for the weighted empirical measures $\nu_n$), these deviation inequalities are naturally established for coboundaries $f-\nu(f)={\mathcal A}(\varphi)$, the assumptions being made on $\varphi$. Our second objective in  this paper is to state our results so that all assumptions could be read  on the source function $f$ itself. This \textcolor{black}{first} requires to solve the Poisson equation in that spirit, that means deriving pointwise regularity results on $\varphi$ from those made on $f$. \textcolor{black}{Again, for Lipschitz sources, this step will require an appropriate regularization procedure}.

 In particular, we will not rely on the Sobolev regularity (\textcolor{black}{see e.g. Pardoux and Veretennikov \cite{pard:vere:01}}) but rather on some Schauder estimates in line with the works by Krylov and Priola~\cite{kryl:prio:10}, which allow to benefit from the elliptic regularity for operators with unbounded coefficients. For more details, we refer to the introduction of Section~\ref{SEC_POISS}.

\subsection{Assumptions and Related Asymptotic Results}\label{SEC_ASS}
 From now on, we will extensively use the following notations. 
 
 For a given step sequence $(\gamma_n)_{n\ge 1} $, we denote: 
\begin{eqnarray*}
\forall \ell \in \R,\ \Gamma_n^{(\ell)} :=\sum_{k=1}^n \gamma_k^\ell, \ 
\Gamma_n := \sum_{k=1}^n \gamma_k = \Gamma_n^{(1)}.
\end{eqnarray*}
In practice, we will consider time step sequences: $\gamma_n \asymp \frac{1}{n^{\theta}}$ with $\theta \in (0, 1]$, where for two sequences $(u_n)_{n\in \N}, \ (v_n)_{n\in \N}$ the notation $u_n \asymp v_n$ means that $\exists n_0 \in \N, \ \exists C\geq 1$ such that $\forall n \geq n_0, \ C^{-1} v_n \leq u_n \leq C v_n$.

\textcolor{black}{For a vector $v\in \R^k,\ k\in \{d,r\}$, we denote by $|v|:=(\sum_{j=1}^k v_j^2)^{\frac 12} $ its \textcolor{black}{(canonical)} Euclidean norm}. Also,  for a function $\psi:\R^q \rightarrow \R^d $,  we set $\|\psi\|_\infty:=\sup_{x\in \R^q} |\psi(x)| $.

\subsection*{Hypotheses}

\begin{trivlist}
\item[\A{C1}] The random variable $X_0$ is supposed to be sub-Gaussian, {\it i.e.} \textcolor{black}{its} square \textcolor{black}{is} exponentially integrable up to some threshold. Namely, there exists $\lambda_0 \in \R^*_+$ such that:
\begin{equation*}\label{cond_X0}
\forall \lambda < \lambda_0,\quad \E\,[\exp(\lambda |X_0|^2)] < +\infty.
\end{equation*}

\item[\A{GC}]
The \textcolor{black}{$\mu $-distributed} i.i.d. innovation sequence $(U_n)_{n\ge 1} $ is such that $\E\,[U_1]=0$ and for all $(i,j,k)\in \left\{ 1,\cdots,r\right\}^3 $, $\E\,[U_1^i U_1^j]=\delta_{ij},\ \E\,[U_1^iU_1^jU_1^k]=0 $. Also, $(U_n)_{n\ge 1} $ and $X_0$ are independent. Eventually, $U_1$  satisfies the following Gaussian concentration property, {\it i.e.} for every $1-$Lipschitz continuous function $g:\R^r\rightarrow \R$ and every $\lambda>0 $:
\begin{equation*}
\E\big[\exp(\lambda g(U_1))\big] \leq \exp\left( \lambda \E\,[g(U_1)] + \frac{ \lambda^2}{2}\right).
\end{equation*}
Observe that if $U_1\overset{({\rm law})}{=} {\mathcal N}(0,I_r) $ or $U_1\overset{({\rm law})}{=}(\frac 12(\delta_1+\delta_{-1}))^{\otimes r}  $, {\it i.e.} for Gaussian or symmetric Bernoulli increments which are the most commonly used sequences for the innovations, the above identity holds. On the other hand, what follows can be adapted almost straightforwardly for a wider class of sub-Gaussian distributions satisfying that \textcolor{black}{for some $\varpi>0 $} and  \textcolor{black}{for all} $\lambda>0 $:
\begin{equation}
\label{GEN_GC}
\E\big[\exp(\lambda g(U_1))\big] \leq \exp\left( \lambda \E\,[g(U_1)] + \frac{ \varpi \lambda^2}{4}\right),
\end{equation}
\textcolor{black}{which yields that} for all \textcolor{black}{$r> 0$}, $\P[|U_1|\ge r]\le 2\exp(-\frac{r^2}{\varpi}) $ (sub-Gaussian concentration of the innovation). \textcolor{black}{The case} $\varpi=2 $ corresponds to the standard Gaussian concentration. This is also the constant in the logarithmic Sobolev inequality \textcolor{black}{fulfilled} by the standard Gaussian measure.\\

\item[\A{C2}]\label{LABEL_BD}
There exists a positive constant $\kappa $ such that, 
\begin{align*}\label{B}
\sup_{x\in \R^d} \|\sigma(x)\|^2 \leq \kappa,
\end{align*}
where \textcolor{black}{
$\|\sigma(x)\| $ stands for the operator norm of $\sigma(x) $, i.e. $\|\sigma(x)\|=\sup_{z\in \R^r, |z|\le 1} |\sigma(x) z| $ \textcolor{black}{(keep in mind that $\|\sigma(x)\|= \| \sigma^*(x)\|= \|\sigma\sigma^*(x)\|^{\frac 12}$)}}. We then set $\|\sigma\|_\infty:=\sup_{x\in \R^d} \|\sigma(x)\| $.\\


\item[($\mathbf{\mathcal{L}_V}$)]

\textcolor{black}{There exists a Lyapunov function $V: \mathbb{R}^d \longrightarrow [v^*, +\infty[$, with $v^* >0$, satisfying the following conditions}:

\begin{enumerate}[label=\roman*)]

\item \textcolor{black}{Regularity-Coercivity.} $V$ is a $\mathcal{C}^2$ function,  $\| D^2 V \| _{\infty} <+ \infty$, and $\lim_{|x| \rightarrow \infty} V(x) = + \infty$.

\item \textcolor{black}{Growth control.} There exists $C_{_V} \in (0,+\infty)$ such that for all $x \in \R^d$:
\begin{equation*}\label{L}
|\nabla V(x)|^2 + |b(x)|^2 \leq C_{_V} V(x).
\end{equation*}

\item \textcolor{black}{Stability}. Let $\mathcal{A}$ be the infinitesimal generator associated with the diffusion Equation~\eqref{eq_diff}, defined for all $\varphi \in C_0^2(\R^d,\R) $ and for all $x\in \R^d $ by:
\begin{eqnarray*}
\mathcal{A} \varphi(x)=  \textcolor{black}{\langle b(x), 
\nabla \varphi(x) \rangle} + \frac1{2} \Tr \big( \Sigma(x) D^2 \varphi(x)\big),\ \Sigma(x):=\sigma\sigma^*(x),
\end{eqnarray*}
where, for two vectors $v_1,v_2 \in \R^d $, the symbol $\textcolor{black}{\langle v_1, 
v_2\rangle} $ stands for the canonical inner product of $v_1$ and $v_2 $ and \textcolor{black}{for $M\in \R^d\otimes \R^d, \ \Tr(M) $ denotes the trace of the matrix $M$}.
 
There exist $\alpha_{_V} >0$, $\beta_V \in \R^+$ such that for all $x \in \R^d$,
\begin{equation*}
\mathcal{A} V(x) \leq -\alpha_{_V} V(x) + \beta_V.
\end{equation*}
\end{enumerate}
As a consequence of \A{${\mathcal L}_V $} i), there exist constants $K$ and $ \bar c$ such that for $|x|\ge K,\ |V(x)| \le 
\bar c|x|^2 $, \textcolor{black}{which in turn implies, from \A{${\mathcal L}_V $} ii),} that  $|b(x)|\le \sqrt{C_{_V}\bar c}|x|$. \\

\item[\A{U}] There \textcolor{black}{exists}  a unique invariant \textcolor{black}{distribution}, \textcolor{black}{denoted from now on by  $\nu $, for}  Equation~\eqref{eq_diff}. \\
\item[\A{S}] For \textcolor{black}{a Lyapunov function} $V$ satisfying \A{${\mathcal L}_V $},
we assume that the \textcolor{black}{step} sequence $(\gamma_k)_{k\ge 1} $ satisfies for all $k\ge 1, \gamma_k\le \frac 12 \min(\frac{1}{\sqrt{C_{_V}\bar c}}, \frac{\alpha_{_V}}{C_{_V} \| D^2 V \|_{\infty}})$.\\ 

Condition \A{S}~ means that we assume the time steps are sufficiently \textit{small} w.r.t. the upper bounds of the coefficients and the Lyapunov function.

\end{trivlist}

\begin{remark}
\label{THE_REM_COEFF}

We have assumed \A{U} without imposing some specific non-degeneracy conditions. Observe that \A{${\mathcal L}_{{\mathbf V}} $} yields existence \textcolor{black}{of an invariant distribution} (see \textcolor{black}{e.g. Chapter 4.9 in} ~\cite{ethi:kurz:97}). Additional structure conditions on the coefficients ((hypo-)ellipticity~\cite{khasminskii2011stochastic},~\cite{pard:vere:01},\textcolor{black}{~\cite{pard:vere:03},~\cite{pard:vere:05}}, ~\cite{vill:09} or confluence~\cite{page:panl:14}) then yield uniqueness.


Assumption \A{S} is a technical condition which is exploited in order to derive the non-asymptotic controls of Theorem~\ref{ineq_fin} (see especially the proof of Proposition~\ref{expV_int} below).
\end{remark}


Observe that, as soon as conditions \A{C2}, \A{$\mathbf{{\mathcal L}_V} $}, \A{U} are satisfied and $\E\,[U_1]=0, \E[U_1^{\otimes 3}]=0$, the following  Central Limit Theorem (CLT) holds (see Theorems 9, 10 in~\cite{lamb:page:02}).
\begin{theo}[CLT] \label{theo}
Under \A{C2}, \A{$\mathbf{{\mathcal L}_V} $}, \A{U}, if $\E\,[U_1]=0,\ \textcolor{black}{\E\,[U_1^{\otimes 3}]=0}$ \textcolor{black}{and $ \E[V(X_0)]<+\infty$}, 
we have the following results.
\begin{enumerate}
\item[$(a)$] \textit{Fast decreasing step.} If  $\lim_n \frac{\Gamma_n^{(2)}}{\sqrt{\Gamma_n}}  
= 0$  and $\E\,[|U_1|^6] < +\infty$, then, for any  Lipschitz continuous function $\varphi$ in $\mathcal{C}^3(\R^d,\R)$ with $D^2 \varphi$ and $D^3 \varphi$ bounded, one has \textcolor{black}{(with $(\mathcal{L})$ denoting  weak convergence)}
\begin{eqnarray*}
\sqrt{\Gamma_n} \nu_n (\mathcal{A} \varphi) \overset{(\mathcal{L})}{\longrightarrow} \mathcal{N}\left(0, \int_{\R^d} | \sigma^* \nabla \varphi|^2 d\nu \right).
\end{eqnarray*}

\item[$(b)$] \textit{Critical and slowly decreasing step.} If $\lim_n \frac{\Gamma_n^{(2)}}{\sqrt{\Gamma_n}} 
= \widetilde{\gamma} \in ]0 ,+ \infty ]$ and if $\E\,[|U_1|^8] < +\infty$, then for every Lipschitz continuous  function $\varphi \in \! \mathcal{C}^4(\R^d,\R) $ with $(D^i \varphi)_{i\in \{2,3,4\}}$ bounded:
\begin{eqnarray*}
\sqrt{\Gamma_n} \nu_n (\mathcal{A} \varphi) &\overset{(\mathcal{L})}{\longrightarrow}& \mathcal{N}\left(\widetilde{\gamma} m, \int_{\R^d} | \sigma^* \nabla \varphi|^2 d\nu \right) \hspace{0.3cm} {\rm if} \hspace{0.1cm}   \widetilde{\gamma} <+\infty,\ (\text{critical decreasing step})\\
\frac{\Gamma_n}{\Gamma_n^{(2)}} \nu_n ({\mathcal A} \varphi ) &\overset{\P}{\longrightarrow} & m \hspace{0.3cm} {\rm if} \hspace{0.1cm}   \widetilde{\gamma} = +\infty,\ (\text{slowly decreasing step}),
\end{eqnarray*}

\begin{eqnarray*}
\mbox{where } \hskip 1,5cm m &:=& -  \int_{\R^d} \Big( \frac{1}{2} D^2 \varphi(x) b(x) ^{\otimes 2} + \Phi_4(x) \Big) \nu(dx),\hskip 2,5cm
\\
\mbox{with }\hskip 1cm  \Phi_4(x) &:=& \int_{\R^r} \Big( \frac{1}{2} \langle D^3 \varphi(x) b(x), ( \sigma(x) u ) ^{\otimes 2} \rangle + \frac{1}{24} D^4 \varphi(x) ( \sigma(x) u)^{\otimes 4} \Big) \mu (du)\hskip 2cm
\end{eqnarray*}
and $\mu$ denotes the distribution of the innovations $(U_k)_{k\ge 1} $. In the above definition of $\Phi_4 $, the term $D^3\varphi $ stands for the order $3$ tensor $(\partial_{x_i,x_j,x_k}^3 \varphi)_{(i,j,k)\in \leftB 1,d\rightB^3} $ and we denote, for all $x\in \R^d$, by $D^3\varphi(x)b(x)$ the $\R^d\otimes \R^d  $ matrix with entries $\big(D^3\varphi(x)b(x)\big)_{ij}=\sum_{k=1}^d (D^3 \varphi(x))_{ijk}b_k(x),\ (i,j)\in \leftB 1,d\rightB^2 $.

\end{enumerate}
\end{theo}

\begin{remark}\label{rem_rate_BIS}
Let us specify that for a step sequence $(\gamma_n)_{n\in \N} $ such that $\gamma_n\asymp n^{-\theta} $, $\theta\in (0,1] $, it is easily checked that case \textit{(a)} occurs for $\theta\in (\frac 13 ,1]$ for which $\frac{\Gamma_n^{(2)}}{\sqrt{\Gamma_n}}\underset{n}{\ra} 0 $. In case \textit{(b)}, that is for $\theta\in (0,\frac13] $, $\frac{\Gamma_n^{(2)}}{\sqrt{\Gamma_n}}\underset{n}{\ra} \widetilde \gamma$, with $\widetilde \gamma<+\infty $ for $\theta=\frac 13 $ and $\widetilde \gamma=+\infty $ for $\theta\in (0,\frac 13). $
\end{remark}

Let us mention that, when  $\textcolor{black}{\lim_n} \frac{\Gamma_n^{(3/2)}}{\sqrt{\Gamma_n}}<+\infty $, {\it i.e.} $\gamma_n \asymp n^{-\theta}, \theta\in(1/2,1] $, the \textcolor{black}{CLT} of point \textit{(a)} holds without the condition $\E\,[U_1^{\otimes 3}] =0$ \textcolor{black}{provided} $\E\,[|U_1|^4]<+\infty $ (see Theorem 9 in~\cite{lamb:page:02}). 
Moreover, the boundedness condition \A{C2} can be relaxed to derive the CLT, which holds provided $\lim_{|x|\rightarrow +\infty}\frac{|\sigma^*\nabla \varphi(x)|^2}{V(x)}=0 $ (strictly sublinear diffusion) in case \textit{(a)} and $\sup_{x\in \R^d}\frac{|\sigma^*\nabla \varphi(x)|^2}{V(x)}<+\infty $ (sublinear diffusion) in case \textit{(b)}. We refer again to Theorems 9 and 10 in~\cite{lamb:page:02} for further considerations.

%

\begin{remark}
\label{rem_rate}
The reader should have in mind that an  ergodic result similar to the one stated in the fast decreasing step setting 
holds for the diffusion itself under the same structure assumptions, {\it i.e.} \A{C2}, \A{${\mathcal L}_{{\mathbf V}} $}, \A{U} (see Bhattacharya~\cite{bhat:82}). In fact \A{C2} can be partially relaxed as well, as  mentioned above.
Precisely, 
$$
\frac 1{\sqrt t} \int_0^t \textcolor{black}{{\mathcal A}\varphi}(Y_s) ds \overset{{\mathcal L}}{\longrightarrow}  {\mathcal N}\Big(0,\int_{\R^d} |\sigma^*\nabla \varphi|^2d\nu\Big)\quad \mbox{ as }\quad t\rightarrow+\infty.
$$
Note that the asymptotic variance corresponds to the usual integral of the ``\textit{carr\'e du champ}" w.r.t. to the invariant distribution (see again Bhattacharya \cite{bhat:82} or the monograph by  Bakry \textit{et al.}~\cite{bakr:gent:ledo:14}), i.e.:
$$\int_{\R^d}|\sigma^* \nabla \varphi(x)|^2 \nu(dx)=-2\int_{\R^d} \langle {\mathcal A} \varphi,\varphi\rangle(x) \nu(dx).$$
 In both settings,  the normalization is the same:   $\sqrt{t}$ for the diffusion and $\sqrt{\Gamma_n}$ for the scheme. Except that, 
as emphasized by Theorem~\ref{theo}, for slowly decreasing step~-- when $\theta<1/3$~--~the time discretization effect becomes prominent and ``hides" the CLT so that $\theta=1/3 $ (critical value between ``fast" and ``slow" settings) yields the fastest rate with a biased CLT.
\end{remark}
\begin{remark}
We would like to mention  that, in the biased case~$(b)$, for steps of the form $\gamma_k=\gamma_0 k^{-1/3},\ k\ge 1 $, it is  important for  a practical implementation  to choose $\gamma_0 $ in  an appropriate  way, namely by  minimizing the function $\gamma_0\mapsto c_1\gamma_0+c_2\gamma_0^{-1/2}, c_1=\lim_n \frac{\sum_{k=1}^{n}k^{-2/3}}{(\sum_{k=1}^n k^{-1/3})^{1/2}}, c_2=\int_{\R^d} | \sigma^* \nabla \varphi|^2 d\nu  $, which corresponds to the mean-variance contribution deriving from the biased limit Theorem. Of course, $c_2$ is usually unknown, and the concrete optimization has to be performed replacing $c_2$ by a computable estimate, like for instance upper bounds, i.e. $c_2\le \|\sigma\|_\infty\|\nabla \varphi\|_\infty $.
\end{remark}

The purpose of this work is to obtain non-asymptotic deviation results which match with the above CLT.
In the current ergodic framework, the very first non-asymptotic results were 
established for the Euler scheme with constant time step by Malrieu and Talay in~\cite{MalrieuTalay} when the diffusion coefficient $\sigma$  in~\eqref{eq_diff} is constant. The key tool in their approach consists in establishing a Log Sobolev inequality, which implies Gaussian concentration, for the Euler scheme. This approach allows to easily control the invariant distribution associated with the diffusion process~\eqref{eq_diff}, see  {\it e.g.}  Ledoux~\cite{ledo:99} or Bakry \textit{et al.}~\cite{bakr:gent:ledo:14} in a general framework. However Log Sobolev, and even Poincar\'e, inequalities turn out to be rather \textit{rigid} tools and are not very well adapted  for discretization schemes like~\eqref{scheme} with or without decreasing steps.

Our approach relies on martingale techniques, which were  already a crucial tool to establish the asymptotic results of~\cite{lamb:page:02} and have been successfully used in Frikha and Menozzi~\cite{frik:meno:12} as well to establish non-asymptotic bounds for the regular Monte Carlo error associated with the Euler discretization of a diffusion over a finite time interval $[0,T] $ and a class of stochastic algorithms of Robbins-Monro type.  Roughly speaking, for a given $n$, we decompose the quantity $\sqrt{\Gamma_n}\nu_n({\mathcal A}\varphi) $ as $M_n+R_n $ where $(M_k)_{k\ge 0} $ is a martingale which has Gaussian concentration and $R_n $ is a remainder term to be controlled in a non-asymptotic way.

We can as well refer to the recent work by Dedecker and Gou\"ezel~\cite{dede:goue:15} who also use a martingale approach to derive non-asymptotic deviation bounds for separately bounded functionals of geometrically ergodic Markov chains on a general state space. 

Let us also mention that many non-asymptotic results have been obtained 
based on functional inequalities. Bolley, Guillin and Villani~\cite{boll:guil:vill:07} derived non-asymptotic controls for the deviations of the Wasserstein distance between a reference measure and its empirical counterpart, establishing a non-asymptotic version of the Sanov theorem. Deviation estimates for sums of weakly dependent random variables (with sub exponential mixing rates) have been considered in Merlev\`ede \textit{et al.}~\cite{merl:peli:rio:10}. From a more dynamical viewpoint, let us mention the work of Joulin and Ollivier~\cite{joul:olli:10}, who introduced for rather general homogeneous Markov chains a kind of \textit{curvature} condition to derive a spectral gap for the chain, and therefore an exponential convergence of the marginal laws towards the stationary distribution. We also mention a work of Blower and Bolley~\cite{blow:boll:06}, who obtain Gaussian concentration properties for deviations of functional of the path for metric space valued homogeneous Markov chains or Boissard~\cite{bois:11} who established non-asymptotic deviation bounds for the Wasserstein distance between the marginal distributions  and the stationary law, still in the homogeneous case. 
The common idea of these works is to prove some contraction properties of the transition kernel of the Markov chain in Wasserstein metric. However,
this usually requires to have some continuity in Wasserstein metric for the transition law involved, see  {\it e.g.}  condition  \textit{(ii)} in Theorems 1.2 and 2.1 of~\cite{blow:boll:06}.  Checking such continuity conditions can be difficult in practice. Sufficient conditions, which require absolute continuity and smoothness of the transition laws are given in Proposition 2.2 of~\cite{blow:boll:06}.

Though potentially less sharp for the derivation of constants, the adopted martingale-based approach in this work turns out to be  rather simple, robust and can be  very naturally adapted to both discrete innovations and inhomogeneous time steps dynamics like the one we currently consider. 

It should as well allow to control deviations for functionals of the path, in the spirit of those considered in~\cite{page:panl:12}, \cite{page:panl:14}. Also, the approach could possibly extend to diffusions with less stringent Lyapunov conditions, like the weakly mean reverting drifts considered in~\cite{lamb:page:03}, or even to more general ergodic Markov processes, \textcolor{black}{see e.g. Pag\`es and Rey \cite{page:rey:17}}. 
These aspects will concern further research.

\medskip
As an application of our non-asymptotic concentration results, we will discuss two important topics:

\smallskip
-- The first one is of numerical interest and deals with non-asymptotic confidence intervals associated with the estimation of the ergodic mean. Such results can be very useful in practice when the computational resources are constrained (by time, by the model itself,\dots). If we assume that $\varphi \in{\mathcal C}^3(\R^d,\R)$, Lipschitz continuous with $(D^{i}\varphi)_{i\in \{2,3\}} $ bounded, \textcolor{black}{such that 
 the mapping $x\in \R^d\mapsto \langle b(x),\nabla \varphi(x)\rangle$  and $D^3\varphi $ are Lipschitz continuous}, we then establish that 
there are explicit sequences $c_n\le 1\le C_n $ converging to $1$ such that  for all $n\in \N $, for all $a>0 $ and $\gamma_k\asymp k^{-\theta},\ \theta\in (\frac 13,1] $,
\begin{equation}
\label{FIRST_DEV_STAT}
\P[\sqrt{\Gamma_n}\nu_n({\mathcal A}\varphi)\ge a]\le C_n \exp\left(-c_n\frac{a^2}{2\|\sigma\|_\infty^2\|\nabla \varphi\|_\infty^2}\right).
\end{equation}
When the diffusion coefficient $\sigma $ is such that $\|\sigma\|^2-\nu(\|\sigma\|^2) $ is itself  a coboundary \textcolor{black}{(or its counterpart for any other norm dominating $\|\cdot\|$)}, the previous bound improves in a certain deviation range for $a$. Namely, we are able to replace $\|\sigma\|_\infty^2 $ by $\nu(\|\sigma\|^2) $ in~\eqref{FIRST_DEV_STAT}, going thus closer to the theoretical limit variance involving the ``\textit{carr\'e du champ}". Moreover,  a mixed regime appears in the non-asymptotic deviation bounds  
which dramatically improves, \textcolor{black}{from the numerical viewpoint}, the general case for a certain deviation range. In particular, the corresponding variance is closer to the asymptotic one given by the ``\textit{carr\'e du champ}"  (see Theorem~\ref{THM_COBORD} below). 
 In accordance with the limit results of Theorem~\ref{theo}, the drifts associated with the fastest convergence rates can be handled as well. We obtain in full generality, results of type~\eqref{FIRST_DEV_STAT} under slightly weaker smoothness assumptions, considering  {\it e.g.}  $D^3\varphi $ being $\beta\in (0,1] $-H\"older continuous. 
Eventually, under suitable ellipticity conditions on $\sigma$, we are able to give non-asymptotic deviation bounds for a Lipschitz source $f$ as well as explicit gradient bounds for the solution $\varphi $ of the corresponding Poisson problem.


\smallskip
-- The second one is mainly theoretical and concerns non-asymptotic deviation bounds for the \textit{celebrated} almost-sure CLT first established by Brosamler and Schatte (see~\cite{bros:88} and~\cite{scha:88}) and revisited through the ergodic discretization schemes viewpoint in~\cite{lamb:page:02}.

\smallskip
Both applications require a careful investigation of the corresponding Poisson \textcolor{black}{ equation $\mathcal{A}\varphi=f-\nu(f)$. We will in particular prove that some pointwise regularity properties  can be transferred from $f$ to   $\varphi$.}

\smallskip
\textcolor{black}{The paper is organized as follows. We conclude this section by introducing some notations. Our  main  results are presented 
 in Section~\ref{MR}. 
We first state therein the specific concentration results for functions $f$ writing $f={\mathcal A\varphi}+\nu(f) $ (see Section~\ref{MCR}). We then proceed with some suitable controls on the Poisson problem associated with ${\mathcal A} $ and $f$ \textcolor{black}{in a confluent framework under the two main cases considered: namely a possibly degenerate setting, which requires a strong confluence condition and \textit{smooth} source and coefficients, and  a non degenerate setting, which allows to weaken the confluence condition as well as the smoothness assumptions on the source and the coefficients since in that case we manage to benefit from an elliptic bootstrap property} 
(see Section~\ref{SEC_PRACTICAL_MAIN_RESULTS}). We eventually give in Section~\ref{PRACT_DEV_BD} some practical and tractable deviation bounds and non-asymptotic confidence intervals, including a Slutsky like result, for a given specific source $f$ under the afore mentioned conditions on the coefficients of~\eqref{eq_diff}.
}

\textcolor{black}{We prove our main concentration result in Section~\ref{SEC_PROOF_CONC}. 
  Section~\ref{SEC_COBORD} is devoted to the case where  $\|\sigma\|^2-\nu(\|\sigma\|^2) $ is  a coboundary.
We then prove in Section~\ref{SEC_POISS} the required controls on the Poisson equation for our deviation result to hold as well as the practical controls of Section~\ref{PRACT_DEV_BD}.   
Section~\ref{TCLPS_SEC} is dedicated to   the non-asymptotic deviation bounds for the almost-sure CLT and Section~\ref{SEC_NUM} to  the numerical illustration of our non-asymptotic confidence intervals.
}

\subsection{Notations}\label{SEC_NOT}
In the following, we will denote by $C$ a constant that may change from line to line and depend, uniformly in time, on known parameters appearing in %
\A{C1}, \A{GC}, \A{C2}, \A{${\mathbf {{\mathcal L}_V} }$}, \A{S}. Other possible dependencies will be explicitly specified. We will also denote by ${\mathscr R}_n $ and $e_n$ deterministic remainder terms that respectively converge to 1 and 0 with $n$. The explicit dependencies of those sequences  again appear in the proofs.

For a function $f\in C^\beta(\R^d,\R),\ \beta\in (0,1]$, we denote 
$$[ f]_\beta:=\sup_{x\neq x'} 
 \frac{| f(x)- f(x')|}{|x-x'|^\beta}< +\infty$$
 its H\"older modulus of continuity. Observe carefully that, when $f $ is additionally bounded, we have that for all $0< \beta'<\beta $:
 \begin{equation}
 \label{HOLDER_PLUS_MOINS}
 [f]_{\beta'}\le [f]_{\beta}^{\frac{\beta}{\beta'}}(2\|f\|_\infty)^{1-\frac{\beta}{\beta'}}.
 \end{equation} 
Additionally, for $f\in {\mathcal C}^p(\R^d,\R),\ p\in \N$, we set for $\beta\in (0,1] $: 
$$[ f^{(p)}]_\beta:=\sup_{x\neq x', 
 |\alpha|=p}\frac{|D^\alpha f(x)-D^\alpha f(x')|}{|x-x'|^\beta}\le +\infty,$$
where $\alpha$ (viewed as an element of $\N_0^d\backslash \{0\}$ with $\N_0:=\N\cup \{0\}$) is a multi-index of length $p$, {\it i.e.} $|\alpha|:=\sum_{i=1}^d \alpha_i=p $. 

For notational convenience,  we also introduce for $k\in \N_0,\beta\in (0,1] $ and $m\in \{1,d,d\times r\} $ the H\"older space
\begin{eqnarray*}
{\mathcal C}^{k,\beta}(\R^d,\R^m)
:=\Big\{  f \in {\mathcal C}^{k}(\R^d,\R^m):  \forall \alpha, |\alpha|\in \leftB 1,k\rightB, \sup_{x\in \R^d} |D^\alpha f(x)|<+\infty  , [f^{(k)}]_\beta<+\infty\Big\} .
\end{eqnarray*}
We also denote by ${\mathcal C}_b^{k,\beta} $ the subset of ${\mathcal C}^{k,\beta}$ for which the \textcolor{black}{functions themselves ares} bounded. In particular, ${\mathcal C}^{0,1}(\R^d,\R^m) $ is the space of Lipshitz continuous functions from $\R^d $ to $\R^m $ \textcolor{black}{and $C_b^{0,\beta}(\R^d,\R^m) $ denotes the space of bounded $\beta $-H\"older continuous functions}. Observe as well that, if $f\in C^{k,\beta},\ k\ge 1 $ then $f$ is Lipschitz continuous.
%

We will as well use the notation $\leftB n,p\rightB $,  $(n,p)\in (\N_0)^2, n\le p $, for the set of integers being between $n$ and $p$.
%
Also, for a given Borel function $f:\R^d\rightarrow E$, where $ E$ can be $\R, \R^d,\R^d\otimes \R^q,\ q\in\{r,d\} $,  we set for $k\in \N_0 $:
\begin{eqnarray*}
f_k &:=& f(X_k).
\end{eqnarray*}
Eventually, for  $k\in \N_0$, we denote by $\F_k:=\sigma\big( (X_j)_{j\in \leftB 0,k\rightB}\big) $.

\mysection{Main results}
\label{MR}
\subsection{Result of non-asymptotic Gaussian concentration} \label{MCR}
Our main concentration result is given by the following theorem. In this theorem, we consider a slightly more general situation than for  the CLT recalled  in Theorem~\ref{theo}. We only assume $\varphi \in {\mathcal C}^{3,\beta}(\R^d,\R)  $, $\beta\in (0,1]$ instead of $\varphi\in {\mathcal C}^4(\R^d,\R) $ with existing bounded partial derivatives up to order four (which in particular implies that in Theorem~\ref{theo} $\varphi \in  {\mathcal C}^{3,1}(\R^d,\R)$). 

\begin{theo}\label{ineq_fin}
Assume 
\A{C1}, \A{GC}, \A{C2}, \A{${\mathbf {{\mathcal L}_V} }$},  \A{U}, \A{S} hold. Consider a Lipschitz continuous (possibly unbounded) function $\varphi \in \mathcal{C}^{3,\beta}(\R^d,\R)$ 
for some $\beta\in (0,1] $. Let us furthermore \textcolor{black}{suppose} that:
\begin{equation}  \exists C_{V,\varphi}>0, \ 
\forall x\in \R^d, |\varphi(x)|\le C_{V,\varphi}(1+\sqrt{V(x)}). 
\label{GV}\tag{{\bf G}${}_{\mathbf V}$}
\end{equation}
\textcolor{black}{Let $\theta\in [1/(2+\beta),1]$}  and assume the step sequence $(\gamma_k)_{k\ge 1} $ is of the form $\gamma_k\asymp k^{-\theta} $. 

\smallskip
%
\noindent $(a)$  \textbf{Unbiased Case (sub-optimal convergence rate):} \textcolor{black}{Let $ \theta \in (\frac{1}{2+\beta}, 1 ]$.}
\begin{itemize}
\item[(i)] \textcolor{black}{Assume that the mapping $x\mapsto \langle \nabla \varphi(x),b(x)\rangle $ 
is 
Lipschitz continuous}.\\ 
Then, there exist two explicit monotonic sequences  
$c_n\le 1\le C_n,\ n\ge 1$,   with $\lim_n C_n = \lim_n c_n =1$,  
such that for all $n \geq 1$ and for every $a >0$:  
 \begin{eqnarray*}
\P\big[ |\sqrt{\Gamma_n}\nu_n( \mathcal{A} \varphi )| \geq a \big] \leq 2 C_n \exp\left( - c_n \frac{a^2}{2\|\sigma\|_\infty^2\|\nabla \varphi\|_\infty^2} 
\right).
\end{eqnarray*}
\item[(ii)] 
\textcolor{black}{Suppose that the mapping $x\mapsto \langle \nabla \varphi(x),b(x) \rangle $ \textbf{is not} Lipschitz continuous}.  
The above result still holds for $ 0<a\le \chi_n\frac{\sqrt{\Gamma_n}}{\Gamma_n^{(2)}}$ for a positive sequence $\chi_n\underset{n}{\rightarrow} 0 $ arbitrarily slowly, so that $\chi_n\frac{\sqrt{\Gamma_n}}{\Gamma_n^{(2)}}\underset{n}{\rightarrow}+\infty $. In particular, for a fixed $a>0$, the above concentration inequality holds for $n$ large enough.

\end{itemize}

\smallskip
\noindent $(b)$ \textbf{Biased Case (Optimal Convergence Rate):} \textcolor{black}{Let $ \theta = \frac{1}{2+\beta}$}. 
We set for all $(k,t,u,x)\in \leftB 1, n\rightB\times [0,1]^2\times \R^d$:
\begin{eqnarray}
\label{DEF_LAMBDA_K}
\textcolor{black}{ \Lambda_{k-1}^\beta}(t,u,x):= \E\Big[\Tr\Big(  \big( D^3 \varphi( x + \gamma_k b(x) + u t\sqrt{\gamma_k} \sigma(x) U_k )  \sigma(x) U_k\big)\big(\sigma(x) U_k\otimes U_k \sigma(x)^*\big) \Big) \Big],
 \end{eqnarray}
\textcolor{black}{keeping in mind that, since $\varphi \in {\mathcal C}^{3,\beta}(\R^d,\R) $, $[D^3 \varphi]_\beta<+\infty$}.
We define subsequently: 
\begin{eqnarray}
E_n^\beta:= \frac{1}{\sqrt{\Gamma_n}}\sum_{k=1}^n \gamma_k^{3/2}    \int_0 ^1 \!\!dt \,( 1-t)t \int_0 ^1\!\! du \,\textcolor{black}{\Lambda_{k-1}^\beta}(t,u,X_{k-1})
 \label{DEF_QT}.
\end{eqnarray}
Set now 
\begin{equation}
\label{DEF_QT_THEO_BIAS}
\begin{split}
{\mathcal B}_{n,\beta}:=&E_n^\beta,\ {\rm if}\ \beta\in (0,1),\\
{\mathcal B}_{n,\beta}:=& E_n^\beta+\frac{1}{\sqrt{\Gamma_n}}\sum_{k=1}^n  \gamma_k^2 \int_0^1 (1-t)\Tr\Big(D^2 \varphi (X_{k-1} + t \gamma_k b_{k-1})  b_{k-1}\otimes b_{k-1}   \Big) dt\\
&\qquad +\frac 1{2\sqrt{\Gamma_n}}\sum_{k=1}^n \gamma_k   \Tr\Big( \big(D^2 \varphi ( X_{k-1} + \gamma_k b_{k-1} ) - D^2\varphi( X_{k-1})\big) \Sigma_{k-1} \Big),\ {\rm if}\ \beta=1.
\end{split}
\end{equation}
There exist two explicit monotonic sequences $c_n \le 1\le  C_n,\ n\ge 1$,  with $\lim_n  C_n=\lim_n  c_n= 1$ 
such that for all $n \geq 1$ and for every $a >0$:  
 \begin{eqnarray*}
\P\big[ |\sqrt{\Gamma_n}\nu_n( \mathcal{A} \varphi )
+{\mathcal B}_{n,\beta} | \geq a\big] \leq 2 C_n \exp\left( - c_n   \frac{a^2} {2 \|\sigma\|_\infty^2 \| \nabla  \varphi \|_\infty^2 }\right).
\end{eqnarray*}
For $\beta\in (0,1) $, the random variables $ |{\mathcal B}_{n,\beta}|=|E_{n}^\beta|\le \frac{ [\varphi^{(3)}]_{\beta}  \|\sigma\|_\infty^{(3+\beta)} \E\big[|U_1|^{3+\beta} \big] }{(1+\beta)(2+\beta)(3+\beta) }   \frac{\Gamma_n^{(\frac{3 +\beta}{2})}}{\sqrt{\Gamma_n}}\underset{n}{\longrightarrow} a_{\beta,\infty}>0 $ a.s. Also, for $\beta=1 $, the  $({\mathcal B}_{n,1})_{n\ge 1} $ are exponentially integrable and if,
  furthermore, $D^3\varphi  $ is ${\mathcal C}^1 $, ${\mathcal B}_{n,1}\underset{n}{\ra} -\widetilde \gamma m$ $a.s.$ where $\widetilde \gamma m$ is as in Theorem~\ref{theo}. 
In any case, a bias appears in our deviation controls when we consider, for a given smoothness of order $\beta\in (0,1] $ for $D^3\varphi $, the fastest associated time steps $\gamma_k\asymp k^{-\theta},\ \theta=\frac{1}{2+\beta} $.
\end{theo}

\begin{remark}Observe that, when $\beta=1$, the above result provides \textcolor{black}{a}  non-asymptotic counterpart of the limit Theorem~\ref{theo}. In particular, the concentration constants appearing in Theorem~\ref{ineq_fin} asymptotically match those of the centered CLT recalled in Theorem~\ref{theo}, up to a substitution of the asymptotic variance $\int_{\R^d}|\sigma^* \nabla \varphi(x)|^2\nu (dx) $ by its natural upper bound $\|\sigma\|_\infty^2\|\nabla \varphi\|_\infty^2$. 

Importantly, these bounds do not require \textcolor{black}{``a priori"} non-degeneracy conditions and only depend on the diffusion coefficient through the sup-norm of the diffusion matrix $\Sigma $, assumption \A{C2}. It will anyhow be very natural to consider a non-degeneracy condition (\cite{pard:vere:01},~\cite{roth:stei:76},~\cite{vill:09}), or a confluence condition (\cite{page:panl:14}), when investigating the deviations for a given function $f$, in order to ensure the solvability of the corresponding Poisson equation ${\mathcal A}\varphi=f-\nu(f)$ and to derive explicit upper bounds for 
 $\|\nabla \varphi\|_\infty $ in terms of the coefficients $b,\sigma $ and the source $f$ which turn out to be crucial to design computable non-asymptotic confidence intervals. These aspects are discussed in Section~\ref{SEC_PRACTICAL_MAIN_RESULTS} below.

\textcolor{black}{The alternative form of the asymptotic variance (see~Remark~\ref{rem_rate})
$\int_{{\mathbb R}^d} |\sigma^*\nabla \varphi(x)|^2 \nu(dx)=-2\int_{{\mathbb R}^d}f(x) \varphi(x) \nu(dx)$
suggests that for bounded source terms $f$, an associated natural variance bound would be $2\|f\|_\infty \|\varphi\|_\infty$. \textcolor{black}{Such a control would a priori require less regularity on $\varphi$ than assumed in Theorem~\ref{ineq_fin}. One could for instance try to exploit suitable regularization procedures, like for instance the one proposed in Section~\ref{REG_LIP_SOURCE} for the proof of Theorem~\ref{THEO_CTR_LIP} below, to establish non-asymptotic deviation results under weaker assumptions}. Our main objective being to capture unbounded Lipschitz functions $f$, these aspects will concern further research.}
\end{remark}

 \begin{remark}[Smoothness and Convergence Rate]
Observe that, in coherence with \textcolor{black}{the} asymptotic setting of the CLT recalled \textcolor{black}{in} Theorem~\ref{theo}, for a given $\varphi \in {\mathcal C}^{3,\beta}(\R^d,\R),\ \beta\in(0,1] $, the fastest convergence rate for the deviations is attained for $\theta=\frac{1}{2+\beta} $. A bias appears, which can  be difficult to estimate in practice  since  $\varphi$ is usually unknown.
\end{remark}

\begin{remark}[On the smoothness property of $x\mapsto \langle b(x),\nabla \varphi(x) \rangle  $]
\label{REM_COND_B_NABLA_PHI}
The Lipschitz continuity assumption on the above mapping appearing in case (i) might seem awkward at first sight. \textcolor{black}{It is non-intrinsic in the sense that it involves both the drift $b$ of the model and the test function $\varphi $}. However, this condition naturally appears when $\varphi $ is a smooth solution to the Poisson equation ${\mathcal A}\varphi= f-\nu(f) $. Indeed, recalling the definition of ${\mathcal A} $ in \A{$\mathbf {\mathcal{L}_V}$}, iii), we can rewrite:
$$\langle \nabla \varphi(x), b(x)\rangle =f(x)-\nu(f)-\frac 12 \Tr\Big(\Sigma(x) D^2 \varphi(x) \Big).$$
Hence,  the Lipschitz continuity of the function in the above left hand side readily follows as soon as the source $f$ is Lipschitz and if  $D^2 \varphi$ is bounded and Lipschitz continuous \textcolor{black}{(since $\sigma$ is also bounded and Lipschitz)}.
Note that with the previous notations for function spaces the previous conditions are implied if $ f\in {\mathcal C}^{1,\beta}(\R^{d},\R) \subset {\mathcal C}^{0,1}(\R^{d},\R), \varphi \in {\mathcal C}^{3,\beta}(\R^{d},\R) \Rightarrow D^2 \varphi \in {\mathcal C}_b^{1,\beta}(\R^{d},\R^d \otimes \R^d ) \subset {\mathcal C}_b^{0,1}(\R^{d},\R^d \otimes \R^d )  $.
  We refer to Section~\ref{SEC_COMPLETE_SMOOTHNESS} for details. 
\end{remark}

\textcolor{black}{We now state an improvement of the previous concentration bound when $\|\sigma\|^2 -\nu(\|\sigma\|^2)$ is itself a coboundary, i.e. when the Poisson problem ${\mathcal A}\vartheta=\|\sigma\|^2 -\nu(\|\sigma\|^2) $ can be solved with $\vartheta $ satisfying the assumptions required for $\varphi $ in Theorem~\ref{ineq_fin}. Precisely, we have the following result.}

\begin{theo}\label{THM_COBORD_SHORT}
$(a)$ Under the assumptions of Theorem~\ref{ineq_fin} and with the notations introduced therein, provided that $\vartheta $ solution to the Poisson equation ${\mathcal A}\vartheta=\|\sigma\|^2 -\nu(\|\sigma\|^2)$  satisfies the \textcolor{black}{same smoothness and growth conditions as $\varphi $}, 
for 
$\beta\in (0,1] $ and $ \theta \in (\frac{1}{2+\beta}, 1 ]$ (unbiased case), there exist two explicit monotonic sequences 
$\tilde c_n\le 1\le \tilde C_n,\ n\ge 1$,   with $\lim_n \tilde C_n = \lim_n \tilde c_n =1$
such that for all $n\ge 1$ for all $0<a\le \chi_n \frac{\sqrt{\Gamma_n}}{\Gamma_n^{(2)}}$ for a positive sequence $\chi_n\underset{n}{\rightarrow} 0 $ arbitrarily slowly, so that $\chi_n\frac{\sqrt{\Gamma_n}}{\Gamma_n^{(2)}}\underset{n}{\rightarrow}+\infty $:
 \begin{equation}\label{eq:devineqimprov}
\P\big[ |\sqrt{\Gamma_n}\nu_n( \mathcal{A} \varphi )| \geq a \big] \leq  2\, \widetilde C_n \exp\left(\! - \widetilde c_n \frac{ a^2}{2\nu(\|\sigma\|^2)\|\nabla \varphi\|_\infty^2}\,
\right).
\end{equation}

\noindent $(b)$ If $\vartheta$ solve the Poisson equation ${\mathcal A}\vartheta=\vvvert\sigma\vvvert^2 -\nu(\vvvert\sigma\vvvert^2)$ mutatis mutandis for a matrix norm dominating the operator norm $(\vvvert \sigma(x) \vvvert\ge \| \sigma(x)\|$), then the above bound~\eqref{eq:devineqimprov} still holds with  $\nu(\vvvert\sigma\vvvert^2)$ instead of $\nu(\|\sigma\|^2)$. \end{theo}
 \textcolor{black}{
  Importantly, the above result allows to improve the \textit{natural variance bound} $\|\nabla \varphi\|_\infty^2\| \sigma \|_\infty^2 $ of Theorem~\ref{ineq_fin} by a more refined, namely $\|\nabla \varphi\|_\infty^2 \nu(\| \sigma \|)^2 $. Such a bound can be particularly interesting when the supremum norm of $\sigma $ is \textit{high} but its average w.r.t. the invariant distribution $\nu $ significantly lower. We refer to Section~\ref{SEC_COBORD}, Theorem~\ref{THM_COBORD} (general form of Theorem~\ref{THM_COBORD_SHORT}) and~\ref{SEC_NUM} (numerical results) for further discussions on that topic.  
}
 
 \textcolor{black}{ Of course Claim $(b)$ is less sharp than $(a)$ stated with  the  operator norm $\|\cdot\|$ but solving the  Poisson equation  for $\| \sigma(x)\|$ seems highly non trivial.  
By contrast, if $\vvvert\sigma(x)\vvvert = \|\sigma(x)\|_F:=\big[{\rm Tr}\big(\sigma\sigma^*(x))\big]^{1/2} $ stands for the Fr\"obenius norm, Theorem~\ref{Poisson_regular}  below yields the expected smoothness properties on $\|\sigma\|_F^2-\nu(\|\sigma\|_F^2) $ that  ensure the existence of  a solution to ${\mathcal A}\vartheta=\|\sigma\|_F^2-\nu(\|\sigma\|_F^2) $ meeting the required smoothness conditions. 
The price to pay with  such computable norms being that they usually induce some dependence on the dimension $d$ on the estimates (observe e.g. for the identity matrix $I_d$ of $\R^d \otimes \R^d $, $\|I_d\|_F=d^{1/2} $).
}

\subsection{Uniqueness of the invariant \textcolor{black}{distribution} and Regularity issues for the Poisson problem}
\label{SEC_PRACTICAL_MAIN_RESULTS}

For our deviation analysis to work, we need \textcolor{black}{to have the uniqueness of the invariant distribution $\nu$ and to establish some pointwise controls on the solution of the associated Poisson equation}. Namely, we need to have quantitative bounds on its derivatives and the associated H\"older continuity modulus up to order 3. 

To do so, additionally to our main assumptions introduced for Theorem~\ref{ineq_fin}, we will work in the confluent setting. 
\textcolor{black}{In dimension one, any ergodic diffusion is in \textit{some sense} confluent (see \cite{khas:80}, Appendix of the English translation, Theorem 2.2 p. 308 and its alternative proof in \cite{lema:page:panl:15} Theorem 2)}. Here, we will suppose that the following condition holds:
\begin{trivlist}
\item[$\bullet $] \textbf{Confluence Conditions} 
\item[\A{D${}_\alpha^p $}] We assume that there exists $\alpha >0 $ and $ p\in \textcolor{black}{(1,2]}
$ such that  for all $x\in \R^d $, $\xi\in \R^d $
$$\left \langle \frac {Db(x)+Db(x)^*} 2\xi,\xi \right\rangle +\frac 12 \sum_{j=1}^r \Big( (p-2)\frac{|\langle  D\sigma_{\cdot j}(x) \xi, \xi\rangle|^2}{|\xi|^2}+|D\sigma_{\cdot j} \xi|^2
\Big)\le -  \alpha |\xi|^2, $$
where $Db$ stands here  for the Jacobian of $b$,
$\sigma_{\cdot j} $ stands for the $j^{{\rm th}} $ column of the diffusion matrix $\sigma $ and $D\sigma_{\cdot j} $ for its  Jacobian matrix.
\end{trivlist}

{Within the confluent framework, we will consider from now on two kinds of assumptions \textcolor{black}{which first give the uniqueness of $\nu$} and that can lead to the required smoothness and to computable gradient bounds, which are crucial since they are precisely the quantities appearing in the non-asymptotic Gaussian deviation controls \textcolor{black}{as emphasized in the statement of Theorem~\ref{ineq_fin}}. 

\begin{trivlist}
\item[-] Strong Confluence condition and regularity of the coefficients, which means that the drift is sufficiently dominant in the dynamics \textcolor{black}{and} the coefficients are \textit{smooth} 
(see assumption \A{C${}_{{\rm\mathbf R}}$} 
below). \textcolor{black}{Note that these conditions may hold for degenerate diffusion coefficients}.
 
\item[-] Non-degeneracy of the diffusion coefficient and mild confluence condition and smoothness on the coefficients (see assumption  \A{C${}_{\rm \mathbf {UE}}$} below).
\end{trivlist}

Under a sufficiently strong confluence condition, i.e. when $\alpha$ is large enough in \A{D${}_\alpha^p$}, and provided that the coefficients $b,\sigma, f$ are sufficiently smooth, it is quite direct to derive, through stochastic flow techniques \textit{\`a la Kunita}, the required pointwise  bounds for the derivatives of the Feynman-Kac representation of the solution to the Poisson equation (see \cite{page:panl:14} and \textcolor{black}{Section~\ref{SEC_POISS_BIS}}). 

In the non-degenerate case, the main advantage is that we can alleviate some restrictions on $\alpha $ and the smoothness assumptions on $b,\sigma, f$ to benefit from an elliptic  regularity bootstrap deriving from suitable Schauder estimates available in the current setting from the work by Krylov and Priola \cite{kryl:prio:10}.
}

We now introduce a smoothness assumption on $b,\sigma, f$ that will be useful in both the considered cases.
\begin{trivlist}
\item[$\bullet $] \textbf{Smoothness of the coefficients and the source}.
For $k\in \{1,3\} $ and $\beta \in (0,1) $ define  
\item[\A{R${}_{k,\beta}$}]
The coefficients in equation~\eqref{eq_diff} are s.t. $b \in {\mathcal C}^{k,\beta}(\R^d,\R^d),\sigma \in {\mathcal C}_b^{k,\beta}(\R^d,\R^d)$. Also, the source $f$ for which we want to estimate $\nu(f)$ belong to ${\mathcal C}^{k,\beta}(\R^d,\R) $.
\end{trivlist}

With these assumptions at hand, we now introduce the first setting we consider. 
\begin{trivlist}
\item[$\diamond $] \textcolor{black}{\textbf{The confluent and regular assumption}
\A{C${}_{{\rm\mathbf R}}$},
holds if \A{D${}_\alpha^p $}, \A{R${}_{3,\beta} $}, for some $ \beta\in (0,1]$, are in force and \textcolor{black}{$\|D\sigma\|_\infty^2\le \frac{2\alpha}{2(3+\beta) -p} $ where $\|D\sigma\|_\infty:=\sup_{x\in \R^d} \Big(\sum_{j=1}^d \|D\sigma_{\cdot j}(x)\|^2 \Big)^{\frac 12}$ recalling that, for every  $j\in \leftB 1,d \rightB $, $\|D\sigma_{\cdot j}(x)\| $ stands for the \textcolor{black}{operator norm} of $D\sigma_{\cdot j}(x) $}.} 
\end{trivlist} 
In particular, we do not impose in this case any  additional structure condition on $\sigma$ which can degenerate.\\

In our second main framework, we will assume some uniform ellipticity conditions.
\begin{trivlist}
\item[$\bullet $] \textbf{Non-degeneracy Conditions}. 
\item[\A{UE}] \label{LABEL_UE}Uniform ellipticity. We assume that w.l.o.g. that $r=d $ ($r\ge d $ could also be considered)  
in~\eqref{eq_diff} and that  the diffusion coefficient $\sigma$  is such that
\begin{equation*}
\exists \,\underline{\sigma}>0,\ \forall \xi \in \R^d,\ \langle \sigma\sigma ^*(x)\xi,\xi\rangle \ge \underline{\sigma} |\xi|^2.
\end{equation*}
\end{trivlist}
We now introduce our second main setting:
\begin{trivlist}
\item[$\diamond $] \textcolor{black}{\textbf{The confluent and non-degenerate assumption}
\A{C${}_{{\rm\mathbf {UE}}}$},  holds if \A{D${}_\alpha^p$}, \A{R${}_{1,\beta} $}, for some $\beta \in (0,1] $,  are in force. If $d>1$, we also assume that $\|D\sigma\|_\infty^2\le \frac{2\alpha}{2(1+\beta) -p} $ and that the diffusion matrix $\Sigma $ is such that, for all  $(i,j)\in \leftB 1,d\rightB^2 $,    $\Sigma_{i,j}(x)=\Sigma_{i,j}(x_{i\wedge j},\cdots,x_d) $.}

\end{trivlist}

\begin{theo}\label{Poisson_regular}
Assume that 
\A{$\mathbf {{\mathcal L}_V} $} 
 and either \A{C${}_{{\rm\mathbf R}}$}  
or \A{C${}_{{\rm\mathbf {UE}}}$}
are in force. Then there exists a unique invariant distribution for the solution of ~\eqref{eq_diff}, i.e. assumption \A{U} holds.

\smallskip
The associated Poisson equation
\begin{equation}
\label{POISSON}
\forall x \in \R^d,\ {\mathcal A} \varphi(x) = f(x)-\nu(f),
\end{equation}
admits a unique solution $\varphi \in {\mathcal C}^{3,\beta}(\R^d,\R)$, $\beta\in (0,1) $ \textcolor{black}{centered w.r.t. $\nu $}. Furthermore, the following gradient bound holds
$$\|\nabla \varphi\|_\infty \le \frac{[f]_1}{\alpha},$$
\textcolor{black}{and the mapping $x\mapsto \langle \nabla \varphi (x),b(x)\rangle $ is Lipschitz continuous.}
\end{theo}

\subsubsection{About the regularity of the coefficients}

Under \A{C${}_{{\rm\mathbf R}}$}, the derivatives can be expressed using iterated tangent processes and we cannot hope, without \textit{a priori} any non-degeneracy condition, for a smoothing effect to hold. To have $\varphi \in {\mathcal C}^{3,\beta}(\R^d,\R) $, we need to consider a source $f\in {\mathcal C}^{3,\beta}(\R^d,\R) $ and the same smoothness on  $b,\sigma $ (Assumption \A{R${}_{3,\beta} $}). We refer to Section~\ref{SEC_POISS_BIS} for the proof of Theorem~\ref{Poisson_regular} under \A{C${}_{{\rm\mathbf R}}$}.

In the non-degenerate case, the solvability of the Poisson problem is usually studied in a Sobolev setting, see  {\it e.g.}  \cite{pard:vere:01}. \textcolor{black}{Let us also indicate that pointwise gradient bounds have been obtained by the same authors in ~\cite{pard:vere:03} for bounded drifts and diffusion coefficients which are additionally supposed to be \textit{smooth}, i.e. at least ${\mathcal C}_b^{2,\gamma} $ with the notations introduced in paragraph~\ref{SEC_NOT}. We point out that these estimates do not apply in our current setting in which the drift has typically linear growth.}

\textcolor{black}{We eventually mention the last paper by these authors, namely~\cite{pard:vere:05}.  They derive therein the uniqueness of the martingale solution to the Poisson equation in a potentially degenerate setting under suitable local Doeblin conditions. In that framework, pointwise controls are obtained as well for the solution itself but not for its \textit{derivatives}.
}

To obtain the required smoothness, we use here in the non-degenerate framework of 
\A{C${}_{{\rm\mathbf{UE}}}$}
some Schauder estimates, deriving from the work of Krylov and Priola \cite{kryl:prio:10}, which allow to benefit from the elliptic regularity. Namely, to obtain the mentioned smoothness on  $\varphi $ solving ${\mathcal A} \varphi=f-\nu(f) $, that we expect to be in ${\mathcal C}^{3,\beta}(\R^d,\R),\ \beta \in (0,1)$, we can take a source $f\in {\mathcal C}^{1,\beta}(\R^d,\R)$ and $b\in C^{1,\beta}(\R^d,\R^d), \sigma \in C_b^{1,\beta}(\R^d,\R^d) $. 

\textcolor{black}{We would eventually like to emphasize that the structure condition on $\Sigma$ might seem weird at first sight. It is actually needed to decouple the PDEs formally satisfied by $(\partial_{x_i}\varphi)_{i\in \leftB 1, d\rightB} $ in order to exploit the \textit{a priori} estimates of \cite{kryl:prio:10} established for scalar valued PDEs. We refer to Section~\ref{SEC_POISS_BIS} for a proof and details.}

\subsubsection{About the confluence condition and the restrictions on $\sigma $}

We work here in the confluent setting of \A{D${}_\alpha^p $}. This assumption will allow, through a pathwise analysis associated with the tangent flow, to derive a pointwise gradient bound. Another possibility to obtain such a bound is to assume  
a so-called Bakry and \'Emery curvature criterion, see \cite{bakr:emer:85, bakr:gent:ledo:14}. Under this condition, the gradient and semi-group commute up to an exponential multiplicative factor (see equation~\eqref{COMM_BE} below). 


\medskip
\noindent $\rhd$ \textit{Bakry and \'Emery curvature criterion.}  First, we recall that the ``\textit{carr\'e du champ}" operator $\Gamma$ of  a Markov process with generator ${\mathcal A} $ reads, for every $f,g $ in its domain ${\mathcal D}({\mathcal A})$
$$
\Gamma(f,g):=\frac 12 \Big( {\mathcal A}(fg)-f{\mathcal A}g-g{\mathcal A}f\Big)\quad\mbox{ and }\quad \Gamma(f):=\Gamma(f,f).
$$
We also need to define the $\Gamma_2 $ operator  
$$
\Gamma_2(f)=\frac 12 \Big({\mathcal A}\Gamma(f)-2\Gamma(f,{\mathcal A}f) \Big).
$$

 In our Brownian diffusion setting, we have 
 $$
 \forall x\in\R^d,\  \Gamma(f)(x)=|\sigma^*\nabla f(x)|^2.
 $$
whereas the computation of $\Gamma_2 $ is significantly more involved. However, if the diffusion matrix  \textcolor{black}{$\Sigma= \sigma\sigma^*$} is constant then:
$$
\Gamma_2(f)(x):=\Tr\big((D^2 f (x)\Sigma)^2\big)\textcolor{black}{-}\langle  \nabla f ,Db \Sigma \nabla f\rangle (x).
$$

With these notations at hand, we say that {\em the semi-group $(P_t)_{t\ge 0} $ of ${\mathcal A} $ satisfies the Bakry and \'Emery curvature criterion with parameter $\rho>0 $} if
 
\begin{equation*}
\mbox{\A{BE${}_\rho $} } \hskip 4cm   \forall\, f\!\in {\mathcal D}({\mathcal A}), \quad \Gamma_2(f)\ge \rho\, \Gamma(f). \hskip 5cm 
\end{equation*}

Observe that for $\Sigma=I_d$ the condition \A{BE${}_\rho$} is actually equivalent to \A{D${}_\alpha^p $} with  $\alpha=\rho $ (and any $p\in (1,2] $ since $D\sigma=0 $) and reads 
\[
\left \langle \frac {Db(x)+Db(x)^*} 2\xi,\xi \right\rangle  \le - \rho |\xi|^2.
\]

The computation of the $\Gamma_2$ for a general non-degenerate diffusion of the form~\eqref{eq_diff} is not easy and is discussed in \cite{arno:carl:ju:08}. In particular, in whole generality, the computation of the $\Gamma_2 $ requires the coefficients of the operator itself to be \textit{smooth} (i.e. at least ${\mathcal C}^2$). We also refer to \cite{arno:mark:tosc:unte:01} if the diffusion matrix is scalar diagonal, {\it i.e.} $\Sigma(x)=\varsigma(x)I_d $, $x\!\in \R^d$, where $\varsigma $ is real valued. In that case, it is then shown that \A{BE${}_\rho$} holds if and only if:
\begin{eqnarray}
\label{GAMMA_2_DIAG}
 - \frac 12 \langle (M(x)+M^*(x))\xi, \xi\rangle\le -\rho \varsigma(x) |\xi|^2,
\end{eqnarray} 
where 
\begin{eqnarray*}
 M(x)=\frac 12 \Big(\varsigma(x)\Delta \varsigma (x)+\langle b(x),\nabla \varsigma(x)\rangle -\|\nabla \varsigma(x)\|^2 \Big)I_d +\Big( \frac 12-\frac d4\Big)\nabla \varsigma \otimes \nabla \varsigma(x)-\varsigma(x)^2Db(x).
\end{eqnarray*}

An important property when \A{BE${}_\rho$} holds,  see again \cite{bakr:emer:85}, \cite{bakr:gent:ledo:14}, is that the following commutation inequality holds:
\begin{equation}\label{COMM_BE}
\forall t\ge 0,\ \forall x\in \R^d,\ \Gamma(P_t f)(x)\le  \exp(-2\rho t)P_t \Gamma(f).
\end{equation}

%
%

To conclude, let us say that the Bakry-Emery curvature condition is a very powerful tool to derive pointwise gradient bounds. In our framework, this is unfortunately not enough \textcolor{black}{as soon as $d>1$}, because additionally to this kind of bounds we also need, to enter in the framework of Schauder estimates under \A{C${}_{\rm \mathbf{UE}}$}, a control of the $\beta $-H\"older modulus of the gradient (see Section~\ref{SEC_COMPLETE_SMOOTHNESS}). It does not seem that the condition \A{BE${}_\rho $} helps to get such controls. The restrictions on the variations of $D\sigma $ appearing in both assumptions \A{C${}_{\rm \mathbf{UE}}$} and \A{C${}_{\rm \mathbf R}$} are precisely needed to derive in the first case the bounds on $[D\varphi]_\beta $ and in the second one to prove that the derivatives exist up to order 3 and that $[D^3 \varphi]_\beta $ is controlled as well. This explains why the conditions on $D\sigma $ are more stringent in the potentially degenerate setting \A{C${}_{\rm \mathbf R}$}. In each case, those bounds are obtained through pathwise analysis and the restrictions on $D\sigma $ ensure the time integrability of the iterated tangent flows, see again Section~\ref{SEC_COMPLETE_SMOOTHNESS} and \textcolor{black}{Appendix A in \cite{page:panl:14}} for details.

\subsection{Practical Deviation Bounds}\label{PRACT_DEV_BD}
\subsubsection{A first Non-asymptotic confidence interval result.}

\begin{theo}[Non-asymptotic confidence intervals without bias]\label{NA_CI_FIRST}
Let the assumptions of Theorem~\ref{Poisson_regular}  be in force. Then, there exists a unique invariant distribution $\nu $ for~\eqref{eq_diff}, i.e. \A{U} holds.
Also, $\varphi$ satisfies~\eqref{GV} introduced in Theorem~\ref{ineq_fin} \textcolor{black}{for $V(x)\asymp 1+|x|^2 $}. 

Assume that  \A{C1} (sub-gaussian tails of the innovation) holds and that the step sequence $(\gamma_k)_{k\ge 1} $ is such that $\gamma_k\asymp k^{-\theta},\ \theta\in (\frac{1}{2+\beta},1]$. Then, for $
(c_n)_{n\ge 1}, (C_n)_{n\ge 1}$ like in Theorem~\ref{ineq_fin} with  $ \lim_n c_n=\lim_n C_n=1$, we have that
for all $n\ge 1$ and $a>0$ and for \textcolor{black}{any matrix norm $\vvvert\cdot \vvvert $ dominating  $\| \cdot \|$}: 
\begin{equation}
\label{TRACT_BOUND_I}
\P\big[\sqrt{\Gamma_n}|\nu_n(f)-\nu(f)|>a\big]\le 2\,C_n \exp\Big(-c_n \frac{a^2 \alpha^2}{2 \vvvert \sigma\vvvert_\infty^2 [f]_1^2} \Big)\;\mbox{ with } \; \vvvert \sigma\vvvert_\infty:=\sup_{x\in \R^r}\vvvert \sigma(x)\vvvert,
\end{equation}
\begin{equation}
\label{NI_BORNE}
\P\left[\nu(f)\in \Big[\nu_n(f)-\frac{a \vvvert \sigma\vvvert_\infty[f]_1}{\alpha \sqrt{\Gamma_n}},\nu_n(f)+\frac{a \vvvert \sigma\vvvert_\infty[f]_1}{\alpha \sqrt{\Gamma_n}} \Big]\right]\ge 1-2C_n\exp\left(-c_n\frac{a^2}{2}
\right),
\end{equation}
where the parameter $\alpha $ is the same as in the pointwise gradient bound of Theorem~\ref{Poisson_regular}.
\end{theo}
\begin{proof}
Equation~\eqref{TRACT_BOUND_I} is a direct consequence of Theorem~\ref{ineq_fin} and the gradient bound in Theorem 
\ref{Poisson_regular}. \textcolor{black}{Indeed, the mean-value Theorem readily yields that~\eqref{GV} holds. It then suffice to
observe that $\nu_n(f)-\nu(f)=\nu_n({\mathcal A} \varphi) $}. To prove~\eqref{NI_BORNE},
setting $a_{\sigma,f,\alpha}:=a\vvvert \sigma \vvvert_\infty \frac{[f]_1}{\alpha} $, it suffices to write:
\begin{eqnarray*}
\P\left[\nu(f)\in \Big[\nu_n(f)-\frac{a_{\sigma,f,\alpha}}{\sqrt{\Gamma_n}},\nu_n(f)+\frac{a_{\sigma, f,\alpha}}{\sqrt{\Gamma_n}} \Big]\right]&=&1-\P[\sqrt{\Gamma_n}\big|\nu_n(f)-\nu(f)\big|\ge a_{\sigma,f,\alpha} ]
\end{eqnarray*}
and conclude by~\eqref{TRACT_BOUND_I}.
\end{proof}

\subsubsection{A more refined non-asymptotic confidence interval when $\vvvert \sigma\vvvert^2-\nu(\vvvert\sigma\vvvert^2) $ is a coboundary.}
\textcolor{black}{We provide in Theorem~\ref{Slutsky_theorem} below  a kind of~\textit{Slutsky's Lemma}
 when, for a matrix norm $\vvvert\cdot \vvvert $ dominating  $\| \sigma(x)\|\le \vvvert \sigma(x) \vvvert$,
s.t. $\vvvert\sigma\vvvert^2 -\nu(\vvvert\sigma\vvvert^2)$ is a coboundary}. 
 \begin{theo}[Slutsky type concentration result for the coboundary case]\label{Slutsky_theorem}
Under the assumptions of  Theorem~\ref{NA_CI_FIRST},
for $ \beta \in (0, 1 ]$
 and $ \theta \in (\frac{1}{2+\beta}, 1 ]$ (unbiased case), \textcolor{black}{assuming as well that there is a unique solution $\vartheta $ to ${\mathcal A}\vartheta=\vvvert\sigma\vvvert^2-\nu(\vvvert\sigma\vvvert^2) $ satisfying the same assumptions as $\varphi $ in Theorem~\ref{NA_CI_FIRST}}, there exist two explicit monotonic sequences 
$c_n\le 1\le C_n,\ n\ge 1$,   with $\lim_n C_n = \lim_n c_n =1$
such that for all $n\ge 1$, for all $a>0$, the following bounds hold:
 if  $\frac{a }{\sqrt{\Gamma_n}} \rightarrow 0$ (Gaussian deviations) then,
\begin{equation}
\label{DEV_SLUTSKY}
\P\big[ |\sqrt{\Gamma_n}\frac{\nu_n( f )-\nu(f)}{\sqrt{\nu_n(\vvvert \sigma\vvvert^2)}}| \geq a \big] 
\leq 2\,  C_n \exp\left(\! -  c_n\frac{ a^2 \alpha^2}{2 [f]_1^2}
\right),
\end{equation}
\begin{equation}
\label{NA_SLUTSKY}
\P\left[\nu(f)\in \Big[\nu_n(f)-\frac{a \sqrt{\nu_n(\vvvert \sigma\vvvert^2)} [f]_1}{\alpha \sqrt{\Gamma_n}},\nu_n(f)+\frac{a \sqrt{\nu_n(\vvvert \sigma\vvvert^2)} [f]_1}{\alpha \sqrt{\Gamma_n}} \Big]\right]\ge 1-2C_n\exp\left(-c_n\frac{a^2}{2}
\right).
\end{equation} 
\end{theo}
 Again, the non-asymptotic confidence interval is explicitly computable in function of the given source $f$, the coefficients in the dynamics and \textcolor{black}{the chosen (computable) matrix norm $\vvvert\cdot \vvvert $}. It is also sharper than the one in~\eqref{NI_BORNE}.
\subsubsection{Towards Lipschitz sources in the non-degenerate case}
We conclude this section stating a non-asymptotic deviation result for Lipschitz sources under some non-degeneracy conditions (assumption \A{C${}_{\rm \mathbf{UE}}$} of Theorem~\ref{Poisson_regular} replacing the condition stated there  for $f$ by a Lipschitz condition). 

\begin{theo}[Non-asymptotic concentration bounds for Lipschitz functions] \label{THEO_CTR_LIP}Let the assumptions of Theorem~\ref{Poisson_regular} with \A{C${}_{\rm \mathbf{UE}}$} hold except that $f$ is here solely a Lipschitz continuous function. For a time step sequence $(\gamma_k)_{k\ge 1} $ of the form $\gamma_k\asymp k^{-\theta} $, $\theta\in (1/2,1] $, we have that, there exist two explicit monotonic sequences  
$c_n\le 1\le C_n,\ n\ge 1$,   with $\lim_n C_n = \lim_n c_n =1$
such that for all $n \geq 1$ and for every $a >0$:  
 \begin{eqnarray}
\label{DEV_LIP}
\P\big[ |\sqrt{\Gamma_n}\big(\nu_n( f)-\nu(f)\big) | \geq a \big] \leq 2 C_n \exp\left( - c_n \frac{a^2\alpha^2}{2\|\sigma\|_\infty^2[f]_1^2} 
\right)
\end{eqnarray}
where $\alpha $ is as in Theorem~\ref{Poisson_regular}.
\end{theo}
Such estimates are important since they allow to get rather close to the natural framework  which appear in functional inequalities (that mainly deal with Wasserstein distances and their possible deviations). Indeed, through the Monge-Kantorovich formulation, the Wasserstein distance involves
Lipschitz functions, 
since it is precisely achieved taking the minimum over Lipschitz functions for all possible coupling with marginal corresponding to the arguments of the distance (see \cite{bakr:gent:ledo:14}). 

In the literature, some non-asymptotic bounds can be found for the deviations from its mean for the Wasserstein distance between the empirical measure of a homogeneous Markov chain and its stationary distribution (see Boissard \cite{bois:11}). Here, we manage to get directly the non-asymptotic deviation bounds over all possible Lipschitz functions for the empirical measure of the scheme aiming directly to approximate the target stationary distribution of the diffusion. Handling the Wasserstein distance in our framework would amount to consider the supremum over the Lipschitz functions in the probability in~\eqref{DEV_LIP}. This will concern  further research.

We eventually point out that Theorem~\ref{THEO_CTR_LIP} is obtained through regularization 
 arguments of the source $f$ exploiting the previous results of Theorems~\ref{ineq_fin} and~\ref{Poisson_regular} (see Section~\ref{REG_LIP_SOURCE} for details). This leads to a constraint on the steps, i.e. $\gamma_n \asymp n^{-\theta},\ \theta \in (\frac 12,1] $. This is the price to pay, indeed a bigger $\theta $ yields  a lower convergence rate, to handle less regular Lipschitz sources. Also, to perform the approximation procedure we precisely need a kind of elliptic bootstrap (like in Theorem~\ref{Poisson_regular} under \A{C${}_{\rm \mathbf{UE}}$}). This is why we impose the non-degeneracy assumptions. 
\section{Proof of the concentration results (Theorem~\ref{ineq_fin})} \label{SEC_PROOF_CONC}
For notational convenience, we say that assumption \A{A} holds whenever  \A{C1}, \A{GC}, \A{C2}, \A{$\mathbf{{\mathcal L}_V} $}, \A{U} and \A{S} are fulfilled. 
We assume throughout this section that \A{A} is in force and that the function $\varphi $ appearing in the lemmas satisfies the smoothness assumptions of Theorem~\ref{ineq_fin}.

\subsection{Strategy}
\label{SEC_STRAT}
To control the deviations of $\nu_n({\mathcal A}\varphi) $ we first give a decomposition lemma, obtained by a standard Taylor expansion. The idea is to perform a kind of \textit{splitting} between the deterministic contributions in the transitions and the random innovations. Doing so, we manage to prove that the contributions involving the innovations can be gathered into conditionally Lipschitz continuous functions of the noise, with small Lipschitz constant (functions $(\psi_{k}(X_{k-1},\cdot))_{k\in \leftB 1, n\rightB} $ below). These functions precisely give the Gaussian concentration, see Lemma~\ref{ineq_Psi}. The other terms, that we will call from now on ``remainders", will be shown to be uniformly controlled w.r.t. $n$ and do not give any asymptotic contribution in the ``fast decreasing" case $\theta>1/(2+\beta) $ (with the terminology of Theorem~\ref{ineq_fin}), see Lemmas~\ref{CTR_BIAS_1},~\ref{ineq_phi} and~\ref{ineq_rest}.

\begin{lemme}[Local Decomposition of the empirical measure
]
\label{decomp_nu}
For all $n \ge 1$ and $k\in \leftB 0,n-1\rightB$:
\begin{eqnarray}
\label{THE_DECOUP_LOC}
\varphi ( X_k) - \varphi ( X_{k-1}) &=&   \gamma_k\mathcal{A} \varphi(X_{k-1}) +  \left[\textcolor{black}{\gamma_k \int_0^1 \langle \nabla  \varphi (X_{k-1} + t \gamma_k b_{k-1})-\nabla \varphi(X_{k-1}),b_{k-1}\rangle dt}
\right.
\nonumber\\
&& \left.+ \frac 12 \gamma_k\, \Tr\Big( \big(D^2 \varphi ( X_{k-1} +  \gamma_k b_{k-1} ) - D^2\varphi( X_{k-1}) \big) \Sigma_{k-1}^2\Big) +  \psi_k(X_{k-1}, U_k)\right]\nonumber\\
&=:&\gamma_k {\mathcal A}(X_{k-1})+\Big(\psi_k(X_{k-1}, U_k) +R_{n,k}^1(X_{k-1})\Big),
\end{eqnarray}
where for all $k\in \leftB 1,n\rightB$, conditionally to $\F_{k-1} $, the mapping $u\mapsto \psi_k(X_{k-1}, u) $ is Lipschitz continuous in  $u$ with constant $\sqrt{\gamma_k}  \| \sigma_{k-1}\| \|\nabla \varphi \|_{\infty} $.
\end{lemme}
Introducing for a given $k$, the mapping $u\mapsto \Delta_k(X_{k-1},u):=\psi_k(X_{k-1},u)-\E\,[\psi_k(X_{k-1},U_k)|\F_{k-1}]$, we then rewrite:
$$  \varphi ( X_k) - \varphi ( X_{k-1})=\gamma_k\mathcal{A} \varphi(X_{k-1}) +\Delta_k(X_{k-1},U_k)+R_{n,k}(X_{k-1}),$$
with $R_{n,k}(X_{k-1}):=R_{n,k}^1(X_{k-1})+\E\,[\psi_k(X_{k-1},U_k)|\F_{k-1}]$.
The contribution $\Delta_k(X_{k-1},U_k) $ can be viewed as a martingale increment. Introduce now the associated (true) martingale 
\begin{equation}\label{INTRO_MART}
M_n:=\sum_{k=1}^n \Delta_k(X_{k-1},U_k).
\end{equation}
Summing over $k$ yields: 
\begin{equation}
\label{LOC_BIS_DPHI}
\varphi(X_n)-\varphi(X_0)=\Gamma_n \nu_n(\mathcal A \varphi)+M_n+\sum_{k=1}^n R_{n,k}(X_{k-1}).
\end{equation}

Defining $R_n:=\sum_{k=1}^n R_{n,k}(X_{k-1})+\varphi(X_0)-\varphi(X_n) $ we obtain the following decomposition of the empirical measure:
\begin{equation}
\label{REWRITE_NU_APHI_NEW}
\nu_n ({\mathcal A}\varphi)=-\frac{1}{\Gamma_n}( M_n+R_n).
\end{equation}

\begin{trivlist}
\item[-] \textit{Unbiased Case (Sub-Optimal Convergence Rate).} This case corresponds to \textit{fast} decreasing steps of the form $\gamma_k\asymp k^{-\theta} $, $\theta>1/(2+\beta) $.
To investigate the non-asymptotic deviations of the empirical measure, the idea is now to write for $a, \lambda>0 $:
\begin{eqnarray}
\P\big[ \sqrt{\Gamma_n}\nu_n({\mathcal A} \varphi)  \geq a\big]&\le&\exp\Big(-\frac{a \lambda}{\sqrt{\Gamma_n}}\Big)  \E\left[\exp\Big(-\frac{\lambda }{\Gamma_n}(M_n+R_n)\Big)\right]\notag\\
&\le & \exp\Big(-\frac{a \lambda}{\sqrt{\Gamma_n}}\Big)\E\left[\exp\Big(-\frac{q \lambda }{\Gamma_n}M_n\Big)\right]^{1/q}\E\left[\exp\Big(\frac{p \lambda }{\Gamma_n}|R_n|\Big)\right]^{1/p}, \frac 1p+\frac 1q=1,\ p,q>1.\label{EXPLI_DEV}
\end{eqnarray}
We actually aim to choose $q:=q(n)\underset{n}{\rightarrow} 1$. For  a suitable choice of $q$ satisfying the previous condition, we manage, in the fast decreasing case, to show that ${\mathscr R }_n:= \E\,[\exp(\frac{p \lambda }{\Gamma_n}|R_n|)]^{1/p}\underset{n}{\rightarrow} 1$. For the term involving the martingale $M_n$ we actually use the Gaussian concentration property \A{GC} of the innovation on its increments $(\Delta_k(X_{k-1},U_k))_{k\in \leftB 1, n\rightB} $. Namely, using the control of the Lipschitz constant of $\Delta_k(X_{k-1},\cdot) $ stated in Lemma~\ref{decomp_nu}, we derive: 
\begin{eqnarray}
\E\left[\exp\Big(-\frac{q \lambda }{\Gamma_n}M_n\Big)\right]&=& \E\left[\exp\Big(-\frac{q \lambda }{\Gamma_n}M_{n-1}\Big)\E\left[\exp\Big(-\frac{q \lambda }{\Gamma_n}\Delta_{n-1}(X_{n-1},U_n)\Big)\,\big|\,\F_{n-1}\right]\right]\notag\\
&\le& \E\left[\exp\Big(-\frac{q \lambda }{\Gamma_n}M_{n-1}\Big)\right] \exp\left(\frac{\lambda^2 q^2}{2\Gamma_n^2}\gamma_n \|\sigma\|_\infty^2\|\nabla \varphi\|_\infty^2\right)\le \exp\left(\frac{\lambda^2 q^2}{2\Gamma_n} \|\sigma\|_\infty^2\|\nabla \varphi\|_\infty^2\right),\label{BLUNT_CTR}
\end{eqnarray}
iterating the procedure to derive the last identity. From~\eqref{EXPLI_DEV}, we thus get:
$$\P\big[ \sqrt{\Gamma_n}\nu_n({\mathcal A} \varphi)  \geq a\big]\le{\mathscr R }_n\exp\Big(-\frac{a \lambda}{\sqrt{\Gamma_n}}+\frac{\lambda^2 q}{2\Gamma_n} \|\sigma\|_\infty^2\|\nabla \varphi\|_\infty^2\Big).$$
Keeping in mind that we manage to find $q:=q(n)\downarrow_n 1 $ such that the remainder ${\mathscr R }_n \downarrow_n 1$, the result of Theorem~\ref{ineq_fin} in the considered case then follows from a quadratic optimization over the parameter $\lambda $. 
\item[-] \textit{Biased Case (Optimal Convergence Rate).} This case corresponds to \textit{slow} decreasing steps of the form $\gamma_k\asymp k^{-\theta} $, $\theta=1/(2+\beta) $. In this setting, some terms of the remainder $R_n $ in~\eqref{REWRITE_NU_APHI_NEW} give a non trivial asymptotic contribution. We choose to substract them before studying the deviation (term ${\mathcal B}_{n,\beta} $ in~\eqref{DEF_QT_THEO_BIAS}).
\end{trivlist}

\subsection{Explicit controls on the remainders}

Summing the increments appearing in~\eqref{THE_DECOUP_LOC}, we now choose for the analysis to write for a given $n\in \N $ the remainder $R_n$ defined after~\eqref{LOC_BIS_DPHI} as 
$$R_n=\sum_{k=1}^n R_{n,k}(X_{k-1})+\varphi(X_0)-\varphi(X_n)=(D_{2,b,n}+D_{2,\Sigma,n})+\bar G_n-L_n,$$ 
where:
\begin{eqnarray}
\label{DECOUP}
D_{2,b,n}&:=&\sum_{k=1}^n  \textcolor{black}{\gamma_k \int_0^1 \langle \nabla  \varphi (X_{k-1} + t \gamma_k b_{k-1})-\nabla \varphi(X_{k-1}),b_{k-1}\rangle dt}
\nonumber\\
D_{2,\Sigma,n}&:=&\frac 12 \sum_{k=1}^n \gamma_k \Tr\Big(  \big(D^2 \varphi ( X_{k-1} +  \gamma_k b_{k-1} ) - D^2\varphi( X_{k-1}) \big) \Sigma_{k-1}^2 \Big),\nonumber\\
\bar G_{n}&:=&\sum_{k=1}^n \E\,[\psi_k(X_{k-1}, U_k)|\F_{k-1}],\notag\\
L_n&:=&\varphi ( X_n) - \varphi ( X_0).
\end{eqnarray}
We refer to the proof of Lemma~\ref{decomp_nu} to check that the above definition of $\bar G_n$ actually matches the term $\sqrt{\Gamma_n}E_n^\beta $ introduced in equation~\eqref{DEF_QT} of Theorem~\ref{ineq_fin}.
We rewrite from~\eqref{REWRITE_NU_APHI_NEW}
\begin{equation}
\label{REWRITE_NU_APHI}
\nu_n ({\mathcal A}\varphi)=-\frac{1}{\Gamma_n}( M_n+R_n)=-\frac{1}{\Gamma_n}\big( M_n+(D_{2,b,n}+D_{2,\Sigma,n})+\bar G_n-L_n\big).
\end{equation}
%

We now split the analysis according to the cases \textit{(a)} and \textit{(b)} introduced in Theorem~\ref{ineq_fin}.
\begin{trivlist}
\item[(\textit{a})] 
$\theta\in (1/(2+\beta),1] $, $\beta\in (0,1] $. From~\eqref{REWRITE_NU_APHI}, the exponential Tchebychev and H\"older inequalities yield that,  for all $\lambda \in \R_ +$ and all $p,q \in (1 , + \infty ) $, $ \frac{1}{p} + \frac{1}{q} =1$,
\begin{eqnarray}
\P\big[  \sqrt{\Gamma_n}\nu_n({\mathcal A} \varphi)  \geq a\big]
\leq  \exp\Big( - \frac{a\lambda}{\sqrt{\Gamma_n}} \Big)
 \left( \E \exp\Big( -\frac{q \lambda}{ \Gamma_n}M_n\Big)\right)^{\frac1q}\nonumber\\
\times \left(\E \exp\Big( \frac{2p \lambda}{ \Gamma_n}\big(\big| L_n\big| +\big| \bar G_n\big|\big)\Big)\right)^{\frac 1{2p}}\left(\E\exp\Big( \frac{4p \lambda}{ \Gamma_n}\big| D_{2,b,n}\big| \Big)\right)^{\frac 1{4p}}\left(\E\exp\Big( \frac{4p \lambda}{ \Gamma_n}\big| D_{2,\Sigma,n}\big| \Big)\right)^{\frac1{4p}}. \label{DECOUP_TCHEB}
\end{eqnarray}
\item[(\textit{b})] 
$\theta=\frac 1{2+\beta},\ \beta\in (0,1]$. If $\beta=1 $, denoting, $D_{2,n}:=D_{2,b,n}+D_{2,\Sigma,n}$, we have from~\eqref{DECOUP} and with the notations of~\eqref{DEF_QT_THEO_BIAS}, $ (\bar G_n+D_{2,n})=\sqrt{\Gamma_n}{\mathcal B}_{n,1}$. We study the deviations of:
\begin{eqnarray}
\P\big[  \sqrt{\Gamma_n}\nu_n({\mathcal A} \varphi)+{\mathcal B}_{n,\beta}  \geq a\big]&=&\P\Big[  \nu_n({\mathcal A} \varphi)+\frac{\bar G_n+D_{2,n}}{\Gamma_n}  \geq \frac a {\sqrt{\Gamma_n}} \Big]\nonumber\\
&\leq& \exp\Big( - \frac{a \lambda}{\sqrt{\Gamma_n}}\Big ) \left( \E \exp\Big( -\frac{q \lambda}{ \Gamma_n}M_n\Big)\right)^{\frac1q}
\left(\E \exp\Big( \frac{p \lambda}{ \Gamma_n}\big| L_n\big| \Big)\right)^{\frac 1{p}}.
\nonumber\\
 \label{DECOUP_TCHEB_BIAIS}
\end{eqnarray}
For $\beta\in (0,1) $, the contributions of $D_{2,n}$ do not yield any asymptotic bias. Recalling from~\eqref{DEF_QT_THEO_BIAS} that ${\mathcal B}_{n,\beta}=E_n^\beta=\frac{\bar G_n}{\sqrt {\Gamma_n}} $, we write:
\begin{eqnarray}
\P\big[  \sqrt{\Gamma_n}\nu_n({\mathcal A}\varphi)+{\mathcal B}_{n,\beta}   \geq a\big]&=&\P\Big[ \nu_n({\mathcal A} \varphi)+\frac{\bar G_n}{\Gamma_n}  \geq \frac a { \sqrt{\Gamma_n}}\Big]
\leq  \exp\Big( - \frac{a\lambda}{\sqrt{\Gamma_n}} \Big) \left( \E \exp\Big(- \frac{q \lambda}{ \Gamma_n}M_n\Big)\right)^{\frac1q}\nonumber\\
&&\times \left(\E \exp\Big( \frac{2p \lambda}{ \Gamma_n}\big| L_n\big| \Big)\right)^{\frac 1{2p}}\left(\E\exp\Big( \frac{4p \lambda}{ \Gamma_n}\big| D_{2,b,n}\big| \Big)\right)^{\frac 1{4p}}\left(\E\exp\Big( \frac{4p \lambda}{ \Gamma_n}\big| D_{2,\Sigma,n}\big| \Big)\right)^{\frac1{4p}}. \label{DECOUP_TCHEB_BIAS_BETA_NEQ_1TIERS}
\end{eqnarray}

\end{trivlist}
\begin{remark}
Observe that in case \textit{(a)}, the ``small steps" and the corresponding sufficient smoothness of $\varphi $ prevent from the appearance of a bias. As a result, the concentration bound is, at the non-asymptotic level, the same as in Theorem~\ref{theo}, up to the  additional upper-bound for the variance. 
In case \textit{(b)}, we subtract the terms $\textcolor{black}{{\mathcal B}_{n,\beta}} $ that asymptotically give a bias. When $\beta=1 $, this is the case for both terms $\frac{\bar G_n}{\Gamma_n}, \frac{D_{2,n}}{\Gamma_n}$. Also,  for $D^3\varphi\in {\mathcal C}^1$, $\textcolor{black}{{\mathcal B}_{n,1}}=\frac{\bar G_n+D_{2,n}}{\sqrt{\Gamma_n}}\underset{n}{\ra}-\widetilde \gamma m $ introduced in Theorem~\ref{theo}. For $\beta\in (0,1) $ and $\varphi \in {\mathcal C}^{3}(\R^d,\R),\ [\varphi^{(3)}]_\beta<+\infty $, the only term giving a bias is $\textcolor{black}{{\mathcal B}_{n,\beta}}=E_n^\beta=\frac{\bar G_n}{\sqrt{\Gamma_n}} $.

\end{remark}

The lemma below provides the Gaussian contribution 
\textcolor{black}{to be exploited in inequalities~\eqref{DECOUP_TCHEB} -~\eqref{DECOUP_TCHEB_BIAS_BETA_NEQ_1TIERS}}.
\begin{lemme}[Gaussian concentration] 
\label{ineq_Psi}
For $a> 0$, $q\in (1,+\infty)$, setting
\begin{equation}
\label{DEF_LAMBDA}
\lambda_n:=\frac{ a}{q\|\sigma\|_\infty^2 \|\nabla \varphi\|_\infty^2}\sqrt {\Gamma_n}, 
\end{equation}
we derive:
\begin{eqnarray*}
 \exp\left(-\lambda_n \frac{a}{\sqrt{\Gamma_n}}\right)\left(\E  \exp\Big(- \frac{q \lambda_n}{ \Gamma_n} M_n\Big) \right)^{\frac 1q}
\leq \exp\left(-\frac{a^2}{2q \|\sigma\|_\infty^2 \|\nabla \varphi\|_\infty^2}\right).
\end{eqnarray*}
\end{lemme}

\begin{lemme}[Bounds for the \textcolor{black}{Conditional Expectations}] 
 \label{CTR_BIAS_1}
With the above notations, we have that for $\beta\in (0,1], \theta\in [\frac{1}{2+\beta},1] $:
$$
|\textcolor{black}{E_n^\beta}|=\frac{|\bar G_n|}{\sqrt{\Gamma_n}}\le a_n:= \frac{ [\varphi^{(3)}]_{\beta}  \big\|\sigma\big\|_\infty^{(3+\beta)} \E\big[|U_1|^{3+\beta} \big] }{(1+\beta)(2+\beta)(3+\beta) }   \frac{\Gamma_n^{(\frac{3 +\beta}{2})}}{\sqrt{\Gamma_n}}, \,a.s.
$$
Moreover, $a_n\underset{n}{\ra} a_\infty$, with $a_\infty=0 $ if $\theta\in(\frac{1}{2+\beta},1] $ and $a_\infty>0 $ if $ \theta=\frac{1}{2+\beta}$. Also, for $\beta\in (0,1],\ \theta\in(\frac{1}{2+\beta},1]  $:
\begin{equation}
\label{CTR_BIAS_HOLDER}
\Bigg(\E\exp\Big(\frac{2p\lambda_n}{\Gamma_n}|\bar G_n|\Big)\Bigg)^{\frac1{2p}}\le \exp\Big(\frac{ \lambda_n }{\sqrt{\Gamma_n}}a_n \Big)\le \exp\Big(\frac{ \lambda_n^2}{2\Gamma_np}+\frac{a_n^2p}{2} \Big),\ \forall p>1.
\end{equation}
\end{lemme}

As indicated before, we now aim at controlling the remainders. In particular, from~\eqref{EXPLI_DEV} and~\eqref{DECOUP}, we are led to handle terms of the form 
$$\E\exp\Big (c \sum_{k=1}^n \gamma_k ^2 |b(X_{k-1})|^2\Big) \underset{\A{{\mathbf{\mathcal L}}_{ V } }}{\le} \E\exp\Big (c\, C_{_V} \sum_{k=1}^n \gamma_k ^2 |V(X_{k-1})|\Big)
$$
for small enough real constants $c>0$.

To this end, we will thoroughly rely on the following important integrability result for the Lyapunov function.
\begin{prop}\label{expV_int}
Under \A{A} there is a constant $c_V:=c_V(\A{A}) >0$ such that for all $\lambda\in[0,c_V] $, $\xi \in [0;1]$: 
$$
 I_V^\xi:=\sup_{n \geq 0 } \E\,[\exp(\lambda V_n^{\xi})] < + \infty.
$$
\end{prop}

We now have the following results for the terms appearing in~\eqref{DECOUP}.
\begin{lemme}[Initial term]
\label{ineq_phi}
Let $q\in (1,+\infty)$ be fixed and $\lambda_n$ be as in~\eqref{DEF_LAMBDA} in Lemma~\ref{ineq_Psi}. For functions $\varphi $ satisfying~\eqref{GV}, {\it i.e.} there exists $C_{V,\varphi}>0$ such that for all $x\in \R^d $, $|\varphi(x)|\le C_{V,\varphi}(1+\sqrt{V(x)})$, for $p:=\frac{q}{q-1} $ and $j\in \{1,2\} $:
\begin{eqnarray*}
\left(\E \exp\Big( jp \lambda_n \frac{|L_n| }{ \Gamma_n}\Big) \right)^{\frac{1}{jp}} \leq 
(I_V^1)^{\frac{1}{jp}} \exp\left(\frac{(j+1) p C_{V,\varphi}^2 \lambda_n^2}{ c_V \Gamma_n^2}+\frac{c_V}{p}\right)
=
 (I_V^1)^{\frac{1}{jp}} \exp\left(\frac{ (j+1) p C_{V,\varphi}^2a^2 }{c_V    q^2  \|\sigma\|_\infty^4 \| \nabla \varphi  \|_{\infty}^4 \Gamma_n}+\frac{c_V}{p}\right),
\end{eqnarray*}
with $c_V,\ I_{V}^1 $ like in Proposition~\ref{expV_int}.
\end{lemme}

\begin{lemme}[Remainders]\label{ineq_rest}
Let $q\in (1,+\infty)$ be fixed and $\lambda_n$ be as in Lemma~\ref{ineq_Psi}. Then, there exists $C_{\ref{ineq_rest_b}}:=C_{\ref{ineq_rest_b}}(\A{A},\varphi)$ such that for $p=\frac{q}{q-1}$:
\begin{equation}\label{ineq_rest_b}
\begin{split}
\left( \E  \exp\Big( \frac{4p \lambda_n}{ \Gamma_n}  \big|D_{2,\Sigma,n}\big|\Big)\right)^{\frac{1}{4p}} 
\leq \exp\left( C_{\ref{ineq_rest_b}}\frac{p \lambda_n^2  (\Gamma_n^{(2)})^2}{  \Gamma_n^2}\right) (I_V^1)^{\frac1{4p}}.
\end{split}
\end{equation}
We also have:
\begin{trivlist}
\item[-] If \textcolor{black}{the mapping $x\mapsto \langle \nabla \varphi(x), b(x) \rangle $ is Lipschitz continuous},
then there exists  $C_{\ref{ineq_rest_b2_phi}}:=C(\A{A},\varphi)>0$ such that
\begin{equation}\label{ineq_rest_b2_phi}
\begin{split}
\left(\E  \exp\Big( \frac{4p \lambda_n}{ \Gamma_n} \big|D_{2,b,n}\big|\Big) \right)^{\frac{1}{4p}} 
\leq \textcolor{black}{\exp\Big(C_{\ref{ineq_rest_b2_phi}}
\frac{p\lambda_n^2 (\Gamma_n^{(2)})^2}{\Gamma_n^2}
\Big)(I_V^1)^{\frac1{4p}}}.
\end{split}
\end{equation}

\item[-] For $a\le \frac{c_v  q}{4C_{_V}p}\frac{\|\sigma\|_\infty^2\|\nabla \varphi\|_\infty^2}{\|D^2 \varphi\|_\infty^2} \frac{\sqrt{\Gamma_n}}{\Gamma_n^{(2)}}$, there exists an $\R^+$-valued sequence $(v_n)_{n\ge 1}$ such that
$\big|v_n\big|\le C_{\ref{ineq_rest_b2}}:=C_{\ref{ineq_rest_b2}}(\A{A},\varphi) $ and
\begin{eqnarray}\label{ineq_rest_b2} 
\left( \E  \exp\Big( \frac{4p \lambda_n}{ \Gamma_n} \big|D_{2,b,n}\big|\Big) \right)^{\frac{1}{4p}} \leq (I_V^1)^{v_n}.
\end{eqnarray}
Also, $v_n\underset{n}{\rightarrow} v_\infty$ where $v_\infty=0$ if $ \theta>1/3$ and $v_\infty>0$ for $\theta=1/3 $.
\end{trivlist}
\end{lemme}
\noindent \textbf{Proof of Theorem ~\ref{ineq_fin}.}
From Lemma~\ref{ineq_Psi} we get:
\begin{equation}
\label{CONC_GAUSS_PREAL}
\left(\E  \exp\Big( -q\lambda_n \frac{M_n}{\Gamma_n} \Big)\right)^{\frac 1q} \exp\Big( -\frac{a  \lambda_n}{ \sqrt{\Gamma_n}} \Big)
\leq
\exp\left( - \frac{ a^2}{2q \|\sigma\|_\infty^2 \|\nabla \varphi\|_\infty^2}\right).
\end{equation}
\begin{trivlist}
 \item[\textit{(a)}] We deal with the case $\beta \in (0,1], \theta\in (\frac1{2+\beta},1] $. 
 
  
 \item[\textit{(i)}] \textcolor{black}{We suppose that the mapping $x\mapsto \langle \nabla \varphi(x),b(x)\rangle $ is Lipschitz continuous}. 
 Plugging in~\eqref{DECOUP_TCHEB} the controls from~\eqref{CONC_GAUSS_PREAL}, Lemma~\ref{CTR_BIAS_1} equation~\eqref{CTR_BIAS_HOLDER}, Lemma~\ref{ineq_phi} (with $j=2$) and Lemma~\ref{ineq_rest} (equations~\eqref{ineq_rest_b},~\eqref{ineq_rest_b2_phi}), we get:
 \begin{eqnarray}
\P\left[ \nu_n( \mathcal{A} \varphi ) \geq \frac a{\sqrt{\Gamma_n}} \right] \leq  
\exp\left( - \frac{ a^2}{2q \|\sigma \|_\infty^2 \|\nabla \varphi\|_\infty^2}
\right) \exp\Big(\frac{\lambda_n^2}{2\Gamma_np}+\frac{pa_n^2}{2}\Big)  \exp\left(\frac{3 p C_{V,\varphi}^2 \lambda_n^2}{ c_V \Gamma_n^2}+\frac{c_V}{p}\right)(I_V^1)^{\frac 1{2p}}\nonumber\\
\times\exp\left ( C_{\ref{ineq_rest_b}}\frac{p \lambda_n^2  (\Gamma_n^{(2)})^2}{  \Gamma_n^2}\right) (I_V^1)^{\frac 1{4p}}\times \exp\left(C_{\ref{ineq_rest_b2_phi}}
\textcolor{black}{\frac{p\lambda_n^2 (\Gamma_n^{(2)})^2}{\Gamma_n^2}}
\right)(I_V^1)^{\frac1{4p}}\nonumber\\
\le  (I_V^1)^{\frac 1p}\exp\left(-\frac{a^2}{2q \|\sigma\|_\infty^2 \|\nabla \varphi\|_\infty^2}\Big(1-\frac{1}{q\|\sigma\|_\infty^2\|\nabla \varphi\|_\infty^2}\Big\{\frac{p}{\Gamma_n}\Big(\frac{6C_{V,\varphi}^2}{c_V}+\textcolor{black}{2\big[C_{\ref{ineq_rest_b}}+ C_{\ref{ineq_rest_b2_phi}}]}(\Gamma_n^{(2)})^2  \Big)+\frac{1}{p}\Big\}\Big)\right)\nonumber\\ \times \exp\Big(\textcolor{black}{\frac{c_V}{p}} 
+\frac{pa_n^2}{2}\Big).\label{BD_P_THM_2}
\end{eqnarray}
 Recall now that for $\theta> \frac{1}{2+\beta}\ge 1/3 $, $\Gamma_n^{(\frac{3+\beta}2)}/\sqrt{\Gamma_n} \underset{n}{\ra}0,\ \Gamma_n^{(2)}/\sqrt{\Gamma_n} \underset{n}{\ra}0$ (see Lemma~\ref{CTR_BIAS_1} and Remark~\ref{rem_rate_BIS}).
 We now  take $p:=p_n \underset{n}{\rightarrow}+\infty$, and therefore $q:=q_n\underset{n}{\ra}1$, such that $p_n^{1/2} \frac{\Gamma_n^{(\frac{3+\beta}{2})}}{\sqrt{\Gamma_n}}\underset{n}{\rightarrow} 0$ so that from Lemma~\ref{CTR_BIAS_1}, $p_n a_n^2\underset{n}{\rightarrow} 0 $. Since $\frac{\Gamma_n^{(\frac{3+\beta}{2})}}{\sqrt{\Gamma_n}}\ge \frac{\Gamma_n^{(2)}}{\sqrt{\Gamma_n}} $ this in turn implies:
 \begin{equation}
 \label{DEF_DN_A}
 d_n:=  \frac{1}{q_n\|\sigma\|_\infty^2\|\nabla \varphi\|_\infty^2}\Big\{ \frac{p_n}{ \Gamma_n }\Big(\frac{6C_{V,\varphi}^2}{c_V}+\big[2C_{\ref{ineq_rest_b}}+3 C_{\ref{ineq_rest_b2_phi}}\big](\Gamma_n^{(2)})^2\Big) +\frac{1}{p_n} \Big\}\underset{n}{\rightarrow} 0. 
 \end{equation}
We conclude from~\eqref{BD_P_THM_2} setting $c_n=q_n^{-1}(1-d_n), C_n:= (I_V^1)^{\frac 1{p_n}}\exp(\frac{1}{p_n}[c_V+\frac{C_{\ref{ineq_rest_b2_phi}}}{2}]+\frac{p_na_n^2}{2})\underset{n}{\rightarrow}1$. Observe that taking an increasing sequence $(p_n)_{n\ge 1}$ readily yields $C_n \downarrow_n 1 $, and $q_n\downarrow_n 1 $. Also, the sequence $(p_n)_{n\ge 1} $ can be chosen in order to have, for $n$ large enough, $d_n\downarrow_n 0 $ so that $c_n\uparrow_n 1 $.
\item[\textit{(ii)}] Assume $a\le \frac{c_V  q}{4C_{_V}p}\frac{\|\sigma\|_\infty^2\|\nabla \varphi\|_\infty^2}{\|D^2 \varphi\|_\infty^2}\frac{\sqrt{\Gamma_n}}{\Gamma_n^{(2)}}$.
 Plugging in~\eqref{DECOUP_TCHEB} the controls from~\eqref{CONC_GAUSS_PREAL}, Lemma~\ref{CTR_BIAS_1}, equation~\eqref{CTR_BIAS_HOLDER} , Lemmas~\ref{ineq_phi} (with $j=2 $),~\ref{ineq_rest} (equations~\eqref{ineq_rest_b},~\eqref{ineq_rest_b2}) we then derive:
 \begin{eqnarray}
\P\Big[ \nu_n( \mathcal{A} \varphi )
 \geq \frac a{\sqrt{\Gamma_n}} \Big] 
\leq  
 \exp\left( - \frac{ a^2}{2q \|\sigma\|_\infty^2 \|\nabla \varphi\|_\infty^2}\right) \exp\Big(\frac{\lambda_n^2}{2\Gamma_np}+\frac{pa_n^2}{2}\Big)  \exp\left(\frac{3 p C_{V,\varphi}^2 \lambda_n^2}{ c_V \Gamma_n^2}+\frac{c_V}{p}\right)(I_V^1)^{\frac 1{2p}}\nonumber\\
\times\exp\left( C_{\ref{ineq_rest_b}}\frac{p \lambda_n^2  (\Gamma_n^{(2)})^2}{  \Gamma_n^2}\right) (I_V^1)^{\frac1{4p}}(I_V^1)^{v_n}\nonumber\\
\le 
(I_V^1)^{v_n+\frac3{4p}}\exp\left( \frac{c_V}{p}+\frac{pa_n^2}{2}\right)\exp\left(-\frac{a^2}{2q \|\sigma\|_\infty^2 \|\nabla \varphi\|_\infty^2}\Big(1-\frac{1}{q\|\sigma\|_\infty^2 \|\nabla \varphi\|_\infty^2}\Big\{\frac{p}{\Gamma_n }\Big(\frac{6C_{V,\varphi}^2}{c_V}+2C_{\ref{ineq_rest_b}}(\Gamma_n^{(2)})^2  \Big)+\frac{1}{p}\Big\}\Big)\right).\nonumber\\
\label{BD_P_THM_1}
\end{eqnarray} 
Since $\theta> \frac{1}{2+\beta}\ge 1/3 $ (see Remark~\ref{rem_rate_BIS}), 
we again take $p:=p_n \uparrow_n
+\infty$ so that $p_n^{1/2} a_n \underset{n}{\rightarrow} 0 $ which also guarantees:
 \begin{equation}
 \label{DEF_DN_B}
 d_n:=\frac{1}{q_n \|\sigma \|_\infty^2\|\nabla \varphi\|_\infty^2}\Big\{ p_n \Big(\frac{6C_{V,\varphi}^2}{ c_V\Gamma_n }+\frac{2C_{\ref{ineq_rest_b}} (\Gamma_n^{(2)})^2}{ \Gamma_n}\Big)+\frac{1}{p_n}\Big\}\underset{n}{\ra}
 0. 
 \end{equation}
 
 In this case, we derive the result by setting $c_n:=q_n^{-1}(1-d_n) \underset{n}{\rightarrow } 
 1 ,\ C_n:=(I_V^1)^{v_n+\frac3{4p_n}}\exp(\frac{c_V}{p_n}+\frac{p_na_n^2}2)
 \underset{n}{\rightarrow }
 1$ (see the limits of $v_n$ following equation~\eqref{ineq_rest_b2} and~\eqref{DEF_VN}). Again, $(p_n)_{n\ge 1} $ can be chosen in order to have the stated monotonicity for $n$ large enough. Set now 
 \begin{equation}
\label{DEF_CHI_N}
 \chi_n:=\frac{c_V\|\sigma\|_\infty^2\|\nabla \varphi\|_\infty^2}{4C_{_V} \|D^2\varphi\|_\infty^2}\frac{q_n}{p_n},
 \end{equation}
  so that $a\le \chi_n \frac{\sqrt{\Gamma_n}}{\Gamma_n^{(2)}} $. Thus, the slower $p_n$ goes to infinity, the wider the domain of validity  for the estimate in the parameter $a$.

\item[\textit{(b)}]It remains to analyze the case $\beta\in (0,1], \theta=\frac{1}{2+\beta} $. Let us deal with  $\beta=1 $.
From~\eqref{DECOUP_TCHEB_BIAIS}, the controls of~\eqref{CONC_GAUSS_PREAL} and  Lemma~\ref{ineq_phi} (with $ j=1$) we get:
 \begin{eqnarray*}
\P\left[ \nu_n( \mathcal{A} \varphi )+\frac{\bar G_n+D_{2,n}}{\Gamma_n}\geq \frac a{\sqrt{\Gamma_n}} \right] \leq  \exp\left( - \frac{ a^2}{2q \|\sigma \|_\infty^2 \|\nabla \varphi\|_\infty^2}\right)   \exp\left(\frac{2 p C_{V,\varphi}^2 \lambda_n^2}{ c_V \Gamma_n^2}+\frac{c_V}{p}\right)(I_V^1)^{\frac 1{p}}.
\end{eqnarray*}
Recalling the definition of $\lambda_n $ in~\eqref{DEF_LAMBDA}, we conclude as previously with obvious modifications of $(c_n)_{n\ge 1}$, $(C_n)_{n\ge 1} $. The case $\beta \in (0,1)$ is handled similarly starting from~\eqref{DECOUP_TCHEB_BIAS_BETA_NEQ_1TIERS}.

Also, when $D^3\varphi\in {\mathcal C}^1 $, we derive similarly to the proof of Theorem 10 in~\cite{lamb:page:02} that ${\mathcal B}_{n,1}\underset{n}{\rightarrow}-\widetilde \gamma m $.

Eventually, the final control involving the two sided deviation is derived by symmetry.
\hfill $\square$
\end{trivlist}
%

\subsection{Proof of the Technical Lemmas}
This section is devoted to the proof of the previously used Lemmas~\ref{decomp_nu}--\ref{ineq_rest}  and Proposition~\ref{expV_int} which were the key ingredients to derive Theorem~\ref{ineq_fin}.\\

\label{SEC_TEC}
\noindent\textbf{Proof of Lemma~\ref{decomp_nu}.}
For $k\in \leftB 1, n\rightB $, we first write:
\begin{equation}
\label{SPLIT_TRANS}
\begin{split}
\varphi(X_k)-\varphi(X_{k-1})&=(\varphi(X_k)- \varphi(X_{k-1} +\gamma_k b_{k-1}))+(\varphi(X_{k-1} +\gamma_k b_{k-1})-\varphi(X_{k-1}))\\
&=:\textcolor{black}{T_{k-1,r}(\varphi)+T_{k-1,d}(\varphi)},
\end{split}
\end{equation}
in order to split the \textcolor{black}{random} and \textcolor{black}{deterministic} contributions in the transitions of the scheme~\eqref{scheme}. 

We then perform a Taylor expansion with integral remainder at order 2 for the function $\varphi$ in the two terms of the r.h.s. of~\eqref{SPLIT_TRANS}. Namely, with the above notations:
\begin{eqnarray*}
T_{k-1,d}(\varphi) 
& = & \gamma_k b_{k-1} \cdot \nabla \varphi(X_{k-1})  + \textcolor{black}{\gamma_k \int_0^1 \langle \nabla  \varphi (X_{k-1} + t \gamma_k b_{k-1})-\nabla \varphi(X_{k-1}),b_{k-1}\rangle dt},
\\
T_{k-1,r}(\varphi)
&= & \sqrt{\gamma_k} \sigma_{k-1} U_k \cdot \nabla \varphi( X_{k-1} +\gamma_k b_{k-1} )  
\\
&& +\gamma_k  \int_0 ^1 ( 1-t) \Tr\Big (D^2\varphi ( X_{k-1} +\gamma_k b_{k-1} + t \sqrt{\gamma_k} \sigma_{k-1} U_k ) \sigma_{k-1} U_k\otimes U_k  \sigma_{k-1}^*\Big)dt.
\end{eqnarray*}
Hence,
\begin{eqnarray}
\varphi ( X_k)  \! \!& \! \!- \! \!& \! \! \varphi ( X_{k-1} )  =  \gamma_k \mathcal{A} \varphi(X_{k-1}) 
\nonumber\\
 &&+ 
\textcolor{black}{\gamma_k \int_0^1 \langle \nabla  \varphi (X_{k-1} + t \gamma_k b_{k-1})-\nabla \varphi(X_{k-1}),b_{k-1}\rangle dt}
+
\sqrt{\gamma_k} \sigma_{k-1} U_k \cdot \nabla \varphi( X_{k-1} +\gamma_k b_{k-1} )
\nonumber\\
&&+ \gamma_k \int_0 ^1 ( 1-t) \Tr\Big (  D^2 \varphi ( X_{k-1} +\gamma_k b_{k-1} + t \sqrt{\gamma_k} \sigma_{k-1} U_k )\sigma_{k-1}U_k\otimes U_k \sigma_{k-1}^*   -  D^2\varphi(X_{k-1} )\Sigma_{k-1} \Big) dt
\nonumber\\
&=&  \gamma_k \mathcal{A} \varphi(X_{k-1})  + \textcolor{black}{\gamma_k \int_0^1 \langle \nabla  \varphi (X_{k-1} + t \gamma_k b_{k-1})-\nabla \varphi(X_{k-1}),b_{k-1}\rangle dt}
\nonumber\\
&&+  \gamma_k  \int_0^1 (1-t) \Tr\Big(\big( D^2 \varphi ( X_{k-1} +  \gamma_k b_{k-1} ) - D^2\varphi( X_{k-1})\big) \Sigma_{k-1} \Big)dt + \psi_k(X_{k-1}, U_k)\nonumber\\
&=:& \gamma_k \mathcal{A} \varphi(X_{k-1})+D_{2,b}^k+D_{2,\Sigma}^k+\psi_k(X_{k-1}, U_k),
\label{DECOUP_PREAL_LEMME1}
\end{eqnarray}
where 
\begin{eqnarray}
 \label{DEF_PSI_K} 
& \psi_k(X_{k-1}, U_k) =\sqrt{\gamma_k} \sigma_{k-1} U_k \cdot \nabla \varphi( X_{k-1} +\gamma_k b_{k-1} ) \hskip 8,5cm \notag\\
 &    +\; \gamma_k  \int_0 ^1 ( 1-t) \Tr\Big (  D^2\varphi  ( X_{k-1} +\gamma_k b_{k-1} + t \sqrt{\gamma_k} \sigma_{k-1} U_k ) \sigma_{k-1}U_k\otimes U_k \sigma_{k-1}^*   - D^2\varphi ( X_{k-1} +\gamma_k b_{k-1}  ) \Sigma_{k-1} \Big )dt.\hskip -1cm \nonumber\\
\end{eqnarray}
Observe now that, conditionally to $\F_{k-1} $, the mapping  $ u\mapsto \psi_k ( X_{k-1}, u)$ is Lipschitz continuous: indeed, the innovation $U_k$ does not appear in the other contributions of the right side of~\eqref{DECOUP_PREAL_LEMME1}. Consequently, as $\varphi $ is Lispchitz continuous we derive, for all $(u,u')\in (\R^d)^2$:
\begin{equation*}
|\psi_k ( X_{k-1}, u)-\psi_k ( X_{k-1}, u')|\le \sqrt{\gamma_k}  \|\sigma_{k-1}\|\,\|\nabla \varphi\|_{\infty}|u-u'|.
\end{equation*}
The result is obtained by summing up the previous identities from $k=1 $ to $n$, observing, with the notations of~\eqref{DECOUP}, that $L_n=\sum_{k=1}^n \varphi(X_k)-\varphi(X_{k-1}), D_{2,b,n}=\sum_{k=1}^n D_{2,b}^k,\ D_{2,\Sigma,n}=\sum_{k=1}^n D_{2,\Sigma}^k, \ G_n:=\sum_{k=1}^n \psi_k ( X_{k-1}, U_k)$. \hfill $\square $\\
\\
\noindent\textbf{Proof of Lemma~\ref{ineq_Psi}.}
The idea is to use conditionally and iteratively the Gaussian concentration property \A{GC} of the innovation. 
Let us note that this strategy was already the key ingredient in~\cite{frik:meno:12}.
In the current framework, we exploit that  the functions $u\mapsto \Delta_k(X_{k-1},u):=\psi_k(X_{k-1},u)-\E\,[\psi_k(X_{k-1},U_k)|\F_{k-1}]$ are conditionally independent w.r.t. $\F_{k-1} $ and Lipschitz continuous with constant $\sqrt{\gamma_{k}}\|\sigma\|_\infty \|\nabla \varphi\|_\infty $ by Lemma~\ref{decomp_nu}. We thus write:
\begin{eqnarray}
\E\exp\Big(-\frac{q \lambda}{\Gamma_n}M_n \Big)&=& \E \exp\left(- \frac{q \lambda}{\Gamma_n} \sum_{k=1}^n \Delta_k(X_{k-1}, U_k) \right)
\nonumber \\
&=&  \E\Big[\exp\Big(-\frac{q \lambda}{\Gamma_n} \sum_{k=1}^{n-1} \Delta_k(X_{k-1}, U_k) \Big)
 \E \Big[ \exp\Big( -\frac{q \lambda}{ \Gamma_n} \Delta_{n}(X_{n-1}, U_n) \Big) | \mathcal{F}_{n-1} \Big] \Big]
 \nonumber\\
&\leq&  \E\Big[\exp\Big(-\frac{q \lambda}{\Gamma_n} \sum_{k=1}^{n-1} \Delta_k(X_{k-1}, U_k) \Big) 
\exp\Big( \frac{q^2 \lambda^2 }{2 \Gamma_n^2} \gamma_n \| \sigma\|_\infty^2\| \nabla \varphi \|_{\infty}^2  \Big)\Big]\label{PREAL_ITE},
\end{eqnarray}
where we used \A{GC} in the third line recalling as well that $\E\,[\Delta_n(X_{n-1},U_n)|\F_{n-1}]=0 $.

Iterating the process over $k$, we obtain:
 \begin{eqnarray}
\label{CTR_MART}
\left(\E\exp\Big(-\frac{q \lambda}{\Gamma_n}M_n \Big)\right)^{\frac 1q}=\left( \E  \exp\Big( -\frac{q \lambda}{ \Gamma_n} \sum_{k=1}^n \Delta_k(X_{k-1}, U_k)\Big) \right)^{\frac 1q}
\leq \exp\Big( \frac{ q\lambda^2  \|  \sigma \|_{\infty}^2 \| \nabla \varphi \|_{\infty}^2 }{2 \Gamma_n} \Big). 
\end{eqnarray}
Finally,
$$\exp\Big(-\frac{\lambda a}{\sqrt{\Gamma_n}}\Big) \left( \E\exp\Big(-\frac{q\lambda}{\Gamma_n} M_n\Big )\right)^{\frac 1q}\le \exp\Big(\frac{g(\lambda)}{\sqrt {\Gamma_n}} \Big),
$$
where $g:\R^+\rightarrow \R $ is defined by $g(\lambda)= -\frac{a}{\sqrt{\Gamma_n}}\lambda+\frac{q \lambda^2}{2\Gamma_n}\|\sigma\|_\infty^2  \|\nabla \varphi\|_\infty^2$. As $a>0$, the function attains its minimum at $\lambda_n $ given in~\eqref{DEF_LAMBDA}. This eventually yields the expected bound.
\hfill $\square $
\vspace*{10pt}

\noindent\textbf{Proof of Lemma~\ref{CTR_BIAS_1}.}
From the definition in~\eqref{DEF_PSI_K} and the Fubini theorem, we have that for all $k \in \leftB 1,n \rightB$:
\begin{eqnarray}
\E\,[ \psi_k(X_{k-1}, U_k) | \mathcal{F}_{k-1}]
  &= \gamma_k \int_0 ^1  (1-t) \Tr  \Big( \E\big[D^2\varphi  ( X_{k-1} + \gamma_k b_{k-1}  + t \sqrt{\gamma_k} \sigma_{k-1} U_k ) \sigma_{k-1} U_k\otimes U_k \sigma_{k-1}^* \nonumber\\
&
 - D^2\varphi ( X_{k-1} + \gamma_k b_{k-1} ) \Sigma_{k-1}| \mathcal{F}_{k-1}\big] \Big) dt.
\label{DEV_EC}
\end{eqnarray}
Recalling that $U_k$ has the same moments as the  standard Gaussian  random variable up to order three (see \A{GC}) and is independent of $\F_{k-1} $, a Taylor expansion yields:
\begin{eqnarray*}
&& \E\Big[\Tr\Big( D^2\varphi ( X_{k-1} + \gamma_k b_{k-1}  + t \sqrt{\gamma_k} \sigma_{k-1} U_k )\sigma_{k-1} U_k\otimes U_k \sigma_{k-1}^* - D^2 \varphi ( X_{k-1} + \gamma_k b_{k-1} ) \Sigma_{k-1}\Big) \Big| \mathcal{F}_{k-1}\Big]
\\
&=& \Tr\Big( D^2 \varphi ( X_{k-1} + \gamma_k b_{k-1} )\sigma_{k-1}\E\,[U_k\otimes U_k]\sigma_{k-1}^*\Big) \\
&& +\int_0 ^1 \E\Big[\Tr\Big(  \big( D^3 \varphi( X_{k-1} + \gamma_k b_{k-1} + u t\sqrt{\gamma_k} \sigma_{k-1} U_k ) t \sqrt{\gamma_k} \sigma_{k-1} U_k\big)\big(\sigma_{k-1} U_k\otimes U_k \sigma_{k-1}^*\big) \Big) \Big|{\F_{k-1}}\Big]du
\\
&&- \Tr\Big(D^2\varphi ( X_{k-1} + \gamma_k b_{k-1} ) \Sigma_{k-1} \Big)
\\
&=&   \Tr\Big(   D^2 \varphi  ( X_{k-1} + \gamma_k b_{k-1} )\sigma_{k-1} \underbrace{(\E\,[U_k\otimes U_k]-I)}_{=0}\sigma_{k-1}^*\Big)  
\\
&& +
t\sqrt {\gamma_k}\int_0 ^1 \E\Big[\Tr\Big( \big( [D^3 \varphi( X_{k-1} + \gamma_k b_{k-1} + u t\sqrt{\gamma_k} \sigma_{k-1} U_k )  
-D^3 \varphi( X_{k-1} + \gamma_k b_{k-1}  )]
\sigma_{k-1} U_k\big) \\
&&\times \big(\sigma_{k-1} U_k\otimes U_k \sigma_{k-1}^*\big) \Big) \Big|{\F_{k-1}}\Big]du,
\end{eqnarray*}
recalling from \A{GC} that for all $(i,j,l)\in \leftB 1,r \rightB $, $\E\,[U_k^iU_k^jU_k^l|\F_{k-1}]=\E\,[U_1^iU_1^jU_1^l]=0$ (cancellation argument).
Hence,
\begin{eqnarray*}
|\E\,[ \psi_k(X_{k-1}, U_k) | \mathcal{F}_{k-1}]|&\leq&\gamma_k\int_0^1 (1-t) t^{1+\beta} [\varphi^{(3)}]_{\beta}  \E\Big[ \gamma_k^{\frac{1 +\beta}{2}} \|\sigma_{k-1}\|^{3+\beta} |U_k| ^{3+\beta}  \int_0 ^1  u^{\beta} du \Big| \mathcal{F}_{k-1}\Big]dt 
\\
&=& \frac{ [\varphi^{(3)}]_{\beta} \gamma_k^{\frac{3 +\beta}{2}} \|\sigma_{k-1}\|^{3+\beta} \E\,[|U_k| ^{3+\beta} ] }{(1+ \beta)(2+\beta)(3+\beta)},
\end{eqnarray*}
recalling that the third derivatives of $\varphi $ are $\beta $-H\"older continuous for the first inequality.
We thus derive:
\begin{eqnarray*}
|\textcolor{black}{E_n^\beta}|=\frac{|\bar G_n|}{ \sqrt{\Gamma_n}} \le
\frac{1}{\sqrt{\Gamma_n}} \sum_{k=1}^{n-1} \big|\E\big[\psi_k(X_{k-1}, U_k)|\F_{k-1} \big]\big|
\le  \frac{ [\varphi^{(3)}]_{\beta}  \|\sigma\|_\infty^{3+\beta} \E\,[|U_1| ^{3+\beta} ] }{(1+ \beta)(2+ \beta)(3+ \beta)}\frac{\Gamma_n^{(\frac{3 +\beta}{2})}}{\sqrt{\Gamma_n}}=:a_n.
\end{eqnarray*} 



\noindent\textbf{Proof of Proposition~\ref{expV_int}}.
First of all, let us decompose the Lyapunov function $V$ with a Taylor expansion like in Lemma~\ref{decomp_nu}. We again use a splitting between the deterministic contributions and those involving the innovation.
We write for all $n \in \N$:
\begin{eqnarray}\label{TAYLOR_V}
 V(X_n)-V(X_{n-1}) &=& \gamma_n \mathcal{A} V(X_{n-1})
+\gamma_n^2 \int_0^1 (1-t) \Tr\Big(  D^2V (X_{n-1} + t \gamma_n b_{n-1} ) b_{n-1}\otimes b_{n-1}\Big)dt 
\nonumber\\
&&- \frac{\gamma_n}{2} \Tr\big( D^2V( X_{n-1})) \Sigma_{n-1} \big) +\sqrt{\gamma_n} \sigma_{n-1} U_n \cdot \nabla V( X_{n-1} +\gamma_n b_{n-1} )\nonumber\\
&&+ \gamma_n  \int_0 ^1 ( 1-t) \Tr\Big (  D^2V  ( X_{n-1} +\gamma_n b_{n-1} + t \sqrt{\gamma_n} \sigma_{n-1} U_n ) \sigma_{n-1}U_n\otimes U_n \sigma_{n-1}^*  \Big )dt\nonumber
\\ 
& \leq&
-\gamma_n \alpha_{_V} V(X_{n-1}) + \gamma_n \beta_V
+C_{_V} \frac{\gamma_n^2}{2} \| D^2V\|_{\infty} V(X_{n-1}) \nonumber
\\
&&+ \frac{\gamma_n}{2}  \|D^2V \|_{\infty}\|\sigma\|_{\infty}^2 +\sqrt{\gamma_n} \sigma_{n-1} U_n \cdot \nabla V( X_{n-1} +\gamma_n b_{n-1} ) + \frac{\gamma_n}{2}  \|D^2V  \|_{\infty} \|\sigma\|_{\infty}^2 |U_n|^{2}\nonumber
\\ 
&\leq& \gamma_n \Big(- \frac{\alpha_{_V}}{2}V(X_{n-1}) + \widetilde c\Big)  +\sqrt{\gamma_n} \sigma_{n-1} U_n \cdot \nabla V( X_{n-1} +\gamma_n b_{n-1} ) + \frac{\gamma_n}{2}  \|D^2V  \|_{\infty} \|\sigma\|_{\infty}^2 |U_n|^{2}
\end{eqnarray}
for a constant $ \widetilde c:=\widetilde c(V,\sigma,\beta_V)$. We have in fact considered the time steps sufficiently small (in \A{S}, we have chosen for all $n \in \N$, $\gamma_n < \min(\frac{1}{2\sqrt{C_{_V}\bar c}}, \frac{\alpha_{_V}}{2C_{_V} \| D^2 V \|_{\infty}}) $).
The two terms involving the innovation $U_n $ in the above decomposition can be controlled thanks to the Gaussian concentration hypothesis \A{GC}.  
Let us define for all $x \in \R^d$ and all $\gamma, \lambda >0$ the quantities:
\begin{eqnarray*}
I_1(\gamma,\lambda, x):=\E\Big [ \exp \big (\lambda \sqrt{\gamma} \sigma(x) U_1 \cdot \nabla V(x+\gamma b(x)) \big ) \Big ]
, \
I_2(\gamma,\lambda):=\E\Big [ \exp \big (\lambda \frac{\gamma}{2}   \|D^2V  \|_{\infty} \|\sigma\|_{\infty}^2 |U_1|^{2} \big ) \Big ].
\end{eqnarray*}
The first one is directly controlled owing to  hypothesis \A{GC}:
\begin{eqnarray}\label{INTEGRABILITY_V1}
I_1(\gamma_n,\lambda,x)&\leq& \exp \Big ( \frac{\lambda^2 \gamma_n  |\sigma^*(x)\nabla V(x+\gamma_n b(x))|^2}{2} \Big )
\underset{\A{{\mathcal L}_V}}{\leq} \exp \Big ( \frac{\lambda^2 \gamma_n C_{_V} \| \sigma \|_{\infty}^2 V(x+\gamma_n b(x))}{2} \Big ).
\end{eqnarray}

Furthermore, under \A{GC}, for all $c <\frac12$, $I_c:=\E\,[\exp(c |U_n|^{2})] < +\infty$. Hence, for all $\lambda < \frac{2c}{ \|D^2V  \|_{\infty} \|\sigma\|_{\infty}^2 \gamma_1}$, Jensen's inequality yields:
\begin{eqnarray}\label{INTEGRABILITY_V2}
I_2(\gamma_n,\lambda)&\leq& \Big [  \E\exp \big (c |U_n|^{2} \big ) \Big ]^{\frac{\lambda \gamma_n \|D^2V\|_{\infty} \|\sigma\|_{\infty}^2}{2 c}}=\exp \Big(\gamma_n \ln(I_{c})\frac{\lambda   \|D^2V\|_{\infty} \|\sigma\|_{\infty}^2}{2 c} \Big ).
\end{eqnarray}


These controls allow to prove the integrability statement of the proposition by induction. For $n=0$, 
recalling from assumption \A{C1} that for all $\lambda<\lambda_0, \E \exp(\lambda |X_0|^{ 2})<+\infty$ and from \textcolor{black}{\A{$\mathbf {{\mathcal L}_V}$}, i)} that $V(x) \leq \bar{c}|x|^2$ outside of a compact set, we derive that for all $\lambda\in (0, \frac{\lambda_0}{\bar{c}})$, there exists $C_{V,\lambda}^0\in (1,+\infty)$ such that $$\E\exp \big (\lambda V(X_0) \big ) \leq C_{V,\lambda}^0.$$

Set now  $\widetilde{\beta}_V := \widetilde{c}+ \ln(I_c) \frac{\|D^2V\|_{\infty} \|\sigma\|_{\infty}^2}{2c}$ and $\widetilde{\alpha}_V:=\min\big( \frac{1}{\gamma_1}, \frac{\alpha_{_V}}2 -\lambda C_{_V} \|\sigma\|_{\infty}^2(1+\gamma_1 C_{_V}[1+\frac{\gamma_1 \|D^2V\|_\infty}{2}])\big) \in (0,\frac{1}{\gamma_1}]$,
for $\lambda <\frac{\alpha_{_V}}{2C_{_V} \|\sigma \|_{\infty}^2(1+\gamma_1 C_{_V}[1+\frac{\gamma_1 \|D^2V\|_\infty}{2}])}$.

Let us assume that 
 for all $\lambda < \lambda_V:=\min\Big(\frac{\lambda_0}{2\bar{c}}, \frac{\alpha_{_V}}{2C_{_V} \|\sigma \|_{\infty}^2(1+\gamma_1 C_{_V}[1+\frac{\gamma_1 \|D^2V\|_\infty}{2}])},\frac{c}{ \|D^2V  \|_{\infty} \|\sigma\|_{\infty}^2 \gamma_1} \Big)$,  the property 
\begin{equation}
\label{HYP_REC}\tag{${\mathcal P}_{n-1}$}
\forall k\in \leftB 0,n-1\rightB,\  \E \exp \big ( \lambda V(X_{k}) \big )  \leq C_{V,\lambda}:=C_{V,\lambda}^0\vee \exp\Big( \frac{\lambda \widetilde \beta_V}{\widetilde \alpha_{_V}}\Big),
\end{equation}
holds for a fixed $n-1\in \N_0 $ and let us prove $({\mathcal P}_{n})$.
By inequalities~\eqref{TAYLOR_V},~\eqref{INTEGRABILITY_V1} and~\eqref{INTEGRABILITY_V2} and the Cauchy-Schwarz inequality,  we derive that for all $\lambda < \lambda_V$,
\begin{eqnarray*}
 \E \exp \big(\lambda V(X_n) \big)  &=& \E \Big[ \exp \big(\lambda V(X_{n-1}) \big)    \E\big[  \exp \big( \lambda (V(X_n) - V(X_{n-1})) \big) \big | \mathcal{F}_{n-1} ]  \Big ]
\\
& \leq &\E \Big[ \exp \big(\lambda [V(X_{n-1})(1-\frac{\alpha_{_V}}2\gamma_n)+\widetilde c\gamma_n] \big) I_1(\gamma_n,2\lambda,X_{n-1} 
)^{1/2}I_2(\gamma_n,2\lambda)^{1/2}\Big]\\
&=& \exp \big(\lambda \gamma_n \widetilde{\beta}_V \big )   
\E\Big[  \exp \Big( \lambda \big(1 - \frac{\alpha_{_V}}2 \gamma_n \big) V(X_{n-1})+ \lambda^2 \gamma_n C_{_V} \|\sigma\|_{\infty}^2 V(X_{n-1}+\gamma_n b_{n-1}) \Big)  \Big ].
\end{eqnarray*}
Recall now that $V(X_{n-1}+\gamma_n b_{n-1}) \le V(X_{n-1})+\gamma_n |\nabla V(X_{n-1})| |b_{n-1}|+\frac{\gamma_n^2}2 \|D^2V\|_\infty |b_{n-1}|^{ 2}\overset{\A{{\mathcal L}_{\mathbf V} }, ii)}{\le} V(X_{n-1})(1+\gamma_n C_{_V}[1+\frac{\gamma_n \|D^2V\|_\infty}{2}]) $. Thus,
\begin{eqnarray*}
 \E\big[ \exp \big(\lambda V(X_n) \big)\big]&\le & 
\exp \big(\lambda \gamma_n \widetilde{\beta}_V \big )   
\E\Big[  \exp \big( \lambda \underbrace{(1 - \gamma_n \widetilde{\alpha}_V)}_{\in [0,1)} V(X_{n-1}) \big)  \Big ]
\\
& \overset{(\rm{Jensen})}{\leq} &\exp \big(\lambda \gamma_n \widetilde{\beta}_V \big )   
\E\big[  \exp \big( \lambda V(X_{n-1}) \big)  \Big ]^{(1 - \gamma_n\widetilde{\alpha}_V)} 
\leq \exp \big(\lambda \gamma_n \widetilde{\beta}_V \big ) C_{V,\lambda}^{(1 - \gamma_n\widetilde{\alpha}_V)}
\end{eqnarray*}
using~\eqref{HYP_REC} for the last inequality. 
From the above equation and the previous definition of $C_{V,\lambda}$ we have:
\begin{eqnarray*}
\exp \big(\lambda \gamma_n \widetilde{\beta}_V \big ) C_{V,\lambda}^{(1 - \gamma_n\widetilde{\alpha}_V)}\le C_{V,\lambda} \iff C_{V,\lambda} \geq \exp\Big(\frac{\lambda \widetilde \beta_V}{\widetilde \alpha_{_V}}\Big).
\end{eqnarray*}
Hence, $({\mathcal P}_n)$ holds. Taking $c_V<\lambda_V$ completes the proof.
\hfill $\square $
\\
\begin{remark}
Noting that $v^*:= \inf_{x \in \R^d} V(x) >0$, we get that for all $(n,\xi) \in \N \times [0,1]$, and for all $\lambda < \lambda_V (v^*)^{1-\xi}$:
\begin{eqnarray*}
\E \exp ( \lambda V_n^{\xi} ) = \E \exp \Big(\lambda (v^*)^{\xi} \underbrace{\Big(\frac{V_n}{v^*}\Big)^{\xi}}_{\geq 1}  \Big) \leq \E \exp \big(\lambda (v^*)^{\xi-1} V_n  \big) \leq C_{V,\lambda (v^*)^{\xi-1}} < + \infty.
\end{eqnarray*}
Thus, we readily get as a by-product of Proposition~\ref{expV_int} that, for all $\xi\in [0,1],\lambda<\lambda_V(v^*)^{1-\xi} $ , $\sup_{n \in \N} \E \exp ( \lambda V_n^{\xi} ) < + \infty$. We refer to Lemaire (see e.g. 
Theorem 17 in \cite{lema:07}) for additional results in that direction.
\end{remark}

\noindent\textbf{Proof of Lemma~\ref{ineq_phi}.}
Recalling from~\eqref{GV} that there exists $C_{V,\varphi}>0$ such that for all $x\in \R^d,\  |\varphi(x)|\le C_{V,\varphi}\big(1+\sqrt{V(x)}\big)$, we get for $j\in \{1,2\} $:
\begin{eqnarray*}
\left[\E \exp\Big( jp \lambda_n \frac{|\varphi ( X_0) - \varphi ( X_n)| }{\Gamma_n}\Big) \right]^{\frac{1}{jp}}
\leq \left[  \E \exp\Big( jp \lambda_n  \frac{ C_{V,\varphi}(2+\sqrt{V(X_0)}+\sqrt{V(X_n)}) }{ \Gamma_n}\Big) \right]^{\frac{1}{jp}}\hskip 3 cm 
\\
\hskip 3 cm \leq \exp\Big(2C_{V,\varphi}\frac{\lambda_n}{\Gamma_n} \Big) \left[ \E \exp\Big( 2j p\, C_{V,\varphi} \lambda_n  \frac{\sqrt{V( X_0)} }{ \Gamma_n}\Big) \right]^{\frac{1}{2jp}}
\left[ \E \exp\Big( 2jp\, C_{V,\varphi}\lambda_n  \frac{ \sqrt{V(X_n)}  }{ \Gamma_n}\Big) \right]^{\frac{1}{2jp}}.
\end{eqnarray*}
Write now for $i\in \{ 0,n\} $ by the Young inequality:
\begin{eqnarray*}
2j p C_{V,\varphi} \lambda_n  \frac{ \sqrt{V( X_i)} }{ \Gamma_n} &\leq& c_V V(X_i) + \frac{(jp)^2 C_{V,\varphi}^2 \lambda_n^2}{c_V \Gamma_n^2},
\end{eqnarray*}
where $c_V$ is the positive real constant such that $I_V^1=\displaystyle \sup_{n\ge 0}\E\,[\exp(c_VV(X_n))]<+\infty $ (see Proposition~\ref{expV_int}). We then get
\begin{eqnarray*}
\left[ \E \exp\Big( jp  \lambda_n \frac{|\varphi ( X_0) - \varphi ( X_n)|}{\Gamma_n}\Big) \right]^{\frac{1}{jp}}
\!\!&\!\!\leq\!\!&\!\!
\exp\Big(2C_{V,\varphi}\frac{\lambda_n}{\Gamma_n} \Big) \exp\Big(\frac{j p C_{V,\varphi}^2 \lambda_n^2}{ c_V \Gamma_n^2}\Big)
\Big(\E \exp( c_V V( X_0)) \Big)^{\frac{1}{2jp}}
\Big(\E \exp(c_V V(X_n)) \Big)^{\frac{1}{2jp}}
\\
\!\!&\!\!\leq\!\!&\!\!
\exp\left(\frac{(j+1) p C_{V,\varphi}^2 \lambda_n^2}{ c_V \Gamma_n^2}\right)\exp\left( \frac{c_V}{p}\right) (I_V^1)^{\frac{1}{jp}}.\hskip 5cm \Box
\end{eqnarray*}

\noindent\hspace{-.2cm}\textbf{Proof of Lemma~\ref{ineq_rest}.}
\begin{trivlist}
\item[$\bullet $] \textit{Proof of inequalities~\eqref{ineq_rest_b2_phi} and~\eqref{ineq_rest_b2}.}
\begin{trivlist}
\item[-] \textit{If \textcolor{black}{$x\mapsto \langle \nabla \varphi(x),b(x)\rangle$ is Lipschitz continuous}.}
We first rewrite from the definition of $D_{2,b,n} $ in~\eqref{DECOUP}:
\begin{eqnarray*}
D_{2,b,n}&=&\sum_{k=1}^n \textcolor{black}{\gamma_k \int_0^1 \langle \nabla  \varphi (X_{k-1} + t \gamma_k b_{k-1})-\nabla \varphi(X_{k-1}),b_{k-1}\rangle dt}\\
&=&\sum_{k=1}^n \gamma_k \Big[\int_0^1 \langle \nabla  \varphi (X_{k-1} + t \gamma_k b_{k-1}),b_{k-1} -b(X_{k-1} + t \gamma_k b_{k-1})\rangle dt\\
&&+\int_0^1 \big(\langle \nabla \varphi,b\rangle(X_{k-1} + t \gamma_k b_{k-1}) -\langle\nabla \varphi ,b \rangle(X_{k-1})\big) dt \Big].
\end{eqnarray*}
From the boundedness of $\nabla \varphi $, and the Lipschitz property of the mappings $ x\mapsto b(x)$ (which has been assumed from the very beginning) and $ x\mapsto \langle \nabla \varphi(x),b(x)\rangle$ (assumed for the current inequality), recalling that $b_{k-1}=b(X_{k-1}) $, one derives that :
\begin{eqnarray}
\label{NEW_BD_D2BN}|D_{2,b,n}|&\le &  \sum_{k=1}^n \gamma_k^2 \Big(\|\nabla \varphi\|_\infty [b]_1 +[\langle \nabla \varphi, b\rangle]_1 \Big) \frac{|b_{k-1}|}2\le C \sum_{k=1}^n \gamma_k^2 |b_{k-1}|,\ C:=C(b,\varphi).
\end{eqnarray}

From this inequality, assumption \A{${\mathbf{ {\mathcal L}_V}} $}, ii) and the Jensen inequality (applied to the exponential function for the measure $\frac{1}{\Gamma_n^{(2)}} \sum_{k=1}^{n} \gamma_k^2 \delta_{k} $), we derive:
\begin{eqnarray*}
\left(\E\exp\Big( \frac{4 p \lambda_n}{ \Gamma_n}|D_{2,b,n}|\Big)\right)^{\frac 1{4p}} 
\!\!&\!\!\leq\!\!&\!\! \left( \frac{1}{\Gamma_n^{(2)}} \sum_{k=1}^n  \gamma_k^{2} \E \exp\Big(\frac{4p  \lambda_n \Gamma_n^{(2)} }{\Gamma_n}  C\sqrt{C_{_V}}\sqrt{V_{k-1}} \Big) \right)^{\frac{1}{4p}}.
\end{eqnarray*}
From the Young inequality we obtain:
\begin{eqnarray*}
\E\exp\Big(\frac{4p  \lambda_n\Gamma_n^{(2)} }{ \Gamma_n} C\sqrt{C_{_V}}  \sqrt{V_{k-1}} \Big) 
\leq \exp\Big(\Big(\frac{2 \sqrt 2p  \lambda_n \Gamma_n^{(2)} }{ \Gamma_n}  \frac{C \sqrt{C_{_V}}}{\sqrt{c_V}} \Big)^2 \Big) 
 \E\,[   \exp(c_V V_{k-1} ) ].
\end{eqnarray*}
 We finally derive with the notations of Proposition~\ref{expV_int}:
\begin{eqnarray*}
\left( \E\exp\Big( \frac{4 p \lambda_n}{ \Gamma_n}|D_{2,b,n}|\Big)\right)^{\frac 1{4p}}&\le& 
\exp\Big(\frac{2p  \lambda_n^2 (\Gamma_n^{(2)})^2 }{ \Gamma_n^2}  \frac{(C \sqrt{C_{_V}})^2}{c_V}  \Big)
 (I_V^1)^{\frac 1{4p}}
\le 
\exp\Big(C_{\ref{ineq_rest_b2_phi}} \frac{p\lambda_n^2 (\Gamma_n^{(2)})^2}{\Gamma_n^2}
\Big)(I_V^1)^{\frac 1{4p}},
\end{eqnarray*}
setting $C_{\ref{ineq_rest_b2_phi}}:=  2\frac{(C \sqrt{C_{_V}})^2}{c_V}$ with $C=\frac 12 \big(\|\nabla \varphi\|_\infty [b]_1 +[\langle \nabla \varphi, b\rangle]_1 \big) $ as in~\eqref{NEW_BD_D2BN}. 

\hfill \textcolor{black}{$\qquad\Box$}
\smallskip
\end{trivlist}

\item[-] \textit{If $a\le \frac{c_V  q}{4C_{_V}p}\frac{\|\sigma\|_\infty^2\|\nabla \varphi\|_\infty^2}{\|D^2 \varphi\|_\infty^2} \frac{\sqrt{\Gamma_n}}{\Gamma_n^{(2)}}=\chi_n \frac{\sqrt{\Gamma_n}}{\Gamma_n^{(2)}}$ with the notation introduced in~\eqref{DEF_CHI_N}}. 
\textcolor{black}{Write first from~\eqref{DECOUP} (definition of $ D_{2,b,n}$), using a Taylor expansion on $\nabla \varphi $:}
\begin{eqnarray}
\left( \E\exp\Big(\frac{4p\lambda_n}{\Gamma_n} |D_{2,b,n}|\Big)\right)^{\frac 1{4p}}
\le 
\left(\E \exp\Big( \frac{4 p \lambda_n}{ \Gamma_n} \sum_{k=1}^n  \gamma_k^2 \int_0^1 (1-t) \Big|\Tr\Big(   
D^2 \varphi (X_{k-1} + t \gamma_k b_{k-1} ) 
b_{k-1}\otimes b_{k-1}\Big)\Big|dt\Big) \right)^{\frac 1{4p}}.\notag\\
\label{PRELIM_LEMME_5}
\end{eqnarray}

We first easily get from the assumptions on $\varphi$ and point ii) of \A{$\mathbf{{\mathcal L}}_V $} that:
\begin{eqnarray*}
\left( \E\exp\Big(\frac{4p\lambda_n}{\Gamma_n} |D_{2,b,n}|\Big)\right)^{\frac 1{4p}}\leq  \left(\E \exp\Big(\frac{2p \lambda_n}{ \Gamma_n} \sum_{k=1}^n  \gamma_k^2 C_{_V} V_{k-1} \| D^2\varphi \|_{\infty} \Big) \right)^{\frac{1}{4p}}.
\end{eqnarray*}
 From the Jensen inequality,
 we derive:
 \begin{eqnarray*}
 \left(\E\exp\Big(\frac{4p\lambda_n}{\Gamma_n} |D_{2,b,n}|\Big)\right)^{\frac1{4p}}
&\leq & \left(\frac{1}{\Gamma_n^{(2)}} \sum_{k=1}^n  \gamma_k^{2} \E \exp\Big(\frac{2p \lambda_n \Gamma_n^{(2)}}{\Gamma_n}  \| D^2\varphi \|_{\infty} C_{_V}   
V_{k-1}  \Big) \right)^{\frac{1}{4p}}.
\end{eqnarray*}

 We then have \textcolor{black}{from the definition o\textcolor{black}{f} $\lambda_n $ in~\eqref{DEF_LAMBDA}} that:
 $$\bar v_n:=\frac{2p \lambda_n \Gamma_n^{(2)}}{\Gamma_n}  \| D^2\varphi \|_{\infty} \frac{C_{_V}}{c_V}=\frac{\Gamma_n^{(2)}}{\sqrt{\Gamma_n}}\frac{2 C_{_V} p}{      c_V  q} \frac{\| D^2\varphi \|_{\infty}}{\|\sigma\|_\infty^2\|\nabla \varphi\|  _{\infty}^2}  \, a  \le 1.$$ 
The Jensen inequality for concave functions yields for all $k\in \leftB 1,n\rightB $:
 \begin{eqnarray*}
 \E \exp\Big(\frac{2p \lambda_n \Gamma_n^{(2)}}{\Gamma_n}  \| D^2\varphi \|_{\infty} C_{_V}   
V_{k-1}  \Big)=\E\exp\Big(\bar v_n c_V V_{k-1}\Big)\le \left(\E\exp\Big( c_VV_{k-1}\Big)\right)^{\bar v_n}. 
 \end{eqnarray*}
Thus,  setting 
 \begin{equation}
 \label{DEF_VN}
 v_n:=\frac{ \bar v_n}{4p}=\frac{ \lambda_n \Gamma_n^{(2)}}{\textcolor{black}{2}\Gamma_n}  \| D^2\varphi \|_{\infty} \frac{C_{_V}}{c_V},
 \end{equation}
 we finally derive,
\begin{eqnarray*}
\left[\E\exp\Big( \frac{4 p \lambda_n}{ \Gamma_n} |D_{2,b,n}| \Big)\right]^{\frac 1{4p}}
\leq  \left[ \frac{1}{\Gamma_n^{(2)}} \sum_{k=1}^n  \gamma_k^{2} \Big(\sup_{l \geq 1} \E\,\big[ \exp(c_V  V_{l-1}  )\big ]\Big)^{\bar v_n}\right]^{\frac 1{4p}}=(I_{V}^1)^{v_n}=: C_n,
\end{eqnarray*}
using again the notations of Proposition~\ref{expV_int}. This gives~\eqref{ineq_rest_b2}.

\item[$\bullet$ ]\textit{Proof of inequality~\eqref{ineq_rest_b}.}
We proceed as for the proof of~\eqref{ineq_rest_b2} and~\eqref{ineq_rest_b2_phi}. Write:
\begin{eqnarray*}
&&\left( \E\exp\Big(\frac{4p\lambda_n}{\Gamma_n}|D_{2,\Sigma,n}|\Big)\right)^{\frac 1{4p}} \\
&&\hskip 2cm \le \left( \E  \exp\Big( \frac{4p \lambda_n}{ \Gamma_n}\sum_{k=1}^n \frac{\gamma_k}2  \Big|\Tr\Big(  \big(D^2\varphi( X_{k-1}+  \gamma_k b_{k-1} ) - D^2\varphi( X_{k-1}) \big)\Sigma_{k-1}\Big)\Big| \Big) \right)^{\frac{1}{4p}}\\
&&\hskip 2cm \leq  \left( \E  \exp\Big( \frac{2p \lambda_n}{ \Gamma_n}  \|\sigma\|_\infty^2 [\varphi^{(2)}]_1\sum_{k=1}^n \gamma_k^2    |b_{k-1}|  \Big) \right)^{\frac{1}{4p}} 
\leq   \left( \E \exp\Big( \frac{2p \lambda_n}{\Gamma_n} \|\sigma\|_\infty^2  [\varphi^{(2)}]_1 C_{_V}^{\frac 12}  \sum_{k=1}^n \gamma_k^2 |V_{k-1}|^{\frac 12} \Big) \right)^{\frac{1}{4p}}
\\
&& \hskip 2cm\leq \left( \frac{1}{\Gamma_n^{(2)}} \sum_{k=1}^n \gamma_k^{2}  \E \exp\Big( \frac{2p\lambda_n  \Gamma_n^{(2)}}{ \Gamma_n}  \|\sigma\|_\infty^2  [\varphi^{(2)}]_1 C_{_V}^{\frac 12}|V_{k-1}|^{\frac 12} \Big) \right)^{\frac{1}{4p}}.
\end{eqnarray*}
Using once again the Young inequality and setting $C_{\ref{ineq_rest_b}}:=\frac{\|\sigma\|_\infty^4[\varphi^{(2)}]_1^2}{4}\frac{C_{_V}}{c_V} $, we obtain:
\begin{eqnarray*}
\left( \E\exp\Big(\frac{4p\lambda_n 
}{\Gamma_n}|D_{2,\Sigma,n}|\Big)\right)^{\frac{1}{4p}}
 \leq 
 \exp\Big(\frac{p\lambda_n^2}{4}\Big(\frac{\Gamma_n^{(2)}}{\Gamma_n}\Big)^2\|\sigma\|_\infty^4[\varphi^{(2)}]_1^2\frac{C_{_V}}{c_V}\Big)(I_V^1)^{\frac1{4p}}
 \le 
 \exp\Big(C_{\ref{ineq_rest_b}}p\lambda_n^2\Big(\frac{\Gamma_n^{(2)}}{\Gamma_n}\Big)^2\Big)(I_V^1)^{\frac1{4p}}.\qquad\Box
\end{eqnarray*}
\end{trivlist}

\mysection{A refinement when $\vvvert \sigma\vvvert^2-\nu(\vvvert\sigma\vvvert^2) $ is a  Coboundary}
\label{SEC_COBORD}

We will assume in this section that there exists a solution $\vartheta $ of the \textcolor{black}{Poisson problem 
${\mathcal A}\vartheta=\vvvert\sigma\vvvert^2-\nu(\vvvert\sigma\vvvert^2) $, where $\vvvert\cdot\vvvert $ is a matrix norm such that $\|\cdot\|\le \vvvert \cdot\vvvert $}, satisfying the \textcolor{black}{assumptions stated for $\varphi $ in Theorem~\ref{ineq_fin}. This is in particular the case for the Fr\"obenius norm $\|\cdot\|_F $ under the assumptions of the previous Theorem~\ref{Poisson_regular}.} 

In this special case, we have a slightly different concentration result improving our previous ones for a certain deviation range.
\begin{theo}\label{THM_COBORD}
Under the assumptions of Theorem~\ref{ineq_fin} and with the notations introduced therein, we have that:
\begin{trivlist}
\item[(a)] For 
($\beta\in (0,1] $ and $ \theta \in (\frac{1}{2+\beta}, 1 ]$), there exist two explicit monotonic sequences 
$\tilde c_n\le 1\le \tilde C_n,\ n\ge 1$,   with $\lim_n \tilde C_n = \lim_n \tilde c_n =1$
such that for all $n\ge 1$ for all $a>0$:
 \begin{eqnarray*}
\P\big[ |\sqrt{\Gamma_n}\nu_n( \mathcal{A} \varphi )| \geq a \big] 
   &\leq& 2\, \tilde C_n \exp\left(\! - \frac{\tilde c_n}{2\nu(\vvvert\sigma\vvvert^2)\|\nabla \varphi\|_\infty^2}\, \Phi_n(a)
\right),\\
\Phi_n(a)&:=&\left[ \!\Bigg(\! a^2\Big( 1-\frac{2}{1+\sqrt{1+4\,\bar c_n^3\,\frac{\Gamma_n}{a^2}}}\Big)  \Bigg)\! \vee \! \Bigg(\!a^{\frac 43}\Gamma_n^{\frac{1}{3}}\bar c_n\left(1-\frac 23\bar c_n\left(\frac{\Gamma_n }{a^{2}} \right)^{\frac 13} \right)_+ \Bigg) \right],
\end{eqnarray*}
where $x_+=\max(x,0)$ and  \textcolor{black}{$ \bar c_n:=\left(\frac{[\varphi]_1}{[\vartheta]_1} \right)^{2/3}\nu(\vvvert\sigma\vvvert^2)\vvvert\sigma\vvvert_\infty^{-2/3}\check c_n$ with $\check c_n $ being an explicit \textcolor{black}{positive} sequence s.t. $\check c_n \downarrow_n 1 $}.

\smallskip
\item[(b)] For $\beta\in (0,1], \theta = \frac{1}{2+\beta}$,
 there exist two explicit monotonic sequences $\tilde c_n\le 1\le \tilde C_n,\ n\ge 1$,   with $\lim_n \tilde C_n = \lim_n \tilde c_n =1$ such that for all $n\ge 1$ for all $a>0$:
 \begin{eqnarray*}
\P\big[ |\sqrt{\Gamma_n}\nu_n( \mathcal{A} \varphi )+{\mathcal B}_{n,\beta}  | \geq a\big]
   \leq 2\, \widetilde C_n \exp\left(\! - \frac{\widetilde c_n}{2\nu(\vvvert\sigma\vvvert^2)\|\nabla \varphi\|_\infty^2}\! \Phi_n(a)
\right). 
\end{eqnarray*}
\end{trivlist}

\end{theo}

\begin{remark}[About deviation rates]
Observe that in order to derive \textbf{global} deviation bounds (valid for every $a>0$) two concentration regimes appear in the previous bounds.
For an arbitrary fixed $a>0$, we have that for $n$ large enough (depending on $a$), the Gaussian concentration regime will give the fastest decay, since $\frac{2}{1+\sqrt{1+4\bar c_n^3\frac{\Gamma_n}{a^2}}}\underset{n}{\rightarrow} 0 $. Also, when $a \asymp \sqrt{\Gamma_n} $ the two above contributions give a Gaussian bound, with suboptimal constants. Eventually, when $a\gg\sqrt{\Gamma_n} $, for a fixed $n$, we have that the first term is ``stuck" at the threshold $\Gamma_n $ whatever level $a$ is considered, {\it i.e.} 
$a^2\big(1- \frac{2}{1+\sqrt{1+4\bar c_n^3\frac{\Gamma_n}{a^2}}}\big)\underset{a\rightarrow \infty}{\longrightarrow} \bar c_n^3 \Gamma_n$ whereas the second clearly becomes bigger. 

To summarize, when the Gaussian regime prevails (i.e. when $\frac{a}{\sqrt{\Gamma_n}}$ is small), the results of Theorem~\ref{ineq_fin} have been improved in the sense that the variance in the deviations is a \textit{sharper} upper bound of the ``\textit{carr\'e du champ}" $\int_{\R^d} |\sigma^*\nabla \varphi(x)|^2\nu(dx) $ appearing in the asymptotic Theorem~\ref{theo}. Indeed, we managed to replace the supremum norm $\vvvert \sigma\vvvert_\infty^2 $ \textcolor{black}{deriving from Theorem~\ref{ineq_fin} and the domination condition on the matrix norms} by $ \nu(\vvvert \sigma\vvvert^2)$. However, our martingale approach naturally leads to a bound in $\|\nabla \varphi\|_\infty^2 $.

On the other hand, the global double regime seems to be the price to pay to benefit from the better approximation of the ``\textit{carr\'e du champ}" in the Gaussian regime.

\textcolor{black}{Eventually, Theorem~\ref{THM_COBORD_SHORT} is a direct consequence of the previous theorem in the Gaussian regime.}
\end{remark}


\begin{proof} We focus on case \textit{(a)} for $\beta\in (0,1) $, $\theta\in (1/(2+\beta),1] $. 
 Case \textit{(b)} could be derived similarly following the proof of Theorem~\ref{ineq_fin}.
We restart from the computations of Section~\ref{SEC_STRAT} that give for all $\lambda>0 $ the control in equation~\eqref{DECOUP_TCHEB}. Let us now deal with  the term giving the concentration and write for all $ \rho>1$:
\begin{eqnarray}
\E \exp \Big (-\frac{q\lambda}{\Gamma_n}M_n\Big)
&\le& 
\left( \E\exp \Big (- \rho \frac{q\lambda}{\Gamma_n}M_n-\frac{\rho^2 (q\lambda)^2[\varphi]_1^2}{2\Gamma_n^2}\sum_{k=1}^n \gamma_k{\mathcal A}\vartheta(X_{k-1}) \Big)\right)^{\frac 1\rho}\nonumber\\
&&\times \left( \E\exp \Big (
\frac{\rho^2 (q\lambda)^2[\varphi]_1^2}{2(\rho-1)\Gamma_n^2}\sum_{k=1}^n \gamma_k{\mathcal A}\vartheta(X_{k-1}) \Big)\right)^{1-\frac 1\rho}
=:{\mathscr T}_1^{\frac 1\rho} {\mathscr T}_2^{1-\frac 1\rho}.\label{CTR_COBORD}
\end{eqnarray}
Since for all $x\in \R^d $, ${\mathcal A}\vartheta(x)=\vvvert\sigma(x)\vvvert^2-\nu(\vvvert \sigma\vvvert^2) $, we obtain:
\begin{eqnarray*}
{\mathscr T}_1&=&\exp\Big(\frac{\rho^2 (q\lambda)^2[\varphi]_1^2\nu(\vvvert\sigma\vvvert^2)}{2\Gamma_n} \Big)\E\exp \Big ( \textcolor{black}{-}\rho \frac{q\lambda}{\Gamma_n}M_n-\frac{\rho^2 (q\lambda)^2[\varphi]_1^2}{2\Gamma_n^2}\sum_{k=1}^n \gamma_k\vvvert \sigma(X_{k-1})\vvvert^2 \Big).
\end{eqnarray*}
\textcolor{black}{The key idea is that we have exploited the Poisson equation  solved by $\vartheta $ to replace the previous \textit{rough} control $\exp\Big(\frac{ (q\lambda)^2[\varphi]_1^2\vvvert \sigma\vvvert_\infty^2}{2\Gamma_n} \Big) $, coming from the martingale increment obtained in equation~\eqref{BLUNT_CTR} \textcolor{black}{and the domination condition on the matrix norms}, by the above  term $\exp\Big(\frac{\rho^2 (q\lambda)^2[\varphi]_1^2\nu(\vvvert \sigma\vvvert^2)}{2\Gamma_n} \Big)$. This last contribution will be part of  the optimization procedure over $\lambda $. This improvement will be all the more  significant that neighborhoods of the points where  the norm of the diffusion coefficient $\sigma$ attains its supremum are not \textit{very much} charged by the invariant distribution. The point for ${\mathscr T}_1$ is then to prove that the remaining expectation is less than $1$. It will be  shown by exhibiting an appropriate underlying supermartingale.}

Set to this end ${\widetilde {\mathscr T}}_1:=\exp\Big(-\frac{\rho^2 (q\lambda)^2[\varphi]_1^2\nu(\vvvert\sigma\vvvert^2)}{2\Gamma_n} \Big) {\mathscr T}_1 $. 
Define now, for a given $n\in \N$ and $m\in \N_0,\ S_m:=\exp \Big ( \textcolor{black}{-}\rho \frac{q\lambda}{\Gamma_n}M_m-\frac{\rho^2 (q\lambda)^2[\varphi]_1^2}{2\Gamma_n^2}\sum_{k=1}^m \gamma_k\vvvert\sigma(X_{k-1})\vvvert^2 \Big)$. From the definition of the martingale $(M_k)_{k\ge 1}$ in~\eqref{INTRO_MART} and the controls of the Lipschitz constants of the functions $\big(\psi_{k}(X_{k-1},\cdot)\big)_{k\in \leftB 1,n\rightB} $ in Lemma~\ref{decomp_nu}, we 
get by iterated conditioning:
\begin{eqnarray*}
\widetilde {\mathscr T}_1
&\le& 
\E\Big[S_{n-1}\exp \Big (-\frac{\rho^2 (q\lambda)^2[\varphi]_1^2}{2\Gamma_n^2} \gamma_n\vvvert\sigma(X_{n-1})\vvvert^2\Big)  
\E\Big[\exp \Big (\textcolor{black}{-} \rho \frac{q\lambda}{\Gamma_n}(M_n-M_{n-1})\Big) \Big|\F_{n-1}\Big] \Big]\\
&\underset{\A{GC}}{\le}&
\E\Big[S_{n-1}\exp \Big ( -\frac{\rho^2 (q\lambda)^2[\varphi]_1^2}{2\Gamma_n^2} \gamma_n\vvvert\sigma(X_{n-1})\vvvert^2\Big) 
\exp(\frac{\rho^2 (q\lambda)^2}{2\Gamma_n^2}\gamma_n[\varphi]_1^2\vvvert\sigma(X_{n-1})\vvvert^2)\Big]
\le 
\E\,[S_{n-1}]
\le 1.
\end{eqnarray*}
In other words, 
$(S_m)_{m\ge 0} $ is a positive supermartingale. We finally get that, for all $\rho>1 $:
\begin{equation}
\label{CTR_T1}
{\mathscr T}_1^{\frac1\rho}\le \exp\Big(\frac{\rho (q\lambda)^2[\varphi]_1^2\nu(\vvvert\sigma\vvvert^2)}{2\Gamma_n} \Big). 
\end{equation}
For the term ${\mathscr T}_2$, we have that setting $\mu:=\mu(q,n,\rho,\lambda)=\frac{(q\lambda)^2\rho^2[\varphi]_1^2}{2(\rho-1)\Gamma_n} $,
\begin{eqnarray*}
{\mathscr T}_2=\E\exp\Big( \frac{\mu}{\Gamma_n}\sum_{k=1}^n \gamma_k {\mathcal A}\vartheta(X_{k-1}) \Big),
\end{eqnarray*}
so that this contribution can be controlled from the previous expansion of Lemma~\ref{decomp_nu} exploiting the technical lemmas of Section~\ref{SEC_STRAT} replacing $\lambda  $ by  $\mu$ and $\varphi $ by $\vartheta $. 

In case (\textit{a}), for $\theta\in (1/(2+\beta),1] $, $\beta\in (0,1] $, the H\"older inequalities yield that  for all $\mu \in \R_ +$ and all $\bar p,\bar q \in (1 , + \infty ) $, $ \frac{1}{\bar p} + \frac{1}{\bar q} =1$, similarly to~\eqref{DECOUP_TCHEB}, 
\begin{eqnarray}
{\mathscr T}_2 = \E \exp\Big( \frac{\mu}{\Gamma_n}\sum_{k=1}^n \gamma_k {\mathcal A}\vartheta(X_{k-1})\Big)
\le
\left( \E \exp\Big( \textcolor{black}{-}\frac{\bar q \mu}{ \Gamma_n}M_n^{\vartheta}\Big)\right)^{\frac1{\bar q}}\nonumber\\
\times \left(\E \exp\Big( \frac{2\bar p \mu}{ \Gamma_n}|L_n^\vartheta|\Big)\right)^{\frac 1{2p}}\left(\E\exp\Big( \frac{4\bar p \mu}{ \Gamma_n}|D_{2,b,n}^\vartheta|\Big)\right)^{\frac 1{4\bar p}}\left(\E\exp\Big( \frac{4\bar p \mu}{ \Gamma_n}|D_{2,\Sigma,n}^\vartheta|\Big)\right)^{\frac1{4\bar p}}, \label{DECOUP_TCHEB_COB}
\end{eqnarray}
where the superscripts in $\vartheta $ emphasize that the contributions to be analyzed are those associated with  the solution $\vartheta $ of 
the Poisson problem with source $\vvvert\sigma\vvvert^2-\nu(\vvvert\sigma\vvvert^2) $.

Still for simplicity, we assume as well  (case \textit{(i)}) that \textcolor{black}{the mapping $x\mapsto \langle b(x), \nabla \vartheta(x) \rangle $ is Lipschitz continuous}. 
 Plugging in~\eqref{DECOUP_TCHEB_COB} the controls established in   Lemma~\ref{ineq_phi} (with $j=2$), Lemma~\ref{ineq_rest} (equations~\eqref{ineq_rest_b}  and~\eqref{ineq_rest_b2_phi}) and~\eqref{CTR_MART}, then replacing $\lambda_n $ by $\mu $,  \textcolor{black}{we get similarly to the first inequality of~\eqref{BD_P_THM_2} and with the notations of Lemma~\ref{CTR_BIAS_1}}:
 \begin{eqnarray*}
{\mathscr T}_2 &\leq& \exp\Big(\frac{\bar q\mu^2\vvvert\sigma\vvvert_\infty^2[\vartheta]_1^2}{2\Gamma_n} \Big)
 \exp\left( \frac{\mu^2}{2\Gamma_n \bar p}+\frac{\bar pa_n^2}{2}\right)  \exp\left(\frac{3 \bar p C_{V,\vartheta}^2 \mu^2}{ c_V \Gamma_n^2}+\frac{c_V}{\bar p}\right)(I_V^1)^{\frac 1{2\bar p}}\nonumber\\
&&\times
\exp\left ( C_{\ref{ineq_rest_b}}\frac{\bar p \mu^2  (\Gamma_n^{(2)})^2}{  \Gamma_n^2}\right) (I_V^1)^{\frac 1{4\bar p}}\times \exp\left(C_{\ref{ineq_rest_b2_phi}}\Big( \frac{3\bar p\mu^2 (\Gamma_n^{(2)})^2}{2\Gamma_n^2}+\frac{1}{2\bar p}\Big)\right)(I_V^1)^{\frac1{4\bar p}}.\\
&\le& \exp\Big(\frac{\mu^2}{\Gamma_n}\Big( \frac{\bar q\vvvert\sigma\vvvert_\infty^2[\vartheta]_1^2}{2}+\bar p\Big(\frac{(\Gamma_n^{(2)})^2}{\Gamma_n}[C_{\ref{ineq_rest_b}}+\frac{3}2C_{\ref{ineq_rest_b2_phi}}] +\frac{3C_{V,\vartheta}^2}{c_V\Gamma_n}\Big)+\frac 1{2\bar p}\Big) \Big)\exp\Big(\frac 1{\bar p}\big(c_V+\frac {C_{\ref{ineq_rest_b2_phi}}}2\big)+\frac{\bar p a_n^2}2 \Big)(I_V^1)^{\frac1{\bar p}}.
\end{eqnarray*}
 Set now 
 \begin{eqnarray}
 \bar C_n&:=&\exp\Big(\frac 1{\bar p}\big(c_V+\frac{C_{\ref{ineq_rest_b2_phi}}}2\big)+\frac{\bar p a_n^2}2 \Big)(I_V^1)^{\frac1{\bar p}},\nonumber\\
 \bar e_n&:=&\bar p\Big(\frac{(\Gamma_n^{(2)})^2}{\Gamma_n}[C_{\ref{ineq_rest_b}}+\frac{3}2C_{\ref{ineq_rest_b2_phi}}] +\frac{3C_{V,\vartheta}^2}{c_V\Gamma_n}\Big)+\frac{1}{2\bar p}.\label{DEF_en}
 \end{eqnarray}
 In the considered case, the exponent $\bar p:=\bar p_n $ can again be taken such that $\bar p_n\underset{n}{\rightarrow}+\infty $ and $\bar p_n \frac{(\Gamma_n^{(2)})^2}{\Gamma_n}\underset{n}{\rightarrow} 0$ 
in order to have, $\bar e_n\underset{n}{\rightarrow}0, \bar C_n \underset{n}{\rightarrow}1$ with the indicated monotonicity for large enough  $n$.

 We derive from the above control and~\eqref{CTR_T1} that for all $q,\rho>1$:
 \begin{eqnarray*}
\left(\E\exp\Big(\textcolor{black}{-}\frac{\lambda q}{\Gamma_n}M_n\Big)\right)^{\frac 1q}\le \Big(\textcolor{black}{ {\mathscr T}_1^{\frac 1{\rho}}{\mathscr T}_2^{1-\frac1{\rho}}}\Big)^{\frac 1q}\le \exp\Big(\frac{\rho q\lambda^2[\varphi]_1^2\nu(\vvvert\sigma\vvvert^2)}{2\Gamma_n} \Big)\bar C_n^{\frac{\rho-1}{\rho q}}\exp\Big(\frac{\rho-1}{\rho q}\frac{\mu^2}{\Gamma_n}\Big( \frac{\bar q\vvvert\sigma\vvvert_\infty^2[\vartheta]_1^2}{2}+\bar e_n\Big)\Big).
 \end{eqnarray*}
Plugging this bound in~\eqref{DECOUP_TCHEB}, using again the controls of Lemmas~\ref{ineq_phi} and~\ref{ineq_rest}, eventually yields:
\begin{eqnarray*}
\P\left[ \sqrt{\Gamma_n }\nu_n( \mathcal{A} \varphi ) \geq  a\right] \leq  
\exp\left( - \frac{ a\lambda}{\sqrt{\Gamma_n}}\right)  
\exp\Big(\frac{\lambda^2}{2\Gamma_n}\big(\rho q[\varphi]_1^2\nu(\vvvert\sigma\vvvert^2)+\frac 1p \big) \Big)\bar C_n^{\frac{\rho-1}{\rho q}}\exp\Big(\frac{\rho-1}{\rho q}\frac{\mu^2}{\Gamma_n}\Big( \frac{\bar q\vvvert\sigma\vvvert_\infty^2[\vartheta]_1^2}{2}+\bar e_n\Big)\Big)\\
\times \exp\Big(\frac{\lambda^2}{\Gamma_n}  p\Big(\frac{(\Gamma_n^{(2)})^2}{\Gamma_n}\big(C_{\ref{ineq_rest_b}}+
\frac 32 C_{\ref{ineq_rest_b2_phi}} \big) +\frac{3C_{V,\varphi}^2}{c_V\Gamma_n}\Big) \Big)\exp\Big(\frac 1{p}\big(c_V+\frac {C_{\ref{ineq_rest_b2_phi}}}2\big)+\frac{pa_n^2}{2} \Big)
(I_V^1)^{\frac1{p}}.
\end{eqnarray*}
Choosing $p:=p_n\underset{n}{\ra}+\infty$ and such that $p_n\frac{(\Gamma_n^{(2)})^2}{\Gamma_n}\underset{n}{\ra}0 $, we get \textcolor{black}{by a standard symmetry} and with the notations introduced in the proof of Theorem~\ref{ineq_fin}:
\begin{eqnarray*}
\P\left[ \Big|\sqrt{\Gamma_n }\nu_n( \mathcal{A} \varphi )\Big| \geq  a\right] &\leq&  
2\,C_n \bar C_n^{\frac{\rho-1}{\rho q}}\exp\left( - \frac{ a\lambda}{\sqrt{\Gamma_n}}\right) 
\exp\Big(\frac{\lambda^2}{\Gamma_n}\big(\frac{\rho q[\varphi]_1^2\nu(\vvvert\sigma\vvvert^2)}{2}+e_n\big) \Big)
\\
&&\times \exp\Big(\frac{\rho-1}{\rho q}\frac{\mu^2}{\Gamma_n}\Big( \frac{\bar q\vvvert\sigma\vvvert_\infty^2[\vartheta]_1^2}{2}+\bar e_n\Big)\Big),
\end{eqnarray*}
where $e_n$ is defined similarly to $\bar e_n $ in~\eqref{DEF_en} replacing $\bar p$ by $p$. In particular $e_n\underset{n}{\ra}0 $. Note that for the previous choices of $p,\bar p$, we have that $\widetilde C_n:=C_n \bar C_n^{\frac{\rho-1}{\rho q}} \underset{n}{\ra} 1 $ uniformly in $\rho>1$. Recalling that $\mu=\frac{(q\lambda)^2\rho^2[\varphi]_1^2}{2(\rho-1)\Gamma_n}$, we are thus led to minimize the polynomial function
\[
P:\lambda\longmapsto  - \frac{ a\lambda}{\sqrt{\Gamma_n}}+\frac{\lambda^2}{\Gamma_n}A_n+\frac{\lambda^4}{\Gamma_n^3}B_n,
 \]
where $A_n =  A_n(\rho) = \rho \widetilde A_n$ and $ B_n=B_n(\rho)=\frac{\rho^3}{\rho-1}\widetilde B_n$  with  
\begin{equation}\label{eq:ABtilde}
\widetilde A_n:= \frac{ q[\varphi]_1^2\nu(\vvvert\sigma\vvvert^2)}{2}+e_n
 \quad\mbox{ and }\quad 
\widetilde B_n 
:=\frac{q^3[\varphi]_1^4}{4}\Big( \frac{\bar q\vvvert\sigma\vvvert_\infty^2[\vartheta]_1^2}{2}+\bar e_n\Big).
\end{equation}
Note that both sequences  $(\widetilde A_n)_{n\ge 1} $ and $(\widetilde B_n)_{n\ge 1}$ are  bounded and bounded away from zero sequences (and do not depend on $\rho$). The function $P$  is clearly convex and coercive so it attains its minimum at  $\lambda_{\min} $, unique zero of the equation $P'(\lambda_{\min})=0$. This equation  reads 
\begin{equation}\label{eq:derivee}
\lambda^3+\frac{A_n\Gamma^2_n }{2B_n } \lambda  -  \frac{a\Gamma_n^{\frac 52}}{4B_n}=0
\end{equation}
which is the canonical form of this third degree equation to apply the Cardan-Tartaglia formula~(\footnote{If the equation $z^3+pz+q=0$ has a unique real zero $z_*$  then 
its discriminant $\Delta= 4p^3+27q^2>0$ and  $z_*= \Big(\frac 12 \big(-q +\sqrt{\frac{\Delta}{27}}\big)\Big)^{\frac 13} +\Big(\frac 12 \big(-q -\sqrt{\frac{\Delta}{27}}\big)\Big)^{\frac 13}$.}) so that
\begin{equation}\label{DEF_LAMBDA_RHO}
 \lambda_{\min} (\rho) = \frac{\Gamma_n}{2}\!\left[  \left(\frac{ a}{\sqrt{\Gamma_n} B_n}+\sqrt{ \Big(\frac{2A_n}{3B_n}\Big)^3 + \frac{a^2}{\Gamma_nB^2_n} }\;\right)^{\frac 13}+  \left(\frac{ a}{\sqrt{\Gamma_n} B_n}-\sqrt{ \Big(\frac{2A_n}{3B_n}\Big)^3 + \frac{a^2}{\Gamma_nB^2_n} }\;\right)^{\frac 13}\right].
 \end{equation}
In order to derive our non-asymptotic bound, we select two ``regimes'' based on a first order expansion of $\lambda_{\min}$ in two cases $\frac {a}{B_n\sqrt{\Gamma_n}}\to 0$ and $\frac {B_n \sqrt{\Gamma_n}}{a}\to 0$, assuming that the free parameter $\rho=\rho_n$ to be specified later on remains bounded,  {\it e.g.}  $\rho \in (1,3]$ (which implies that both  quantities $\frac{A_n}{B_n}$ and $\frac{1}{B_n}$ remain bounded as well). 
Also, note that if $\rho\to 1$, then $\frac{1}{B_n}$ and $\frac{A_n}{B_n}\to 0$.
First, one easily checks that if $(x_n)_{\ge 1}$ and $(a_n)_{n\ge1}$ are two sequences of positive real numbers where $(a_n)_{n\ge1}$ is bounded, then
\begin{equation}\label{eq:minilemme}
\Big(x_n +\sqrt{a_n^3+x_n^2}\Big)^{\frac 13}+\Big(x_n -\sqrt{a_n^3+x_n^2}\Big)^{\frac 13}\sim\left\{\begin{array}{lll}
\frac 23 \frac{x_n}{a_n}& \mbox{if}& x_n = o\big(a_n^{\frac 32}\big)\;(\mbox{then }x_n\to 0),\\
(2x_n)^{\frac 13} & \mbox{if}& a_n= o\big(x_n^{\frac 2 3}\big)\;(\mbox{then  }x_n \to +\infty).
\end{array}\right.
\end{equation}

$\bullet$  If $\frac{a}{B_n\sqrt{\Gamma_n} } = o\Big(\Big(\frac{A_n}{B_n}\Big)^{\frac 32}\Big)$ (hence goes to $0$), setting then  $x_n = \frac{a}{B_n\sqrt{\Gamma_n}}$ and $a_n = \frac {2A_n}{3B_n}$ yields 
$$
\lambda_{\min}\textcolor{black}{(\rho)}\sim \lambda^*(\rho):=
\frac{a\sqrt{\Gamma_n}}{2A_n}\quad \mbox{ as } n\to +\infty.
$$
Note that $ \lambda^*:= \lambda^*(\rho)$ corresponds to the optimization of the quadratic part  of $P$. Then
\[
P(\lambda^*)=  - \frac{a^2}{ 4A_n}\Big(1- \frac{a^2}{4A_n^3}\frac{B_n}{\Gamma_n}\Big)=  - \frac{a^2}{ 4\widetilde A_n\rho }\Big(1- \frac{a^2}{4\widetilde A_n^3(\rho-1)}\frac{\widetilde B_n}{\Gamma_n}\Big). 
\]
Set now $\xi_n:=\frac{\alpha_n(a)}{\rho-1}$ with $\alpha_n(a) =\frac{\widetilde B_n}{4\widetilde A_n^3} \frac{a^2}{\Gamma_n}$.   Then 
\[
P(\lambda^*)=  -\frac{a^2}{4\widetilde A_n}\frac{1-\xi_n}{1+\frac{\alpha_n(a)}{\xi_n}} .
\]
It remains to maximize the mapping $\xi\mapsto \frac{1-\xi}{1+\alpha_n(a) \xi^{-1}} $ over $(0,1)$. Its optimum is attained for $\xi_n^*=\frac{1}{1+\sqrt{1+\frac{1}{\alpha_n(a)}}} $, which in turn yields  
\begin{equation}
\label{OPTI_QUADRA}
P(\lambda^*)=-\frac{a^2}{4\widetilde A_n}\left(1-\frac{2}{1+\sqrt{1+4\frac{\widetilde A_n^3\Gamma_n}{\widetilde B_na^2}}} \right).
\end{equation}

 Note that, with the resulting specification of $\rho=\rho_n^*:=1+\frac{\alpha_n(a)}{\xi_n^*}\!\in (1,3]$ (at least for large enough $n$),  the above condition $x_n = o\big(a_n^{\frac 32}\big)$ in~\eqref{eq:minilemme} is satisfied {\em a posteriori}.

\smallskip
$\bullet$   If $\frac{a}{B_n\sqrt{\Gamma_n} } \to +\infty$, then, still owing to~\eqref{eq:minilemme},
\[
\lambda_{\min}\textcolor{black}{(\rho)} \sim \bar \lambda^*(\rho)= \frac{\Gamma_n}{2} \Big(\frac{2a}{B_n\sqrt{\Gamma_n}}\Big)^{\frac 13} =\left(\frac{a\Gamma_n}{4B_n} \right)^{\frac 13}\sqrt{\Gamma_n} \quad \mbox{ as } n\to +\infty.
\] 
The value $ \bar \lambda^*(\rho)$ corresponds to the \textit{quartic} pseudo-optimum of $P$ ($i.e.$ obtained by neglecting the quadratic term). 
This yields, when reintroducing the parameter $\rho$, 
\[
P\big(\bar \lambda^*(\rho)\big)=-a^{\frac{4}{3}}\Gamma_n^{\frac 13} 
\frac{(\rho-1)^{\frac 13}}{\rho(4\widetilde B_n)^{\frac 13}}
\left( 
\frac 34-  \frac{\widetilde A_n}{(4\widetilde B_n)^{\frac 13}}\frac{\Gamma_n^{\frac 13} }{a^{\frac{2}{3}}} (\rho-1)^{\frac 13}\right).
\]
The right hand side of this equality is a function of $\rho\!\in (1,+\infty)$. Its analysis yields that the optimum is attained in $(1,3/2] $ and that it tends asymptotically in $n$ to $3/2 $ in our considered regime. Taking as suboptimal $\rho=3/2$ gives:
\begin{equation}
\label{OPTI_QUARTIC}
P\big(\bar \lambda^*(\rho)\big)\le -\frac{a^{\frac{4}{3}}}{4}\left(\frac{\Gamma_n}{\widetilde B_n}\right)^{\frac 13}\left( 1-\frac 23 \frac{\widetilde A_n}{\widetilde B_n}\Big(\frac{\Gamma_n}{a^2 }\Big)^{\frac 13} \right).
\end{equation}
From~\eqref{OPTI_QUARTIC},~\eqref{OPTI_QUADRA} and~\eqref{eq:ABtilde}, we conclude the proof of case \textit{(a)} by setting $\bar c_n:= \widetilde A_n\widetilde B_n^{-\frac 13}$ \textcolor{black}{which matches with the definition in the statement of the Theorem}.

\smallskip
In the biased case, the result follows similarly from the corresponding analysis performed in Section~\ref{MR} taking $A_n(\rho)= \frac{\rho q[\varphi]_1^2\nu(\vvvert\sigma\vvvert^2)}{2}$.
\end{proof}

\begin{remark}$\bullet$ \label{LA_RQ_QUI_CONCENTRE}
When $a\asymp \sqrt{\Gamma_n}$, one checks that $\lambda_{\min}\textcolor{black}{(\rho)} \asymp \Gamma_n$ and $P(\lambda_{\min}\textcolor{black}{(\rho)})\asymp - \Gamma_n$. This behavior is consistent with our non-asymptotic bound. However, 
%
for practical and numerical purposes observe that the optimum can be estimated. Namely,  plugging the identity~\eqref{eq:derivee}  satisfied by $\lambda_{\min}\textcolor{black}{(\rho)}$ in~\eqref{DEF_LAMBDA_RHO} into the definition of $P$, yields
\begin{eqnarray*}
P\big(\lambda_{\min}(\rho)\big)&=&-\frac{\lambda_{\min}(\rho)}{2\sqrt{\Gamma_n}}\left(\frac {3a}{2}-\frac{\lambda_{\min}(\rho)\rho\widetilde A_n}{\sqrt{\Gamma_n}}\right)\\
&=& \textcolor{black}{-}\frac{\sqrt{\Gamma_n}}{4}\frac{(\rho-1)^{\frac 13}}{\rho}\Phi_n(a,\rho)\left(\frac {3a}{2}-\frac{\sqrt{\Gamma_n}}{2}(\rho-1)^{\frac 13}\widetilde A_n \Phi_n(a,\rho)\right),
\end{eqnarray*}
\[
\mbox{where }\quad\Phi_n(a,\rho)= 
\Bigg( \frac{a}{\sqrt{\Gamma_n}\widetilde B_n }+\left((\rho-1)\Big(\frac{2 \widetilde A_n}{3\widetilde B_n}\Big)^3+\frac{a^2}{\widetilde B_n^2\Gamma_n}\right)^{\frac 12}\Bigg)^{\frac 13}+\Bigg( \frac{a}{\sqrt{\Gamma_n}\widetilde B_n }-\left( (\rho-1)\Big(\frac{2 \widetilde A_n}{3\widetilde B_n}\Big)^3+\frac{a^2}{\widetilde B_n^2\Gamma_n}\right)^{\frac 12}\Bigg)^{\frac 13}. 
\]
Then, an optimization in $\rho\!\in(1,+\infty) $ for given $a,\Gamma_n $ can  be performed (noting that $\rho\mapsto (\rho-1)^{i/3}\rho^{-1}$, $i\in \{ 1,2\}$ are bounded functions over $(1,+\infty)$).
\end{remark}

\mysection{Smoothness Results for the Poisson Problems (Proof of Theorem~\ref{Poisson_regular})}
\label{SEC_POISS}
\textcolor{black}{
We first prove here Theorem~\ref{Poisson_regular} 
which allows to derive from the deviation results of Theorems~\ref{ineq_fin} and~\ref{THM_COBORD_SHORT} the practical deviation bounds of Section~\ref{PRACT_DEV_BD} (i.e. Theorems~\ref{NA_CI_FIRST},~\ref{Slutsky_theorem} and~\ref{THEO_CTR_LIP}).
We recall that we work in the confluent setting of \A{D${}_\alpha^p$} and that we additionally consider two main types of assumptions:
}
\medskip
\begin{itemize}
\item[-] \textcolor{black}{Strong confluence conditions and smoothness \A{C${}_{\rm \mathbf R}$}. Namely, assumptions 
\A{$\mathbf {{\mathcal L}_V} $}, \A{D${}_\alpha^p$} and \A{R${}_{3,\beta} $} introduced in Sections~\ref{SEC_ASS} and~\ref{SEC_PRACTICAL_MAIN_RESULTS} with the condition $\|D \sigma\|_\infty^2\le  \frac{2\alpha}{
2(3+\beta) -p
}
$. 
}
\smallskip

\item[-] \textcolor{black}{Mild confluence conditions and non-degeneracy \A{C${}_{\rm \mathbf{UE}}$}.  Namely, assumptions 
\A{$\mathbf {{\mathcal L}_V} $}, \A{D${}_\alpha^p $}, \A{R${}_{1,\beta} $} and  \A{UE} introduced in Sections~\ref{SEC_ASS} and~\ref{SEC_PRACTICAL_MAIN_RESULTS} together, when $d>1$, with the condition $ \|D \sigma\|_\infty^2\le \frac{2\alpha}{
2(1+\beta) -p
} $ and the technical structure assumption on the diffusion coefficient that for all  $(i,j)\in \leftB 1,d\rightB^2 $,    $\Sigma_{i,j}(x)=\Sigma_{i,j}(x_{i\wedge j},\cdots,x_d) $.}
\end{itemize}

\textcolor{black}{
It is well known that when \A{C${}_{\rm \mathbf R}$} or \A{C${}_{\rm \mathbf{UE}}$} are in force,  there exists a unique invariant distribution for~\eqref{eq_diff}, i.e. assumption \A{U} holds. We refer to \cite{khas:80},  \cite{page:panl:14}, \cite{page:01}, \cite{pard:vere:01} for proofs of this assertion. The next step consists precisely in investigating the smoothness of the corresponding  Poisson problem as well as some associated quantitative pointwise bounds on the gradient of its solution, which is one of the key terms appearing in the deviation bounds of Theorems~\ref{ineq_fin} and~\ref{THM_COBORD_SHORT}. }

\textcolor{black}{Let us indicate that the conditions appearing in \A{C${}_{\rm \mathbf R}$} depend on pure pathwise properties, whereas the case \A{C${}_{\rm \mathbf{UE}}$} takes advantage of the regularity of the underlying semi-group which allows to alleviate some smoothness assumptions on the coefficients and some restrictions on the variations of $\sigma$. 
When the dimension increases, it becomes useful to benefit from the smoothing effects of a non-degenerate semi-group, especially if we keep in mind that one of our goals is to handle Lipschitz continuous sources.
}

\subsection{Proof of Theorem~\ref{Poisson_regular}}
\label{SEC_POISS_BIS}

Under \A{C${}_{\rm \mathbf{UE}}$} or \A{C${}_{\rm \mathbf R}$}, it is well known that the Poisson  equation~\eqref{POISSON} that we now recall:
\begin{equation*}
\forall x \in \R^d,\ {\mathcal A} \varphi(x) = f(x)-\nu(f),
\end{equation*}
admits a unique solution centered w.r.t. $\nu$ and with linear growth, in $W_{p,{\rm loc}}^2(\R^d,\R) $ for any $p>1 $ under \A{C${}_{\rm \mathbf{UE}}$} (see~\cite{pard:vere:01}), or in ${\mathcal C}^{3,\beta}(\R^d,\R)$ under \A{C${}_{\rm \mathbf R}$} (see Proposition A.8 in \cite{page:panl:14}).
In both cases, we have the following representation:
\begin{equation}
\label{REP_POISSON}
\textcolor{black}{\varphi(x)= -\int_{\R_+} \big( P_t f(x) -\nu(f)\big)dt\quad \mbox{ where }\quad  P_t f(x):=\E\,[ f( Y_t^{0,x}) ]}
\end{equation}
and  $Y_t^{0,x}$ solves~\eqref{eq_diff} with $Y^{0,x}_0=x$. To comply with the framework of the above Theorems~\ref{NA_CI_FIRST} and~\ref{Slutsky_theorem}, the first step is to establish a pointwise gradient control.


\subsubsection{Gradient Control} 
Under \A{C${}_{\rm \mathbf{UE}}$} or \A{C${}_{\rm \mathbf R}$} we manage to obtain pointwise gradient bounds for $\varphi $.
In our current confluent setting, these estimates are obtained through controls on the tangent flow, again without any \textit{a priori} uniform ellipticity condition of type \A{UE}.

\begin{lemme}[Pointwise Gradient Bounds]\label{PT_GR_BD}
Assume   that \A{C${}_{\rm \mathbf{UE}}$} or \A{C${}_{\rm \mathbf R}$} 
holds. Then
$$\|\nabla \varphi\|_\infty \le \frac{[f]_1}{\alpha},$$
with $ \alpha$ as in \A{D${}_\alpha^p $}. 
\end{lemme}

\begin{proof}

\smallskip
{\em Gradient Control in the Confluent  framework.}  Assume now that \A{D${}_\alpha^p $} holds.
Observe that, \textcolor{black}{as soon as \A{R${}_{1,\beta} $} 
holds}, it is well known that 
that $\nabla_x Y_t^{0,x} $ is well defined and belongs to $L^2(\P) $, see \cite{iked:wata:80}. Hence, for $t>0,\ i\in \leftB 1,d\rightB$:
\begin{eqnarray*}
\partial_{x_i} \E\,[ f( Y_t^{0,x}) ]=\E\,[\langle \nabla f(Y_t^{0,x}),\partial_{x_i} Y_t^{0,x}\rangle ],\ \partial_{x_i} Y_t^{0,x}=e_i+\int_{0}^t D b (Y_s^{0,x})\partial_{x_i} Y_s^{0,x} ds+\sum_{j=1}^d\int_0^t D\sigma_{\cdot j}(Y_s^{0,x})\partial_{x_i} Y_s^{0,x} dW_s^j,
\end{eqnarray*}
where $e_i $ stands for the $i^{\rm th} $ canonical vector 
and  $Db, D\sigma_{\cdot j} \in \R^d\otimes \R^d$. 

Let $\textcolor{black}{p\in (1,2]}$ be given such that \A{D${}_\alpha^p$} holds. Considering the mapping $y \in \R^d\mapsto |y|^p $, where  $|\cdot| $ stands for the Euclidean norm of $\R^d$, it is easily seen from It\^o's formula that:  
\begin{eqnarray}
|\partial_{x_i} Y_t^{0,x}|^p&=&1+p\int_0^t \Big \langle  \frac{\partial_{x_i} Y_s^{0,x}}{\
|\partial_{x_i} Y_s^{0,x}|} , D b (Y_s^{0,x})\frac{\partial_{x_i} Y_s^{0,x}}{|\partial_{x_i} Y_s^{0,x}|}\Big\rangle |\partial_{x_i} Y_s^{0,x}|^p ds\notag\\
&&+ p\sum_{j=1}^d\int_0^t \Big\langle \frac{\partial_{x_i} Y_s^{0,x}}{|\partial_{x_i} Y_s^{0,x}|},D\sigma_{\cdot j}(Y_s^{0,x})\frac{\partial_{x_i} Y_s^{0,x}}{|\partial_{x_i} Y_s^{0,x}|}\Big\rangle  |\partial_{x_i} Y_s^{0,x}|^p dW_s^j\notag\\
&&+\frac{p}{2}\sum_{j=1}^d\int_0^t 
\Big(\frac{|D\sigma_{\cdot j}(Y_s^{0,x})\partial_{x_i}Y_s^{0,x}|^2}{  |\partial_{x_i} Y_s^{0,x}|^{2}}+(p-2)\frac{|\langle \partial_{x_i} Y_s^{0,x}, D\sigma_{\cdot j}(Y_s^{0,x})\partial_{x_i} Y_s^{0,x}\rangle |^2}{ |\partial_{x_i} Y_s^{0,x}|^{4}}\Big)\notag\\
&&\times |\partial_{x_i} Y_s^{0,x}|^{p}ds\label{DEV_PUISS_P_FLOW}\\
&=&\exp\Bigg( p\int_0^t \Big\langle\frac{\partial_{x_i} Y_s^{0,x}}{|\partial_{x_i} Y_s^{0,x}|}, D b (Y_s^{0,x}) \frac{\partial_{x_i} Y_s^{0,x}}{|\partial_{x_i} Y_s^{0,x}|}\Big\rangle ds\Bigg)
\times{\mathcal E}\big(M\big)_t \notag \\
&&\times \exp\Bigg(\frac p2 \sum_{j=1}^d\int_0^t  
\Big(\frac{|D\sigma_{\cdot j}(Y_s^{0,x})\partial_{x_i}Y_s^{0,x}|^2}{  |\partial_{x_i} Y_s^{0,x}|^{2}}+(p-2)\frac{|\langle \partial_{x_i} Y_s^{0,x}, D\sigma_{\cdot j}(Y_s^{0,x})\partial_{x_i} Y_s^{0,x}\rangle |^2}{ |\partial_{x_i}Y_s^{0,x}|^{4}}\Big)ds\Bigg)\notag
\end{eqnarray}
 where $(M_t)_{t\ge 0}:=\big(p\sum_{j=1}^d\int_0^t \big\langle \frac{\partial_{x_i} Y_s^{0,x}}{|\partial_{x_i} Y_s^{0,x}|},D\sigma_{\cdot j}(Y_s^{0,x})\frac{\partial_{x_i} Y_s^{0,x}}{|\partial_{x_i} Y_s^{0,x}|} \big \rangle  dW_s^j\big)_{t\ge 0}$ is a  square integrable martingale with bounded integrand and ${\mathcal E}(M)_t:=\exp(M_t-\frac 12 \langle M\rangle_t)$ denotes the associated Dol\'eans exponential martingale. From condition \A{D${}_\alpha^p $}, we thus get:
\begin{eqnarray}
\label{DER_FLOW_EXPO_CONTROLE}
|\partial_{x_i} Y_t^{0,x}|^p&\le & \exp( -\alpha p t)\times {\mathcal E}(M_t).
\end{eqnarray}
We eventually derive:
\begin{eqnarray*}
 \int_0^{+\infty} |\E\,[\langle \nabla f(Y_t^{0,x}), \partial_{x_i}Y_t^{0,x}\rangle  ]| dt&\le &[f]_1\int_0^{+\infty}\E\,[|\partial_{x_i}Y_t^{0,x}|^p]^{1/p} dt\le [f]_1\int_0^{+\infty}\!\!\! \exp\left( -\alpha  t\right)dt = \frac{[f]_1}{\alpha}.
 \end{eqnarray*}
From the above control and equation~\eqref{REP_POISSON}, we thus derive:
\begin{equation}
\label{BD_GRAD_POISS}\forall i\in \leftB 1,d\rightB,\ \forall x\in \R^d,\ |\partial_{x_i} \varphi(x)|\le \frac{[f]_1}{\alpha}.
\end{equation}

Similarly, for all $ x\in \R^d, \nabla \varphi(x)=\int_{0}^{+\infty}\E[ (\nabla Y_t^{0,x})^* \nabla f(Y_0^{t,x})]dt$ where $\nabla Y_t^{0,x}=\left(\begin{array}{ccc} \partial_{x_1} Y_t^{0,x}& \cdots &\partial_{x_d} Y_t^{0,x}
\end{array}\right) $ so that $(\nabla Y_t^{0,x})^*=\left(\begin{array}{c} (\partial_{x_1} Y_t^{0,x})^*\\
\vdots\\
(\partial_{x_d} Y_t^{0,x})^*
\end{array}\right) $. Hence, recalling that $ |\cdot|$ stands for the Euclidean norm, $|\nabla \varphi(x)|\le \int_0^{+\infty}\E[\|(\nabla Y_t^{0,x})^*\| |\nabla f(Y_t^{0,x})|] dt$ where we recall that for $A\in \R^d\otimes \R^d, \|A\|:=\sup_{|z|\le 1, z\in \R^d}|Az|$ denotes the operator (or spectral) matrix norm. 
Thus, $|\nabla \varphi(x)| \le \|\nabla f\|_\infty \int_0^{+\infty}\E[\|(\nabla Y_t^{0,x})^*\|^p]^{1/p} dt=\|\nabla f\|_\infty \int_0^{+\infty}\E[\|\nabla Y_t^{0,x}\|^p]^{1/p} dt$. Now, 
\begin{eqnarray*}
\|\nabla Y_t^{0,x}\|&=&\sup_{|z|\le 1 }|\nabla Y_t^{0,x}z|.
\end{eqnarray*}
For any $z\in \R^d,|z|\le 1 $,  setting $Z_t^{0,x,z}:=\nabla Y_t^{0,x} z $, one has the following dynamics for the $\R^d $-valued process $(Z_s^{0,x,z})_{s\in [0,t]}$:
$$Z_t^{0,x,z}:=z+\int_{0}^t D b (Y_s^{0,x}) Z_s^{0,x,z} ds+\sum_{j=1}^d\int_0^t D\sigma_{\cdot j}(Y_s^{0,x}) Z_s^{0,x,z} dW_s^j.$$
Hence, we derive similarly to~\eqref{DER_FLOW_EXPO_CONTROLE} that $|Z_t^{0,x,z}|^p\le |z|^p\exp(-p\alpha t){\mathcal E}(M_t) $, where ${\mathcal E}(M_t) $ does not depend on $z$. Write now,
\begin{equation}\label{CTR_NORM_OPERATOR}
\E[\|\nabla Y_t^{0,x}\|^p]^{1/p}=\E[\sup_{|z|\le 1} |Z_t^{0,x,z}|^p]^{1/p}\le \E[\sup_{|z|\le 1}|z|^p\exp(-p\alpha t){\mathcal E}(M_t)]^{1/p}\le \exp(-\alpha t).
\end{equation}
This eventually proves the claim $\|\nabla \varphi\|_\infty:=\sup_{x\in \R^d}|\nabla \varphi(x)|\le \frac{\|\nabla f\|_\infty}{\alpha} $.
\end{proof}


\subsubsection{Additional smoothness}
\begin{trivlist}
 \item[--] Theorem~\ref{Poisson_regular} can be derived under \A{C${}_{\rm \mathbf R}$}, by iterating computations similar to the ones performed in  Lemma~\ref{PT_GR_BD}. On the other hand, to have the required smoothness, since we cannot expect some smoothing effect from a non-degenerate diffusion coefficient, we have to impose that $b,\sigma,f $  themselves lie  in ${\mathcal C}^{3,\beta}(\R^d, \R)$ and the restriction on the variations of $\sigma $ which ensures exponential integrability in time for the expectations of the iterated tangent flows, see Lemma A.8 in \cite{page:panl:14} for details (see the parallel between the above condition on $D\sigma $ and assumption \A{AC${}_p $} appearing p. 559 in \cite{page:panl:14}). 
\item[--] Proving Theorem~\ref{Poisson_regular} under \A{C${}_{\rm \mathbf{UE}}$} requires more sophisticated tools (Schauder estimates for operators with unbounded coefficients). 
\end{trivlist}

\begin{proof}[Proof of Theorem~\ref{Poisson_regular} under \A{C${}_{\rm \mathbf{UE}}$}]
 Let us begin with the scalar case. For $d=1$, set
 for all $x\in \R$, 
 \begin{equation}
\label{DEF_V_SCAL}
 v(x):=-\int_{0}^{+\infty}dt\, \E[\Psi(Y_t^{0,x})\partial_x Y_t^{0,x}]=-\int_{0}^{+\infty}dt \,\E\!\left[\Psi(Y_t^{0,x})\exp\Big(\int_0^tb'(Y_s^{0,x})ds\Big){\mathcal E}\Big(\int_0^t\sigma'(Y_s^{0,x}) dW_s\Big)\right]
 \end{equation}
 where  for all $y\in \R,\ \Psi(y):=\partial_{y} f(y)$. We observe that $\partial_{x}\varphi(x)=v(x) $. Also, from our assumptions on $f$, $b,\sigma $, we have that $ \Psi,b',\sigma' \in {\mathcal C}_{\textcolor{black}{b}}^{\textcolor{black}{0},\beta}(\R^d,\R)$. Theorems 2.4-2.6 in Krylov and Priola,~\cite{kryl:prio:10} then yield the existence of  a unique solution to the PDE:
\begin{equation}
\label{PDE_PREAL_KP}
\widetilde {\mathcal A}w(x)+b'(x) w(x)=\Psi(x),\quad\mbox{ where }\quad  \widetilde {\mathcal A}w(x)= {\mathcal A}w(x)+\sigma\sigma'(x) w'(x),
\end{equation}
belonging to ${\mathcal C}_{\textcolor{black}{b}}^{2,\beta}(\R^d,\R) $ and such that the following Schauder estimate holds:
\begin{equation}
\label{SCHAUDER}
\exists\, C\ge 1,\ \|w\|_{2,\beta}\le C(1+\|\Psi\|_\beta) .
\end{equation} 
Indeed, from \A{D${}_\alpha^p$}, we get that $b'(x)\le -\textcolor{black}{\alpha}
 <0 $ and the potential in~\eqref{PDE_PREAL_KP} has the good sign.
From~\eqref{DEF_V_SCAL} and the Girsanov theorem, we also get:
\begin{eqnarray*}
v(x)=-\int_0^{+\infty} \!\!dt\, \E\left[\Psi(\widetilde Y_t^{0,x})\exp\Big(\int_0^t b'(\widetilde Y_s^{0,x})ds\Big)\right],
\end{eqnarray*}
where $d\widetilde Y_s^{0,x} = \big(b(\widetilde Y_s^{0,x})+\sigma\sigma'(\widetilde Y_s^{0,x})\big) ds+\sigma(\widetilde Y_s^{0,x})dW_s $. Note that $\widetilde Y $ has generator $\widetilde {\mathcal A} $.
A simple identification procedure, similar  to the proof of Theorem II.1.1 in Bass~\cite{bass:97} then gives $v=w $. The result follows from~\eqref{SCHAUDER}. Let us emphasize that this is a quite deep and involved result for unbounded coefficients.

In the multi-dimensional setting, recalling the technical condition that for all $i\in \leftB 1,d\rightB,\ j\ge i $, $\Sigma_{i,j}(x)=\Sigma_{i,j}(x_i,\cdots,x_d) $, we have that differentiating formally the PDE~\eqref{POISSON} in the space variable $x_i,\ i\in \leftB 1,d\rightB$ yields that $\partial_{x_i} \varphi=v_i$ should satisfy:
\begin{eqnarray}
\label{SYST_I}
 \widetilde {\mathcal A} \,w_i(x)+\partial_{x_i} b_i(x)  w_i(x)&=&
 \textcolor{black}{}\Psi_i(x)-\sum_{j\in \leftB 1,d\rightB\setminus\{ i\}} \partial_{x_i}b_j(x) v_j(x)\notag\\
 &&\textcolor{black}{-\frac 12 \sum_{
j\in \leftB 1,i-1\rightB }\partial_{x_i}\Sigma_{j,j}(x)\partial_{x_j}v_j(x)-\hskip -0.5cm 
 \sum_{j\in \leftB 1,i-1\rightB}\ \sum_{k\in \leftB j+1,d \rightB \setminus\{i\} }
 \hskip -0.5cm \partial_{x_i}
\Sigma_{j,k}(x)\partial_{x_j} v_k(x)},
\end{eqnarray}
with  $\textcolor{black}{\Psi_i(x)}:=\textcolor{black}{\partial_{x_i}f(x)}$ and 
\begin{eqnarray*}
\widetilde {\mathcal A}w_i(x):= {\mathcal A}\,w_i(x)\textcolor{black}{+\frac 12\partial_{x_i}\Sigma_{i,i}(x)\partial_{x_i} w_i(x)}+ \sum_{\textcolor{black}{j\in \leftB 1,d\rightB\setminus\{ i\}}}
 \partial_{x_i}\Sigma_{i,j}(x)\partial_{x_j} w_i(x).
\end{eqnarray*}

 \textcolor{black}{We would now like to enter the previous framework of Schauder estimates.
 To do so, we  first observe from \A{D${}_{\alpha}^p $} and the Cauchy-Schwarz inequality that $\partial_{x_i}b_i(x)\le -\alpha <0$. Consider now $i=1$ in \eqref{SYST_I}. From our current assumptions on $f$, $b $ and the previous computations on the gradient for the multi-dimensional case, it remains to prove} 
$\widetilde \Psi_1(x):=\Psi_1(x)-\sum_{j\ne 1} \partial_{x_1}b_j(x) v_j(x) \in {\mathcal C}_{\textcolor{black}{b}}^{\textcolor{black}{0},\beta}(\R^d,\R)$. \textcolor{black}{This will be the case, once we will have proved that  $\nabla \varphi $ is $\beta $-H\"older continuous, which is a priori not direct. This property is assumed for the remaining of the proof and shown below. In particular, it leads to the restriction concerning the variations of $\sigma $ when $d>1$}. Hence, 
Theorems 2.4-2.6 in Krylov and Priola,~\cite{kryl:prio:10} still apply and give that there exists a unique solution $w_1\in {\mathcal C}_{\textcolor{black}{b}}^{2,\beta}(\R^d,\R)$ to~\eqref{SYST_I} which also satisfies:
\begin{equation}
\label{SCHAUDER_1}
\exists \,C\ge 1,\ \|w_1\|_{2,\beta}:=\sum_{\alpha, |\alpha|\in \leftB 0,2\rightB}\|D^\alpha w_1\|_\infty+[D^{(2)}w_1]_\beta\le C(1+\|\widetilde \Psi_1\|_\beta) =:\bar C(\textcolor{black}{
(\mathbf {{\mathcal L}_V}) , ({\mathbf R}_{1,\beta}), \A{UE}}).
\end{equation} 
The identification $w_1=\partial_{x_1} \varphi =v_1$ is standard. 
The control~\eqref{SCHAUDER_1} allows to iterate, since it gives that $\nabla w_1=(\partial_{x_{1}} v_1,\cdots,\partial_{x_{d}} v_1)=(\partial_{x_{1},x_1} \varphi_1,\cdots,\partial_{x_{d},x_1} \varphi)  $ is  $\beta$-H\"older. 
We thus get by induction, from the specific chosen structure on  $\sigma $ and by Theorems 2.4-2.6 in Krylov and Priola,~\cite{kryl:prio:10}, that for all $i\in \leftB 1,d\rightB $ there exists a unique solution $w_i\in {\mathcal C}_{\textcolor{black}{b}}^{\textcolor{black}{2},\beta}(\R^d,\R)$ to~\eqref{SYST_I} such that:
\begin{equation}
\label{SCHAUDER_2}
\begin{split}
\exists \,C\ge 1,\ \|w_i\|_{2,\beta}\le C(1+\|\widetilde \Psi_i\|_\beta) =:\bar C(\textcolor{black}{ 
(\mathbf {{\mathcal L}_V}) , ({\mathbf R}_{1,\beta}), \A{UE}}
),\\
 \widetilde \Psi_i(x):=\Psi_i(x)-\sum_{j\in \leftB 1,d\rightB \setminus \{i\}} \partial_{x_i}b_j(x) v_j(x)-\frac 12\sum_{{\tiny \begin{array}{c}1\le j<i,\\
 k\in \leftB 1,d\rightB \setminus\{ i\}
 \end{array}}}\partial_{x_i}\Sigma_{j,k}(x)\partial_{x_j} v_k(x).
 \end{split}
\end{equation}

\textcolor{black}{ The Lipschitz property of the mapping $x\mapsto \langle \nabla \varphi(x), b(x)\rangle  $ is eventually derived following the procedure described in Remark \ref{REM_COND_B_NABLA_PHI}.}

\end{proof}

\begin{remark}[Structure of $\sigma $] We emphasize that the structure condition on $\sigma $ assumed in Theorem~\ref{Poisson_regular} under \A{C${}_{\rm \mathbf{UE}}$} is mainly technical. It is of course always verified in dimension $d=1$. For $d>1 $ it is motivated by the fact that,  differentiating~\eqref{POISSON} without this assumption yields to consider a system of coupled linear PDEs with growing coefficients for which the Schauder estimates have not been established  yet. Following the existing literature for Schauder estimates for systems (see  {\it e.g.}  Boccia \cite{bocc:13}), we think that the results of Krylov and Priola should extend to this case. This would allow to get rid of the indicated condition. Here, the condition simply allows to decouple the system.

Let us  mention too that the results by Priola \cite{prio:09} could also be a starting point to investigate the smoothness of the Poisson problem for degenerate kinetic models.

These aspects will concern further research. 
\end{remark}

\textit{Additional Smoothness continued: $\beta $-H\"older continuity of the gradient through pathwise analysis. 
}
\label{SEC_COMPLETE_SMOOTHNESS}
{We control here, under \A{D${}_\alpha^p $}, $p\in (1,2] $ and \A{R${}_{1,\beta}$}, $\beta\in (0,1] $, the $\beta $-H\"older modulus of continuity of the gradient. We will progressively see how the restrictions on $D\sigma $ 
come out.
 For $(x,x') \in \R^{2d} $, write for all $i\in \leftB 1,d \rightB $:
\begin{eqnarray}
\Big|\partial_{x_i} \varphi(x)-\partial_{x_i} \varphi(x')\Big|= \bigg|\int_0^{+\infty} \Big(\E\,[\langle \nabla f(Y_t^{0,x}), \partial_{x_i}Y_t^{0,x}\rangle  ]-  \E\,[\langle \nabla f(Y_t^{0,x'}), \partial_{x_i}Y_t^{0,x'}\rangle  ]\Big) dt\bigg|\notag \\
\le \bigg|\int_0^{+\infty} \Big([\nabla f]_\beta\E\,[ |Y_t^{0,x}-Y_t^{0,x'}|^\beta|\partial_{x_i}Y_t^{0,x}| ]+ \|\nabla f\|_\infty \E\,[|\partial_{x_i}Y_t^{0,x}- \partial_{x_i}Y_t^{0,x'}|  ]\Big) dt\bigg|=:(G_1^\beta+G_2^\beta)(x,x').\notag \\
\label{PREAL_CTR_GRAD_BETA}
\end{eqnarray}
Let us first deal with the expectation in $G_1^\beta $. Namely, write
\begin{eqnarray*}
\E\,[ |Y_t^{0,x}-Y_t^{0,x'}|^\beta|\partial_{x_i}Y_t^{0,x}| ]\le \E\,[ |Y_t^{0,x}-Y_t^{0,x'}|^{\bar p \beta}]^{\frac{1}{\bar p}}\E[|\partial_{x_i}Y_t^{0,x}|^{\bar q} ]^{\frac 1{\bar q}},\ \bar p,\bar q>1,\ \bar p^{-1}+\bar q^{-1}=1. 
\end{eqnarray*}
Take now $\bar  p\beta =\bar q \iff \bar p=\frac{1+\beta}{\beta}, \ \bar q=1+\beta $ which leads to the same integrability constraints on the flows. 

If $\beta+1\le p $ in \A{D${}_{\alpha}^{p}$}, then we readily get similarly to~\eqref{DER_FLOW_EXPO_CONTROLE} that $\E[|\partial_{x_i}Y_t^{0,x}|^{\bar q} ]^{\frac 1{\bar q}}\le \exp(-\alpha t)$.

If now $\beta+1>p $, as soon as \A{D${}_{\bar \alpha}^{1+\beta}$} holds for some $\bar \alpha>0 $, which is actually the case provided that 
\begin{equation}\label{FIRST_CTR_DSIG}
\|D\sigma\|_\infty^2\le \frac{2\alpha}{
1+\beta-p
},  
\end{equation}
for
 $\bar q=1+\beta$, we again get similarly to~\eqref{DER_FLOW_EXPO_CONTROLE} that $\E[|\partial_{x_i}Y_t^{0,x}|^{\bar q} ]^{\frac 1{\bar q}}\le \exp(-\bar \alpha t) $. On the other hand, the mean value theorem yields:
\begin{eqnarray*}
\E\,[ |Y_t^{0,x}-Y_t^{0,x'}|^{\bar p \beta}]^{\frac{1}{\bar p}}&\le& |x-x'|^\beta\E\,[ \int_0^1 d\lambda \| \nabla Y_t^{0,x'+\lambda(x-x')}\|^{\bar p \beta}]^{\frac{1}{\bar p}}
\le |x-x'|^\beta \Big(\int_0^1 d\lambda \E\,[  \|  \nabla Y_t^{0,x'+\lambda(x-x')}\|^{\bar p \beta}]\Big)^{\frac{1}{\bar p}}\\
&\le& |x-x'|^\beta 
\left[\exp(-\alpha \beta t)\I_{1+\beta\le p}+\exp(-\bar \alpha \beta t)\I_{1+\beta>p}\right],
\end{eqnarray*}
exploiting ~\eqref{CTR_NORM_OPERATOR} for the last inequality provided that~\eqref{FIRST_CTR_DSIG}, which in turn implies that \A{D${}_{\bar \alpha}^{1+\beta}$} for some $\bar \alpha>0 $, holds if $1+\beta >p $. Plugging these bounds in~\eqref{PREAL_CTR_GRAD_BETA} gives that:
\begin{equation}
\label{CTR_G1} 
\forall (x,x')\in (\R^d)^2,\ |G_1^\beta(x,x')|\le \frac{
[\nabla f]_\beta}{(1+\beta)}\left[\frac{\I_{1+\beta\le p}}{\alpha}+\frac{\I_{1+\beta>p}}{\bar \alpha} \right] |x-x'|^\beta.
\end{equation}
We already see that, when $1+\beta> p $, for the parameter $p$ of the initial confluence condition \A{D${}_{\alpha}^p $}, a first constraint on the variations of $\sigma$, namely~\eqref{FIRST_CTR_DSIG}
appears. 

Let us now turn to $G_2^\beta $. Following the expansion of~\eqref{DEV_PUISS_P_FLOW} write:
\begin{align*}
 |\partial_{x_i} Y_t^{0,x}-\partial_{x_i} Y_t^{0,x'}|^{2}= &\;
 2\int_0^t \Big \langle  \partial_{x_i} Y_s^{0,x}
-\partial_{x_i} Y_s^{0,x'} , D b (Y_s^{0,x})\partial_{x_i} Y_s^{0,x}-D b (Y_s^{0,x'})\partial_{x_i} Y_s^{0,x'}\Big\rangle 
ds\notag\\
&+ 2\sum_{j=1}^d\int_0^t \Big\langle \partial_{x_i} Y_s^{0,x}-\partial_{x_i} Y_s^{0,x'},D\sigma_{\cdot j}(Y_s^{0,x})\partial_{x_i} Y_s^{0,x}-D\sigma_{\cdot j}(Y_s^{0,x'})\partial_{x_i} Y_s^{0,x'}\Big\rangle  
dW_s^j\notag\\
&+\sum_{j=1}^d\int_0^t 
|D\sigma_{\cdot j}(Y_s^{0,x})\partial_{x_i}Y_s^{0,x}-D\sigma_{\cdot j}(Y_s^{0,x'})\partial_{x_i}Y_s^{0,x'}|^2ds.\notag 
\end{align*}
Let $u(t):= \E\,|\partial_{x_i} Y_t^{0,x}-\partial_{x_i} Y_t^{0,x'}|^{2}$, $t\ge 0$. First note that $u(0)=0$. Taking now the expectation and interchanging expectation and time integration yields
\[
u(t) = \int_0^t \E\, \Xi_sds
\]
where $(\Xi_t)_{t\ge 0}$ is a pathwise continuous process clearly determined by the terms inside the above time integrals. One readily checks that, $t\mapsto \E\, \Xi_s$ is continuous so that $u$ is continuously differentiable and satisfies
\begin{align*}
u'(t)=&\;
 2\, \E\, \Big \langle  \partial_{x_i} Y_t^{0,x}
-\partial_{x_i} Y_t^{0,x'} , D b (Y_t^{0,x})\partial_{x_i} Y_t^{0,x}-D b (Y_t^{0,x'})\partial_{x_i} Y_t^{0,x'}\Big\rangle 
\notag\\
&+\sum_{j=1}^d \E\, |D\sigma_{\cdot j}(Y_t^{0,x})\partial_{x_i}Y_t^{0,x}-D\sigma_{\cdot j}(Y_t^{0,x'})\partial_{x_i}Y_t^{0,x'}|^2.\notag 
\end{align*}

Using the Young inequality for a parameter $\varepsilon \in (0,1] $,  small enough and to be chosen further, we derive:
\begin{align*}
u'(t)
\le& \; 2\,\E\Bigg[ \Big \langle  \frac{\partial_{x_i} Y_t^{0,x}-\partial_{x_i} Y_t^{0,x'}}{|\partial_{x_i} Y_t^{0,x}-\partial_{x_i} Y_t^{0,x'}|} , D b (Y_t^{0,x})\frac{\partial_{x_i} Y_t^{0,x} -\partial_{x_i} Y_t^{0,x'}}{|\partial_{x_i} Y_t^{0,x} -\partial_{x_i} Y_t^{0,x'}|}\Big\rangle |\partial_{x_i} Y_t^{0,x}-\partial_{x_i} Y_t^{0,x'}|^{2} 
\\
&+
\int_0^t \|D b (Y_t^{0,x})-D b (Y_t^{0,x'})\| |\partial_{x_i} Y_t^{0,x'}| |\partial_{x_i} Y_t^{0,x}-\partial_{x_i} Y_t^{0,x'}| \Bigg]\\
&+ \E\Bigg[  (1+\varepsilon)\sum_{j=1}^d\bigg|D\sigma_{\cdot j}(Y_t^{0,x})\frac{\partial_{x_i}Y_t^{0,x}-\partial_{x_i}Y_t^{0,x'}}{|\partial_{x_i}Y_t^{0,x}-\partial_{x_i}Y_t^{0,x'}|}\bigg|^2 |\partial_{x_i}Y_t^{0,x}-\partial_{x_i}Y_t^{0,x'}|^{2}\\
&+ (1+\varepsilon^{-1})\sum_{j=1}^d\|D\sigma_{\cdot j}(Y_t^{0,x})-D\sigma_{\cdot j}(Y_t^{0,x'})\|^2|\partial_{x_i} Y_t^{0,x'}|^2  \Bigg].
\end{align*}
From this computation, the point is now to make the confluence condition \A{D$ {}_\alpha^p$}  appear and to separate the components for which we will exploit the $\beta $-H\"older continuity, namely $Db, (D\sigma_{\cdot j})_{j\in \leftB 1,n \rightB} $. To do so we first observe that:
\begin{eqnarray*}
&&\Big\langle  \frac{\partial_{x_i} Y_t^{0,x}
-\partial_{x_i} Y_t^{0,x'}}{|\partial_{x_i} Y_t^{0,x}
-\partial_{x_i} Y_t^{0,x'}|} , D b (Y_t^{0,x})\frac{\partial_{x_i} Y_t^{0,x} -\partial_{x_i} Y_t^{0,x'}}{|\partial_{x_i} Y_t^{0,x} -\partial_{x_i} Y_t^{0,x'}|}\Big\rangle
+
\frac{1}2 \sum_{j=1}^d\Bigg(\bigg|D\sigma_{\cdot j}(Y_t^{0,x})\frac{\partial_{x_i}Y_t^{0,x}-\partial_{x_i}Y_t^{0,x'}}{|\partial_{x_i}Y_t^{0,x}-\partial_{x_i}Y_t^{0,x'}|}\bigg|^2 
\Bigg)\\
&&\le -\alpha+(2-p)\frac 1 
2\|D\sigma\|_\infty^2=-\tilde \alpha,
\end{eqnarray*}
where we suppose from now on that 
\begin{equation}
\label{COND_TILDE_ALPHA}
-\tilde \alpha:=-\alpha+(2-p)\frac  1
2\|D\sigma\|_\infty^2<0 \iff \|D\sigma\|_\infty^2 <\frac{2 \alpha}{
2-p
}.
\end{equation}
Hence,
\begin{align*}
u'(t)\le &\; 2\E\bigg[ (-\tilde \alpha+\frac {\varepsilon} 2 \|D\sigma\|_\infty^2 ){|\partial_{x_i} Y_t^{0,x} -\partial_{x_i} Y_t^{0,x'}|^2} \bigg]\\
&+2 \E\bigg[\|D b (Y_t^{0,x})-D b (Y_t^{0,x'})\| |\partial_{x_i} Y_t^{0,x'}| |\partial_{x_i} Y_t^{0,x}-\partial_{x_i} Y_t^{0,x'}| \bigg]\\
&+ 
\E\bigg[(1+\varepsilon^{-1})\sum_{j=1}^d\|D\sigma_{\cdot j}(Y_t^{0,x})-D\sigma_{\cdot j}(Y_t^{0,x'})\|^2|\partial_{x_i} Y_t^{0,x'}|^2  
 \bigg].
\end{align*}
Using now again the Young inequality, with $\eta\in(0,1] $ small enough, for the middle term of the above r.h.s., we obtain: 
\begin{align}
u'(t) \le &\; 2 \Big(-\tilde \alpha+\frac {\varepsilon} 2 \|D\sigma\|_\infty^2 +\frac{\eta}{2} \Big) \E\Big[|\partial_{x_i} Y_t^{0,x} -\partial_{x_i} Y_t^{0,x'}|^{2}\Big] \notag\\
&+  \eta^{-1}\E\big[\|D b (Y_t^{0,x})-D b (Y_t^{0,x'})\|^{2} |\partial_{x_i} Y_t^{0,x'}|^{2}\big]  + (1+\varepsilon^{-1} )\sum_{j=1}^d\E\Big[\|D \sigma_{\cdot j} (Y_t^{0,x})- D \sigma_{\cdot j} (Y_t^{0,x'})\|^{2} |\partial_{x_i} Y_t^{0,x'}|^{2}\Big] 
\notag\\
\le &\; 2 \Big(-\tilde \alpha+\frac {\varepsilon} 2 \|D\sigma\|_\infty^2 +\frac \eta 2 \Big) \E\Big[|\partial_{x_i} Y_t^{0,x} -\partial_{x_i} Y_t^{0,x'}|^{2}\Big] \notag\\
&+ \eta^{-1}[Db]^{2}_\beta\E\big[|Y_t^{0,x}-Y_t^{0,x'}|^{2 \beta} |\partial_{x_i} Y_t^{0,x'}|^{2}\big] + (1+\varepsilon^{-1}) [D\sigma]_\beta^{2}
\E\Big[| Y_t^{0,x}- Y_t^{0,x'}|^{2 \beta } |\partial_{x_i} Y_t^{0,x'}|^{2}\Big] .\label{PREAL_GRONWALL_G2}
\end{align}
Denote:
\begin{eqnarray*}
-\tilde \alpha_{\varepsilon,\eta,\sigma} := -\tilde \alpha+\frac {\varepsilon} 2 \|D\sigma\|_\infty^2 +\frac \eta 2<0,
\end{eqnarray*}
for $\varepsilon,\ \eta $ small enough. Setting for every $t\ge 0$,
\[
r(t):= \eta^{-1}[Db]^{2}_\beta\E\big[|Y_t^{0,x}-Y_t^{0,x'}|^{2\beta} |\partial_{x_i} Y_t^{0,x'}|^{2}\big]+(1+\varepsilon^{-1})D\sigma]_\beta^{2}
\E\Big[| Y_t^{0,x}- Y_t^{0,x'}|^{2 \beta } |\partial_{x_i} Y_t^{0,x'}|^{2}\Big] ,
\]
equation~\eqref{PREAL_GRONWALL_G2} reads an ordinary differential inequation:
\begin{eqnarray*}
u'(t)\le  -2 \tilde \alpha_{\varepsilon, \eta,\sigma} u(t) + r(t) , \quad u(0) =0.
\end{eqnarray*}
We derive from the Gronwall lemma that 
$$
u(t)= \E[|\partial_{x_i} Y_t^{0,x}-\partial_{x_i} Y_t^{0,x'}|^{2}]
\le \exp(-2\tilde \alpha_{\varepsilon,\eta,\sigma} t)\int_0^t \exp(2\tilde \alpha_{\varepsilon,\eta,\sigma} s) r(s) ds .
$$

Reproducing as well the computations that led to~\eqref{CTR_G1}, we derive:
\[
u(t) \le C_{\eta,\varepsilon,\beta}|x-x'|^{2 \beta} \!\!
\int_0^t\!\! \exp\big(-2\tilde \alpha_{\varepsilon,\eta,\sigma} (t-s)\big)\bigg( \E\Big[ |\partial_{x_i} Y_s^{0,x'}|^{2(1+\beta)}\Big]^{\frac{1}{1+\beta}} \!\!+\!\!\int_{0}^1\! d\lambda \E\Big[\|\nabla Y_s^{0,x'+\lambda (x-x')}\|^{2(1+\beta)} \Big]^{\frac{\beta}{1+\beta}}
\bigg)ds.
\]}
From the analysis leading to~\eqref{DER_FLOW_EXPO_CONTROLE},~\eqref{CTR_NORM_OPERATOR} we now derive:
\begin{align*}
 \E[|\partial_{x_i} Y_t^{0,x}&-\partial_{x_i} Y_t^{0,x'}|^{2}]\\
&\le\frac{C_{\eta,\varepsilon,\beta}}2|x-x'|^{2 \beta} \exp(-2\tilde \alpha_{\varepsilon,\eta,\sigma} t) \bigg[\frac{\exp\big(2(\tilde \alpha_{\varepsilon,\eta,\sigma}-\tilde \alpha_{2(1+\beta)}) t\big)}{\tilde \alpha_{\varepsilon,\eta,\sigma}-\tilde \alpha_{2(1+\beta)}}+\frac{\exp\big(2 (\tilde \alpha_{\varepsilon,\eta,\sigma}-\beta\tilde \alpha_{2(1+\beta)}) t\big)}{\tilde \alpha_{\varepsilon,\eta,\sigma}-\beta\tilde \alpha_{2(1+\beta)}}\bigg],\\
\mbox{and }\hskip 2cm &\\
-\tilde \alpha _{2(1+\beta)}&\le -\alpha +\big(2(1+\beta)-p\big) \frac 1 2\|D\sigma\|_\infty^2.
\end{align*}
Thus, $\tilde \alpha _{2(1+\beta)}<0 $ as soon as
\begin{equation} \label{BD_TILDE_P_BETA}
\|D\sigma\|_\infty^2<
\frac{2\alpha}{
2(1+\beta)-p
} 
,
\end{equation}
which is precisely the restriction on the variations of $ \sigma$ appearing in \A{C${}_{\rm \mathbf{UE}}$} when $d> 1$, then $
\tilde \alpha_{2(1+\beta)} >0$ and:
\begin{eqnarray*}
&&\E[|\partial_{x_i} Y_t^{0,x}-\partial_{x_i} Y_t^{0,x'}|^{2}]\le \bar C_{\eta,\varepsilon,\beta}|x-x'|^{2 \beta}\exp(-\tilde \alpha_{2(1+\beta)} t).
\end{eqnarray*}
This last control then gives the expected bound for the $\beta $-H\"older modulus of the gradient. Namely,
from~\eqref{PREAL_CTR_GRAD_BETA},~\eqref{CTR_G1},
$$[\partial_{x_i} \varphi ]_\beta <\frac{
[\nabla f]_\beta}{(1+\beta)}\Big[\frac{\I_{ 1+\beta \le p}}{\alpha}+\frac{\I_{1+\beta>p}}{\bar \alpha} \Big]+\frac{\|\nabla f\|_\infty \bar C_{\eta,\varepsilon,\beta}}{\tilde \alpha_{2(1+\beta)} }.$$

\subsection{Proof of the Practical Results of Section~\ref{PRACT_DEV_BD}}
We first begin with the proof of the 

\subsubsection{Slutsky like Theorem~\ref{Slutsky_theorem}} We keep here for simplicity the generic notation $\|\cdot\| $ for any admissible matrix norm according to the assumptions of the theorem. 
We first write:
\begin{equation}
\P \left [\sqrt{\Gamma_n} \frac{\nu_n(f)-\nu(f)}{\sqrt{\nu_n \left (\|\sigma\|^2  \right )  } } \geq a \right ]
 = \P \left [\nu_n(\mathcal A \varphi ) \geq \frac{a}{\sqrt{\Gamma_n} } \sqrt{\nu_n \left (   \|\sigma\|^2  \right )}  \right ].
\end{equation}
We then proceed similarly to Theorem~\ref{THM_COBORD}, with an exponential Bienaymé-Tchebychev inequality, for all $\lambda>0$ we have:
\begin{eqnarray*}
 \P \left [\sqrt{\Gamma_n} \frac{\nu_n(f)-\nu(f)}{\sqrt{\nu_n \left (  \|\sigma\|^2  \right )   } }\right.\!\!&\!\! \geq\!\! &\! \!\! a \Bigg]
\leq 
\E \Big [ \exp \left (  - \frac{a\lambda}{\sqrt{\Gamma_n} } \sqrt{ \nu_n \left (   \|\sigma\|^2  \right )  } \right )  
\exp \left ( \lambda \nu_n( \mathcal A \varphi) \right )
\Big ]
\\
&=&
 \exp \left (  - \frac{a\lambda}{\sqrt{\Gamma_n} }\sqrt{  \nu\left (   \|\sigma\|^2  \right ) } \right )  
\E \Big [ \exp \left (  - \frac{a\lambda}{\sqrt{\Gamma_n} } \left [ \sqrt{  \nu_n \left (   \|\sigma\|^2  \right )  }  -
\sqrt{  \nu\left (   \|\sigma\|^2  \right ) }   \right ] \right )
\exp \left ( \lambda \nu_n( \mathcal A \varphi) \right )
\Big ]
\\
&=&
 \exp \left (  - \frac{a\lambda}{\sqrt{\Gamma_n} }\sqrt{  \nu\left (   \|\sigma\|^2  \right ) } \right )  
\E \Big [ \exp \left (  - \frac{a\lambda}{\sqrt{\Gamma_n} } \frac{   \nu_n \left (   \|\sigma\|^2  \right )    -
 \nu\left (   \|\sigma\|^2  \right )    }{\sqrt{  \nu_n \left (   \|\sigma\|^2  \right )  }  +
\sqrt{  \nu\left (   \|\sigma\|^2  \right ) }  } \right )
\exp \left ( \lambda \nu_n( \mathcal A \varphi) \right )
\Big ]
\\
&=&
 \exp \left (  - \frac{a\lambda}{\sqrt{\Gamma_n} }\sqrt{  \nu\left (   \|\sigma\|^2  \right ) } \right )  
\E \Big [ \exp \left (  - \frac{a\lambda}{\sqrt{\Gamma_n} } 
\frac{   \nu_n \left (  \mathcal A \vartheta  \right )    }{\sqrt{  \nu_n \left (   \|\sigma\|^2  \right )  }  +
\sqrt{  \nu\left (   \|\sigma\|^2  \right ) }  } \right )
\exp \left ( \lambda \nu_n( \mathcal A \varphi) \right )
\Big ].
\end{eqnarray*}
By the H\"older inequality, for $\widetilde p, \widetilde q>1$, such that $\frac 1 {\widetilde p} + \frac 1 {\widetilde q}=1$:
\begin{eqnarray*}
\P \left [\sqrt{\Gamma_n} \frac{\nu_n(f)-\nu(f)}{\sqrt{\nu_n \left ( \|\sigma\|^2  \right )   } }\right.\!\!&\!\! \geq\!\! &\! \!\! a \Bigg]
\\
\!\!&\!\! \leq\!\! &\! \!\
  \exp \left (  - \frac{a\lambda}{\sqrt{\Gamma_n} }\sqrt{  \nu\left (   \|\sigma\|^2  \right ) } \right )  
\hskip-0.2cm \left [\E  \exp \left (  - \frac{a \widetilde p \lambda}{\sqrt{\Gamma_n} } 
\frac{   \nu_n \left (  \mathcal A \vartheta  \right )    }{\sqrt{  \nu_n \left (   \|\sigma\|^2  \right )  }  +
\sqrt{  \nu\left (   \|\sigma\|^2  \right ) }  } \right )\right ]^{1/\widetilde p}\hskip-0.5cm 
\Big [\E  \exp \left ( \lambda \widetilde q \nu_n( \mathcal A \varphi) \right )
\Big ]^{1/ \widetilde q}\hskip-0.25cm .
\end{eqnarray*}
The proof of Theorem~\ref{THM_COBORD} yields:
\begin{eqnarray}\label{Tcheb_Slutsky}
\P  \Bigg[\sqrt{\Gamma_n} \frac{\nu_n(f)-\nu(f)}{\sqrt{\nu_n \left (  \|\sigma\|^2  \right )   } }\geq a \Bigg]
&\leq&  \mathscr R_n  \exp \left (  - \frac{a\lambda}{\sqrt{\Gamma_n} }\sqrt{  \nu\left (   \|\sigma\|^2  \right ) } \right )   \exp \left ( \frac{\rho 
\widetilde q \lambda^2}{\Gamma_n} \widetilde A_n  +  \frac{\rho^3 
 \widetilde q^3 \lambda^4}{(\rho-1)\Gamma_n} \widetilde B_n\right ) 
\nonumber \\
&&
\times \left (\E  \exp \left (  - \frac{a \widetilde p \lambda}{\sqrt{\Gamma_n} } 
\frac{   \nu_n \left (  \mathcal A \vartheta  \right )    }{\sqrt{  \nu_n \left (   \|\sigma\|^2  \right )  }  +
\sqrt{  \nu\left (   \|\sigma\|^2  \right ) }  } \right )\right )^{1/\widetilde p},
\end{eqnarray}
where we recall from identity~\eqref{eq:ABtilde}:
\begin{equation*}
\widetilde A_n= \frac{ q[\varphi]_1^2\nu(\|\sigma\|^2)}{2}+e_n
 \quad\mbox{ and }\quad 
\widetilde B_n 
=\frac{q^3[\varphi]_1^4}{4}\Big( \frac{\bar q\|\sigma\|_\infty^2[\vartheta]_1^2}{2}+\bar e_n\Big).
\end{equation*}
Also, $\mathscr R_n \underset{n}{\rightarrow} 1$ denotes a ``generic" remainder.
Observe that thanks to the \textcolor{black}{bounds of Theorem~\ref{Poisson_regular} (stated in the above Lemma~\ref{PT_GR_BD})}, we get:
\begin{equation}\label{Grad_Slutsky}
\widetilde A_n \leq \frac{ q[f]_1^2\nu(\|\sigma\|^2)}{2 \alpha 
}+e_n.
\end{equation}
Let us now handle the remainder $\left [\E  \exp \left (  - \frac{a \widetilde p \lambda}{\sqrt{\Gamma_n} } 
\frac{   \nu_n \left (  \mathcal A \vartheta  \right )    }{\sqrt{  \nu_n \left (   \|\sigma\|^2  \right )  }  +
\sqrt{  \nu\left (   \|\sigma\|^2  \right ) }  } \right )\right ]^{1/\widetilde p}$ :
\begin{eqnarray*}
&&
\left [ \E  \exp \left (  - \frac{a \widetilde p \lambda}{\sqrt{\Gamma_n} } 
\frac{   \nu_n \left (  \mathcal A \vartheta  \right )    }{\sqrt{  \nu_n \left (   \|\sigma\|^2  \right )  }  +
\sqrt{  \nu\left (   \|\sigma\|^2  \right ) }  } \right )\right ]^{1/\widetilde p}
\\
&=& 
\left[\E \left( \exp \left [  - \frac{a \widetilde p \lambda}{\sqrt{\Gamma_n} } 
\frac{   \nu_n \left (  \mathcal A \vartheta  \right )    }{\sqrt{  \nu_n \left (   \|\sigma\|^2  \right )  }  +
\sqrt{  \nu\left (   \|\sigma\|^2  \right ) }  } \right )  \left ( \mbox{\bf 1}_{\nu_n  ( \mathcal A  \vartheta   ) \geq 0}+ \mbox{\bf 1}_{\nu_n  ( \mathcal A  \vartheta   ) < 0} \right ) \right )\right]^{1/\widetilde p}
\\
&\leq&
\Bigg( \left [\E  \exp \left (   \frac{a \widetilde p^2 \lambda}{\sqrt{\Gamma_n} } 
\frac{   \nu_n \left (  \mathcal A \vartheta  \right )    }{
\sqrt{  \nu\left (   \|\sigma\|^2  \right ) }  } \right )  \right ]^{1/\widetilde p} \P\big[\nu_n  ( \mathcal A  \vartheta   ) \geq 0\big]^{1/\widetilde q}
\\
&&+\left[\E  \exp \left (  - \frac{a \widetilde p^2 \lambda}{\sqrt{\Gamma_n} } 
\frac{   \nu_n \left (  \mathcal A \vartheta  \right )    }{
\sqrt{  \nu\left (   \|\sigma\|^2  \right ) }  } \right )  \right]^{1/\widetilde p} \P\big[\nu_n  ( \mathcal A  \vartheta   ) < 0\big]^{1/\widetilde q}\Bigg)^{1/\widetilde p}.
\end{eqnarray*}
Let us mention that we introduced the above partition in order to get a sharper constant in the final inequality, 1 below instead of 2, which would follow getting rid of the indicator functions.
Now, by Theorem ~\ref{ineq_fin} we easily get:
\begin{eqnarray}\label{slutsky_ineq_final}
&&\E \left [ \exp \left (  - \frac{a \widetilde p \lambda}{\sqrt{\Gamma_n} } 
\frac{   \nu_n \left (  \mathcal A \vartheta  \right )    }{\sqrt{  \nu_n \left (   \|\sigma\|^2  \right )  }  +
\sqrt{  \nu\left (   \|\sigma\|^2  \right ) }  } \right )\right ]^{1/\widetilde p}
\nonumber \\
&&\leq
\mathscr R_n \exp \left(  \frac{a^2\widetilde p^2 }{\textcolor{black}{\Gamma_n \nu(\|\sigma\|^2})} \frac{  \lambda^2}{\Gamma_n }  \Big( \frac{\|\sigma \|_\infty^2 \|\nabla \vartheta \|_\infty^2}2+ e_n\Big) \right )
\left [
\P\big[\nu_n  ( \mathcal A  \vartheta   ) \geq 0\big]^{1/\widetilde q}
+ \P\big[\nu_n  ( \mathcal A  \vartheta   ) < 0\big]^{1/\widetilde q}
\right ]^{1/\widetilde p}.
\end{eqnarray}
We choose $\widetilde p:= \widetilde p(n) \to + \infty$, such that $ \frac{\widetilde p^2a^2 }{\Gamma_n } \to 0$, and so $(\P\big[\nu_n  ( \mathcal A  \vartheta   ) \geq 0\big]^{1/\widetilde q}
+ \P\![\nu_n  ( \mathcal A  \vartheta   ) < 0
]^{1/\widetilde q})^{1/\widetilde p} \le 2^{1/\widetilde p}\to1$. 
Moreover, exploiting again that for the \textit{Gaussian regime}, $\frac{\widetilde p^2a^2 }{\Gamma_n }  \to 0$, we obtain by~\eqref{slutsky_ineq_final} and~\eqref{Tcheb_Slutsky}:
\begin{eqnarray}\label{Tcheb_Slutsky_DEF}
&&\P \left [\sqrt{\Gamma_n} \frac{\nu_n(f)-\nu(f)}{\sqrt{\nu_n \left ( \|\sigma\|^2  \right )   } } \geq a \right ]
\leq  \mathscr R_n  \exp \left (  - \frac{a\lambda}{\sqrt{\Gamma_n} }\sqrt{  \nu\left (   \|\sigma\|^2  \right ) } \right )   \exp \left ( \frac{\rho 
\widetilde q \lambda^2}{\Gamma_n} (\widetilde A_n+e_n)  +  \frac{\rho^3 
\widetilde q^3 \lambda^4}{(\rho-1)\Gamma_n}\widetilde B_n\right ) .
\end{eqnarray}
From identity~\eqref{Tcheb_Slutsky_DEF}, the optimization over $\lambda$ is similar to the one performed in the proof of Theorem~\ref{THM_COBORD}. This yields the deviation bound~\eqref{DEV_SLUTSKY}. The non-asymptotic confidence interval in~\eqref{NA_SLUTSKY} is derived as for Theorem~\ref{NA_CI_FIRST} from \textcolor{black}{the gradient bounds of Theorem~\ref{Poisson_regular}
} and~\eqref{DEV_SLUTSKY}. \hfill $\qed $

\subsection{Regularization of Lipschitz Sources}\label{REG_LIP_SOURCE}
We assume here that assumptions \A{C2}, \A{${\mathbf {\mathcal L}}_{{\mathbf V}} $}, \A{UE} 
are in force. We suppose as well that the following smoothness holds for $b,\sigma$:
\begin{trivlist}
\item[\A{R${}_{b,\sigma} $}] Regularity and Structure. We assume that there exists $\beta\in (0,1) $ such that $b,\sigma $ in~\eqref{eq_diff} belong to ${\mathcal C}^{1,\beta}(\R^d,\R^d)$ and ${\mathcal C}_b^{1,\beta}(\R^d,\R^{d}\otimes \R^d ) $ respectively. Also, for all $(i,j)\in \leftB 1,d\rightB^2 $, $\Sigma_{i,j}(x)=\Sigma_{i,j}(x_{i\wedge j},\cdots,x_d) $.
\end{trivlist}
Importantly, we are interested, under assumptions \A{C2}, \A{${\mathbf {\mathcal L}}_{{\mathbf V}} $}, \A{UE}, \A{R${}_{b,\sigma} $}, in giving controls for the estimation of $\nu(f) $ when the source $f$ is simply Lipschitz continuous. This is indeed the natural framework for the source which can be handled through functional inequality techniques, see \cite{MalrieuTalay}, \cite{bois:11}.  

To \textcolor{black}{comply with} our previous framework, namely to exploit the smoothness result of Theorem~\ref{Poisson_regular} \textcolor{black}{under \A{C${}_{\rm \mathbf{UE}}$}}, we need to regularize the source. Let $\eta $ be a mollifier (i.e. a non-negative compactly supported function such that $\int_{\R^d} \eta(x)dx=1 $). Define for $\delta>0 $, $\eta_\delta(x)=\frac{1}{\delta^{d}}\eta(\frac{x}{\delta} ) $. We regularize $f$ introducing $f_\delta:=f\star \eta_\delta $ where $\star  $ stands for the convolution on $\R^d$. From usual estimates, we obtain:
\begin{eqnarray}
\label{CTR_REG_DELTA}
\exists \,C_\eta>0,\ \forall x\in \R^d,&& |f_\delta(x)-f(x)|\le C_\eta \delta [f]_1,\nonumber\\
\forall \beta\in (0,1),&&  [\nabla f_\delta]_\beta \le C_\eta [f]_1\delta^{-\beta}. 
\end{eqnarray}
We emphasize here that we will choose $\beta $ later in order to be compatible with a certain range of step sequences.
We assume for simplicity that $\theta\in (1/3,1] $ (no bias). Recall that we want to investigate:
\begin{eqnarray}
\label{DEV_REG_DELTA}
\P[\sqrt{\Gamma_n} (\nu_n(f)-\nu(f))\ge a]&=&\P\Big[(\nu_n(f_\delta)-\nu(f_\delta))+R_{n,\delta}(f)\ge \frac a{\sqrt{\Gamma_n} }\Big],\nonumber\\
R_{n,\delta}(f)&:=&[(\nu_n(f)-\nu(f))-(\nu_n(f_\delta)-\nu(f_\delta))].
\end{eqnarray}
From~\eqref{CTR_REG_DELTA}, one readily gets:
\begin{equation}
\label{CTR_RN_DELTA}
|R_{n,\delta}(f)|\le 2C_\eta \delta [f]_1.
\end{equation}
On the other hand, the coefficients $b,\sigma $ and the source $f_\delta$ satisfy assumption \A{R${}_{1,\beta}$}
 (observe indeed that the mollified function $f_\delta\in C^{1,\beta}(\R^d,\R)$). Hence, Theorem~\ref{Poisson_regular} yields that there exists a unique solution $\varphi_\delta\in C^{3,\beta}(\R^d,\R) $ to the equation:
\begin{equation}
\label{POISS_DELTA}
{\mathcal A} \varphi_\delta=f_\delta-\nu(f_\delta).
\end{equation}
Observe from the proof of Theorem~\ref{Poisson_regular} \textcolor{black}{under \A{C${}_{\rm \mathbf{UE}}$}} (see equations~\eqref{BD_GRAD_POISS} and \textcolor{black}{\eqref{SCHAUDER_2}}) and~\eqref{CTR_REG_DELTA}  that:
\begin{equation}
\label{BD_POISS_REG}
\begin{split}
&\|\nabla \varphi_\delta\|_\infty\le \alpha^{-1}[f]_1,\  \forall \beta\in (0,1),\ \exists \,C_\beta>0,\ \forall i\in \{1,2\},\ [ \varphi_\delta^{(i)}]_1\le C_\beta(1+\|\nabla f_\delta\|_{{\mathcal C}^\beta})\le C_\beta\delta^{-\beta},\ [\varphi_\delta^{(3)} ]_\beta\le  C_\beta\delta^{-\beta},\\
&
 \textcolor{black}{[\langle \nabla \varphi_\delta, b \rangle]_1\le C_\beta\delta^{-\beta}}.
\end{split}
\end{equation}
Now, from~\eqref{POISS_DELTA} the deviation in~\eqref{DEV_REG_DELTA} rewrites:
\begin{equation}
\P\big[\sqrt{\Gamma_n} (\nu_n(f)-\nu(f))\ge a\big]=\P\Big[\nu_n({\mathcal A} \varphi_\delta)+R_{n,\delta}(f)\ge \frac{a}{\sqrt{\Gamma_n}}\Big].
\end{equation}
From~\eqref{CTR_RN_DELTA}, the term $R_{n,\delta}(f) $ can be seen as a remainder as soon as $\frac{a}{\sqrt{\Gamma_n}}\gg 2C_\eta \delta [f]_1  \ge |R_{n,\delta}(f)| $. 
On the other hand, the deviations of $\nu_n({\mathcal A}\varphi_\delta) $ can be analyzed as above, reproducing the proofs of Theorems~\ref{ineq_fin} and~\ref{THM_COBORD}, replacing the bounds on $ ([\varphi^{(i)}]_1)_{i\in \{1,2\} }, [\varphi^{(3)}]_\beta $ appearing therein by those of equation~\eqref{BD_POISS_REG}. Precisely, we get from~\eqref{CTR_RN_DELTA}, similarly to~\eqref{BD_P_THM_2} (replacing the controls on $\varphi $ by those on $\varphi_\delta $ in the proofs of Lemmas~\ref{CTR_BIAS_1} and~\ref{ineq_rest}):
 \begin{eqnarray}
\P\left[ \Big|\nu_n( \mathcal{A} \varphi_\delta )+R_{n,\delta}(f)\Big| \geq \frac a{\sqrt{\Gamma_n}} \right] \leq  
2
\Bigg[ \E\exp\Big(\textcolor{black}{-\frac{q\lambda_n}{\Gamma_n}} M_n \Big)\Bigg]^\frac{1}{q}\exp\Big(-\frac{a \lambda_n}{\sqrt {\Gamma_n}}(1-\frac{\sqrt{\Gamma_n}2C_\eta [f]_1\delta}{a}) \Big) \notag\\
\times \exp\Big(\frac{\lambda_n^2}{2\Gamma_np}+\frac{p(a_n^\delta)^2}{2}\Big)  \exp\left(\frac{3 p C_{V,\varphi}^2 \lambda_n^2}{ c_V \Gamma_n^2}+\frac{c_V}{p}\right)(I_V^1)^{\frac 1{2p}}\notag\\
\times\exp\left ( C_{\ref{ineq_rest_b}}^{\delta}\frac{p \lambda_n^2  (\Gamma_n^{(2)})^2}{  \Gamma_n^2}\right) (I_V^1)^{\frac 1{4p}}\times \exp\left(C_{\ref{ineq_rest_b2_phi}}^{\delta}
\textcolor{black}{\frac{p\lambda_n^2 (\Gamma_n^{(2)})^2}{\Gamma_n^2}}
\right)(I_V^1)^{\frac1{4p}}\label{DEV_PROV_LIP}
\end{eqnarray}
where $C_{\ref{ineq_rest_b2_phi}}^\delta:= C_\beta\delta^{-\beta }\textcolor{black}{
\frac{(C \sqrt{C_{_V}})^2}{c_V}
}$ and $C_{\ref{ineq_rest_b}}^\delta:=\frac{\|\sigma\|_\infty^4C_\beta^2\delta^{-2\beta}}{4}\frac{C_{_V}}{c_V} $
precisely correspond to the modifications of the constants $C_{\ref{ineq_rest_b2_phi}}$ and $C_{\ref{ineq_rest_b}}:=\frac{\|\sigma\|_\infty^4[\varphi^{(2)}]_1^2}{4}\frac{C_{_V}}{c_V} $ introduced in the proof of Lemma~\ref{ineq_rest} when replacing $\|D^2\varphi\|_\infty $ by $\|D^2\varphi_\delta\|_\infty \le C_\beta\delta^{-\beta}$ and  
\textcolor{black}{$[\langle \nabla \varphi, b \rangle]_1\le C $ by $[\langle \nabla \varphi_\delta, b \rangle]_1\le C_\beta\delta^{-\beta}C $}. 
Similarly, 
$$a_n^\delta := \frac{ [\varphi_\delta^{(3)}]_{\beta}  \big\|\sigma\big\|_\infty^{(3+\beta)} \E\big[|U_1|^{3+\beta} \big] }{(1+\beta)(2+\beta)(3+\beta) }   \frac{\Gamma_n^{(\frac{3 +\beta}{2})}}{\sqrt{\Gamma_n}}\le \frac{ C_\beta \delta^{-\beta}  \big\|\sigma\big\|_\infty^{(3+\beta)} \E\big[|U_1|^{3+\beta} \big] }{(1+\beta)(2+\beta)(3+\beta) }   \frac{\Gamma_n^{(\frac{3 +\beta}{2})}}{\sqrt{\Gamma_n}} , $$ is obtained from the definition of $a_n$ in Lemma~\ref{CTR_BIAS_1} replacing $[\varphi^{(3)}]_\beta $ by $[\varphi_\delta^{(3)}]_\beta $. From the above equation and Lemma~\ref{ineq_Psi} we get:
\begin{eqnarray*}
\P\left[|\nu_n(f)-\nu(f)|\ge \frac{a}{\sqrt{\Gamma_n}}\right]=\P\left[ \Big|\nu_n( \mathcal{A} \varphi_\delta )+R_{n,\delta}(f)\Big| \geq \frac a{\sqrt{\Gamma_n}} \right]
\le  2(I_V^1)^{\frac 1p}\exp\Big(\textcolor{black}{\frac{c_V}{p}} 
+\frac{p(a_n^\delta)^2}{2}\Big)\notag\\
\times \exp\left(-\frac{a^2}{2q \|\sigma\|_\infty^2 \|\nabla \varphi\|_\infty^2}\Big(1-\frac{\sqrt{\Gamma_n}4C_\eta [f]_1\delta}{a} -\frac{1}{q\|\sigma\|_\infty^2\|\nabla \varphi\|_\infty^2}\Big\{\frac{p}{\Gamma_n}\Big(\frac{6C_{V,\varphi}^2}{c_V}+\textcolor{black}{2\big[C_{\ref{ineq_rest_b}}^\delta+ C_{\ref{ineq_rest_b2_phi}}^\delta]}(\Gamma_n^{(2)})^2  \Big)+\frac{1}{p}\Big\}\Big)\right).
\end{eqnarray*}
The Young inequality yields that for all $\varepsilon_n>0 $:
\begin{eqnarray}
\P\left[|\nu_n(f)-\nu(f)|\ge \frac{a}{\sqrt{\Gamma_n}}\right]\le 2(I_V^1)^{\frac 1p}\exp\Big(\textcolor{black}{\frac{c_V}{p}} 
+
\frac{p(a_n^\delta)^2}{2}+\varepsilon_n^{-1} \Gamma_n \delta^{2}\Big)\notag\\
\times \exp\left(-\frac{a^2}{2q \|\sigma\|_\infty^2 \|\nabla \varphi\|_\infty^2}\Big(1- \underbrace{\frac{1}{q\|\sigma\|_\infty^2\|\nabla \varphi\|_\infty^2}\Big\{2 \varepsilon_nC_\eta^2[f]_1^2+\frac{p}{\Gamma_n}\Big(\frac{6C_{V,\varphi}^2}{c_V}+\textcolor{black}{2\big[C_{\ref{ineq_rest_b}}^\delta+ C_{\ref{ineq_rest_b2_phi}}^\delta]}(\Gamma_n^{(2)})^2  \Big)+\frac{1}{p}\Big\}}_{=:d_n^\delta}\Big)\right).\notag\\
\label{BD_P_THM_2_F_DELTA}
\end{eqnarray}
We now want to let $p:=p(n)\underset{n}{\rightarrow} +\infty,\  \varepsilon_n\underset{n}{\rightarrow} 0$ so that the associated contributions in the above equation can be viewed as remainders.
From the previous definitions of $C_{\ref{ineq_rest_b2_phi}}^\delta,\ C_{\ref{ineq_rest_b}}^{\delta} $, we see that, to achieve this goal, two constraints need to 
be fulfilled: namely, we must choose $\delta,\ p $ such that
$$
\varepsilon_n^{-1} \Gamma_n \delta^2\underset{n}{\rightarrow} 0\quad\mbox{ and }\quad p(a_n^\delta)^2\underset{n}{\rightarrow} 0.
$$
Now, if $\theta\in (1/2,1] $ there exists $\beta \in (0,1)$ such that $\Gamma^{(\frac{3+\beta}{2})}\le C $. In that case: $a_n^\delta\le \frac{C}{\sqrt{\Gamma_n}}\delta^{-\beta}=\Gamma_n^{-(\frac 12(1-\beta)-\beta\varepsilon)} \underset{n}{\rightarrow} 0 
$ for $\delta=\textcolor{black}{\Gamma_n^{-(\frac 12+\varepsilon)}} 
$ and $\varepsilon<\frac{1-\beta}{2\beta} $. Taking $p:=p(n)=\Gamma_n^{(\frac 12(1-\beta)-\beta\varepsilon)}$ yields $p(a_n^\delta)^2\underset{n}{\rightarrow}0 $. On the other hand, $\varepsilon_n=\Gamma_n^{-\varepsilon} $ also yields $\varepsilon_n^{-1}\Gamma_n \delta^2=\Gamma_n^{-\varepsilon} \underset{n}{\rightarrow} 0 $.

For $\theta\in (1/3,1/2) $, $\Gamma_n^{(\frac{3+\beta}{2})} $ diverges for all $\beta\in (0,1) $, we then have $\frac{\Gamma_n^{(\frac{3+\beta}{2})}}{\sqrt{\Gamma_n}}\le C  n^{\frac 12 -\theta(1+\frac \beta 2)}$. Hence, there exists $\beta\in (0,1) $ such that $\frac{\Gamma_n^{(\frac{3+\beta}{2})}}{\sqrt{\Gamma_n}}\le C  n^{\frac 12 -\theta(1+\frac \beta 2)} \underset{n}{\rightarrow}0$. However, taking $\delta= \textcolor{black}{\Gamma_n^{-(\frac 12+\varepsilon)}}
$, which seems to be an almost ``necessary" choice to satisfy the first constraint $ \varepsilon_n^{-1} \Gamma_n \delta^2\underset{n}{\rightarrow} 0$, yields:
$$a_n^\delta\asymp \delta^{-\beta} \frac{\Gamma_n^{(\frac{3+\beta}{2})}}{\sqrt{\Gamma_n}}\asymp n^{(1+\beta)(\frac 12 -\theta)+\varepsilon \beta(1-\theta)}\underset{n}{\rightarrow}+\infty,$$
so that the second constraint cannot be fulfilled. This means that the regularization induces a constraint on the time steps which must not be too large. In other words, under the sole Lipschitz assumption on the source $f$,  \textcolor{black}{the fastest  convergence regime is out of reach}.

Summing up the previous computations, we complete the proof of Theorem~\ref{THEO_CTR_LIP}.

\mysection{Applications}
\label{APP_SECTIONS}
\subsection{Non-Asymptotic Deviation Bounds in the Almost Sure CLT}
\label{TCLPS_SEC}

Let $(U_n)_{n\ge 1} $ be an i.i.d sequence of centered $d $-dimensional  random variables  with unit covariance matrix. We define the sequence of normalized partial sums by $Z_0=0 $ and
\begin{eqnarray*}
 Z_n := \frac{\sum_{k=1}^n U_k}{\sqrt{n}}, n\ge 1.
\end{eqnarray*}
The almost sure Central Limit Theorem (denoted from now on $a.s.$\hspace{2.5pt}CLT) describes how the weighted sum of the renormalized sums $Z_n $ which appear in the \textit{usual} \textit{asymptotic} CLT, behaves viewed as a random measure. Precisely, it states that setting for $k\ge 1,\ \gamma_k=1/k $:
\begin{equation}
\label{TCL_PS}
\nu_n^Z:=\frac{1}{\Gamma_n}\sum_{k=1}^n \gamma_k \delta_{Z_k} \overset{w,\ a.s.}{\underset{n}{\longrightarrow} } G,\ G(dx):=\exp\Big(-\frac{|x|^2}{2}\Big)\frac{dx}{(2\pi)^{d/2}}.
\end{equation}
The above convergence had been established in~\cite{lamb:page:02}, as a by-product of their results concerning the approximation of invariant distributions, under the minimal moment condition $U_i\in L^2(\P) $, thus weakening the initial assumptions by Brosamler and Schatte (see~\cite{bros:88} and~\cite{scha:88}). The underlying idea is to use a reformulation of the dynamics of $(Z_n)_{n\ge 0} $ in terms of a discretization scheme appearing as a perturbation of~\eqref{scheme}.
One indeed easily checks that, for $n\ge 0$:
\begin{eqnarray}
\label{DEF_ZN}
Z_{n+1} = Z_n - \frac{\gamma_{n+1}}2  Z_n + \sqrt{\gamma_{n+1}} U_{n+1} + r_nZ_n, \ r_{n}:=\sqrt{1-\frac{1}{n+1}} -1 + \frac{1}{2(n+1)} =O\Big(\frac{1}{n^2}\Big).
\end{eqnarray}
Thus, the sequence $(Z_n)_{n\ge 0}$ appears as a perturbed Euler scheme with decreasing step $\gamma_n=\frac 1n $  of the Ornstein-Uhlenbeck process $dX_t=-\frac 12 X_tdt+dW_t$ whose invariant distribution is $G$. Then the regular Euler scheme
\begin{equation}
\label{SCHEME_OU}
X_{n+1} = X_n - \frac{\gamma_{n+1}}2  X_n + \sqrt{\gamma_{n+1}} U_{n+1},
\end{equation}
satisfies~\eqref{CV_ERGO} with $\nu=G $. The $a.s.$ weak convergence~\eqref{TCL_PS} established in~\cite{lamb:page:02} follows as a consequence of the (fast enough) convergence of $Z_n$ towards $X_n$ as $n$ goes to infinity.

Moreover, this rate is fast enough to guarantee that the sequence $\nu_n^Z$  satisfies the conclusion of Theorem~\ref{theo} point \textit{(a)} (when $ \gamma_n=\frac 1n, \frac{\Gamma_n^{(2)}}{\sqrt{\Gamma_n}}\underset{n}{\rightarrow }0$), {\it i.e.} its convergence rate is ruled by a CLT at rate $\sqrt{\log(n)} $. In fact this holds under a lower moment assumption $U_1\in L^3(\P) $.


Let us mention that the convergence rates related to the $a.s.$ CLT had already been investigated by several authors. Let us quote among relevant works, Cs\"org\H{o} and Horv\'ath~\cite{csor:horv:92}, for real valued i.i.d. \!random variables, Cha\^abane and Ma\^aouia~\cite{chaa:maao:00}, who investigate the convergence rate of the strong quadratic law of large numbers for some extensions to vector-valued martingales,  and Heck~\cite{heck:98}, for large deviation results. 
As an application of our previous results, we will derive some new non-asymptotic Gaussian deviation bounds for the $a.s.$ CLT, when the involved random variables $(U_n)_{n\ge 1} $  satisfy \A{GC}.
We insist here that the sub-Gaussianity of the innovations is crucial to get a non-asymptotic Gaussian deviation bound. 
{The result  readily extends to the wider class of  innovations satisfying the general sub-Gaussian exponential deviation inequality~\eqref{GEN_GC}. Also, we slightly  weaken  the regularity assumptions needed on the function $f$ in~\cite{lamb:page:02} for the associated $a.s.$ CLT to hold}.

\subsubsection{Non-Asymptotic Deviation Bounds.}
\begin{theo}\label{ineq_TCLps}
Assume the innovation sequence $(U_n)_{n \geq 1}$ satisfies \A{GC} and let $f$ be a Lipschitz continuous function such that  $G(f)=\int_{\R^d}f(x)G(dx)=0 $. Then, 
there exist two explicit monotonic sequences 
$c_n\le 1\le C_n,\ n\ge 1$,   with $\lim_n C_n = \lim_n c_n =1$
such that for all  $a >0$ and $n\ge 1$:
\begin{eqnarray}
\P\Big[\sqrt{\log( n)+1} |\nu_n^Z(f)| \geq a   \Big] &\leq& 2 C_n \exp\left( - c_n \frac{a^2}{2\|\nabla \varphi\|_\infty^2} 
\right),
\end{eqnarray}
where $\varphi$ denotes the solution of the Poisson equation:
\begin{equation}
\label{POISSON_BIS}
\forall x \in \R^d,\ \frac{1}{2} \Delta \varphi(x) - \frac 12 x \cdot \nabla \varphi(x) = f(x),
\end{equation}
which, under the current assumptions,  is unique and belongs to $W_{p,{\rm loc}}^2(\R^d,\R)$, for any $p>1$, with $\|\nabla \varphi\|_\infty\le 2[f]_1 $. 
\end{theo}
\noindent\textit{Proof.}
For $(Z_n)_{n\ge 0} $  as in~\eqref{DEF_ZN}, and $(X_n)_{n\ge 0}$ as in~\eqref{SCHEME_OU} we introduce:
\begin{eqnarray*}
\Delta_n &:= &Z_n - X_n.
\end{eqnarray*}
With the definition of $\nu_n^Z $ in~\eqref{TCL_PS}, write $ \nu_n^Z (f)=  \frac{1}{\Gamma_n} \sum_{k=1}^n \gamma_k f(Z_{k-1})$. We also have similarly $\nu_n^X(f):=\frac{1}{\Gamma_n} \sum_{k=1}^n \gamma_k f(X_{k-1}) $. For all $\lambda >0$, we derive similarly to~\eqref{DECOUP_TCHEB} \textcolor{black}{(see as well~\eqref{DEV_PROV_LIP}) and with the notations of~\eqref{DEV_REG_DELTA}}:
\begin{eqnarray}
\P\big[\sqrt{\Gamma_n} |\nu_n^Z(f)|  \geq a \big] &=& \P\Big[ \sqrt{\Gamma_n}\Big|\frac{1}{\Gamma_n} \sum_{k=1}^n \gamma_k \big( f(Z_{k-1}) - f(X_{k-1}) \big) +  \nu_n^X (f) \Big| \geq a  \Big]
\nonumber\\
&\leq & \P\Big[ \sqrt{\Gamma_n}\Big|\frac{1}{\Gamma_n} \sum_{k=1}^n \gamma_k \big( f(Z_{k-1}) - f(X_{k-1}) \big) +  \nu_n^X ({\mathcal A}\varphi_\delta)+R_{n,\delta}(f) \Big| \geq a  \Big]
\nonumber\\
&\leq & 2\exp\Big( - \frac{\lambda a}{\sqrt{\Gamma_n}} \Big(1 -\frac{2\sqrt{\Gamma_n}C_\eta[f]_1\delta}a\Big) \Big) \left(\E \exp\Big( \bar p[f]_1 \lambda \nu_n^{\Delta}(|\cdot|)\Big)\right)^{\frac 1{\bar p}}\left( \E\exp\Big( \textcolor{black}{-\frac{q \bar q \lambda}{\Gamma_n} M_n}  \Big)\right)^{\frac 1{q\bar q}}\nonumber\\
&&\times \left(\E \exp\Big( \frac{2p \bar q\lambda}{ \Gamma_n}(|L_n|+|\bar G_n|)\Big)\right)^{\frac 1{2p\bar q}}\left(\E\exp\Big( \frac{4p\bar q \lambda}{ \Gamma_n}|D_{2,b,n}|\Big)\right)^{\frac 1{4p\bar q}}\left(\E\exp\Big( \frac{4p \bar q \lambda}{ \Gamma_n}|D_{2,\Sigma,n}|\Big)\right)^{\frac1{4p\bar q}} \label{ineq_WITH_Delta}
\end{eqnarray}
for $\bar q, q\in (1,+\infty$), $\bar p=\frac{\bar q}{\bar q-1},p=\frac{q}{q-1} $. Also, $\varphi_\delta $ corresponds to the solution of the Poisson equation~\eqref{POISSON_BIS}
 obtained replacing $f$ by its mollified version $f_\delta $.
Now, we need the following lemma to control $\nu_n^{\Delta}(|\cdot|):=\frac{1}{\Gamma_n}\sum_{k=1}^n \gamma_k |\Delta_{k-1}| $.

\begin{lemme}\label{LEMME_TEC}
There is a non-negative constant $C_{\ref{ineq_Delta}}$  such that for all $\lambda>0 $:
\begin{equation}\label{ineq_Delta}
\E\exp\Big(\lambda\nu_n^\Delta (|\cdot|) \Big)=\E\exp\Big(\frac{\lambda}{\Gamma_n} \sum_{k=1}^n \gamma_k |\Delta_{k-1}| \Big)  \le \exp\left( \frac{C_{\ref{ineq_Delta}} \lambda \E\,[|U_1|] \Gamma_n^{(\frac{3}{2})}}{\Gamma_n}+ \frac{C_{\ref{ineq_Delta}}^2 \lambda^2 \Gamma_n^{(3)}}{2\Gamma_n^2}\right).
\end{equation}
\end{lemme}

For clarity, we postpone the proof to the end of the current section.

On the other hand, from Section~\ref{REG_LIP_SOURCE} we have that $\varphi_\delta\in {\mathcal C}^{3,\beta}(\R^d,\R) $ 
for all $\beta\in (0,1) $. We derive from~\eqref{ineq_WITH_Delta},~\eqref{ineq_Delta} similarly to the proof of Theorem~\ref{THEO_CTR_LIP} by setting $\bar \lambda_n:=\frac{a\sqrt{\Gamma_n}}{q\bar q \|\nabla \varphi\|_\infty^2} $:
\begin{eqnarray*}
\P\big[\sqrt{\Gamma_n} |\nu_n^Z(f)|  \geq a  \big]&\le&2 \exp\Big( - \frac{a^2  
}{2q\bar q\|\nabla \varphi\|_\infty^2}\Big(1-\frac{\textcolor{black}{4 \sqrt{\Gamma_n}}C_\eta [f]_1  \delta}{a}\Big)  \Big)\exp\Big( \frac{C_{\ref{ineq_Delta}}\bar \lambda_n [f]_1 \E\,[|U_1|] \Gamma_n^{(\frac 32)}}{\Gamma_n} \Big)\\
&&\exp\Big(\frac{C_{\ref{ineq_Delta}}^2\bar p [f]_1^2\bar \lambda_n^2 \Gamma_n^{(3)}}{2\Gamma_n^2}\Big)
(I_V^1)^{\frac {1}{p\bar q}}\exp\Big(\frac{1}{p\bar q}\big( c_V+\frac{C_{\ref{ineq_rest_b2_phi}}^\delta }{2}\big) +\frac{(pa_n^\delta)^2}{2\bar q}\Big)
\\
&&\times  
\exp\Big( \bar \lambda_n^2  \Big(p\bar q \Big(\frac{3C_{V,\varphi}^2}{c_V \Gamma_n^2}+\big[C_{\ref{ineq_rest_b}}^\delta+\frac 32 C_{\ref{ineq_rest_b2_phi}}^\delta \big] \frac{(\Gamma_n^{(2)})^2}{\Gamma_n^2}\Big)+\frac{1}{2p\bar q}\Big)\Big)\\
&\le &2 (I_V^1)^{\frac {1}{p\bar q}}\exp\Big(\frac{1}{p\bar q}\big( c_V+\frac{C_{\ref{ineq_rest_b2_phi}}^\delta}{2}\big) +\frac {p(a_n^\delta)^2}{2\bar q}+\textcolor{black}{\varepsilon_n^{-1}
\Gamma_n \delta^2
}\Big)\\
&&\times \exp\left( - \frac{a^2}{2 q\bar q \|\nabla \varphi\|_\infty^2}\Big(1-d_n^\delta- \frac{\bar p}{q\bar q\|\nabla \varphi\|_\infty^2}\frac{ [f]_1^2C_{\ref{ineq_Delta}}^2 \Big( \Gamma_n^{(3)}+\E\,[|U_1|]^2(\Gamma_n^{(\frac 32)})^2\Big)}{\Gamma_n}\Big) \right),  
\end{eqnarray*}
for $\varepsilon_n>0 $ and $d_n^\delta$ as in~\eqref{BD_P_THM_2_F_DELTA}. Choose again $(p_n)_{n\ge 1} $ and $\delta $ as in Section~\ref{REG_LIP_SOURCE} so that
 $q_n\underset{n}{\rightarrow} 1, d_n^\delta\underset{n}{\rightarrow} 0 $ with the indicated monotonicity for $n$ large enough.
We can now take $\bar p:=\bar p_n\underset{n}{\rightarrow}+\infty $ such that $\frac{\bar p}{\Gamma_n} \underset{n}{\rightarrow}0$.
The above inequality then gives the result up to a direct modification of the sequences $(C_n)_{n\ge 1},  (c_n)_{n\ge 1} $.\hfill $\square $

\subsubsection*{Proof of Lemma~\ref{LEMME_TEC}}

%
%
 The definition of $\Delta_n$ implies:
  \begin{eqnarray*}
 \Delta_{n+1} = \Delta_n \Big( 1- \frac{\gamma_{n+1}}{2}\Big) + r_n Z_n,
 \end{eqnarray*}
where we recall from~\eqref{DEF_ZN} that $r_n:=\sqrt{1-\frac{1}{n+1}} -1 + \frac{1}{2(n+1)} =O(\frac{1}{n^2})$. In particular, there exists $\bar C_{1}>0$ such that for all $n \geq 1 $,
 \begin{equation}\label{rn}
| r_n| \leq \frac{\bar C_1}{n^2}.
\end{equation}

Setting now $\rho_0=1 $ and for $n\ge 1$:
\begin{eqnarray*}
\rho_n := \Big[ \prod_{k=1}^n (1 - \frac{\gamma_k}{2})\Big]^{-1}= \prod_{k=1}^n \frac{2k}{2k-1},
\end{eqnarray*}
a direct  induction on $\Delta_n$ yields:
 \begin{equation}\label{defDelat}
 \Delta_n = \frac{1}{\rho_n} \sum_{k=1}^n r_k \rho_k Z_k
=\frac{1}{\rho_n} \sum_{k=1}^n r_k \rho_k\Big ( \sum_{l=1}^k \frac{U_l}{\sqrt{k}} \Big ) = \frac{1}{\rho_n} \sum_{l=1}^n \Big ( \sum_{k=l}^n \frac{r_k \rho_k}{\sqrt{k}} \Big ) U_l.
 \end{equation}
Also, from the Wallis formula $\rho_n \sim_n \sqrt{\pi n}$, which implies that there exists $\bar C_{2}\ge 1$ such that for all $n \geq 1 $:
\begin{equation}\label{cn}
\bar C_{2}^{-1} \sqrt n\le   \rho_n \le \bar C_{2} \sqrt{n}.
\end{equation} 

We now get from~\eqref{defDelat} and  the Fubini theorem:
\begin{equation}\label{sum_Delta}
\Gamma_n \nu_n^\Delta(|\cdot|)= \sum_{k=1}^n \gamma_k |\Delta_{k-1}| \le  \sum_{k=1}^n \frac{\gamma_k}{\rho_{k-1}} \sum_{l=1}^{k-1} \Big ( \sum_{m=l}^{k-1} \frac{|r_m| \rho_m}{\sqrt{m}} \Big ) |U_l|
 =\sum_{l=1}^{n-1} \Big[ \sum_{k=l+1}^{n}  \frac{\gamma_k}{\rho_{k-1}}  \big ( \sum_{m=l}^{k-1} \frac{|r_m| \rho_m}{\sqrt{m}} \big ) \Big] |U_l|.
 \end{equation}

 Combining~\eqref{rn} and~\eqref{cn}, we get that there exist constants $\bar C_{3},\bar C_{4}>0$ 
 such that for all $k \in \leftB l+1, n \rightB$.
\begin{equation}\label{sum_cn}
\frac{\gamma_k}{\rho_{k-1}}\sum_{m=l}^{k-1} \frac{|r_m| \rho_m}{\sqrt{m}} \leq \frac{\bar C_3}{k^{3/2}l},\  \sum_{k=l+1}^n \frac{\gamma_k}{\rho_{k-1}}\sum_{m=l}^{k-1} \frac{|r_m| \rho_m}{\sqrt{m}}  \leq \frac{\bar C_4}{l^{3/2}}.
\end{equation}
 Plugging this inequality in~\eqref{sum_Delta}, we derive:
 \begin{equation}
\label{CTR_NU_N_DELTA}
 \nu_n^\Delta(|\cdot|)\le \frac{1}{\Gamma_n} \sum_{l=1}^{n-1} \Big [\sum_{k=l+1}^n \frac{\gamma_k}{\rho_{k-1}}\sum_{m=l}^{k-1} \frac{|r_m |\rho_m}{\sqrt{m}} \Big ] |U_l  | \leq \frac{\bar C_4}{\Gamma_n} \sum_{l=1}^{n-1} \frac{|U_l|}{l^{3/2}}.
\end{equation}
For any $\lambda>0$, Equation~\eqref{CTR_NU_N_DELTA} and the Gaussian concentration property \A{GC} of the innovation entail:
\begin{eqnarray*}
\E\exp\Big( \lambda\nu_n^\Delta(|\cdot|)\Big)
&\leq &\prod_{k=1}^{n-1} \E \exp \Big ( \frac{C_4 \lambda}{\Gamma_n k^{\frac{3}{2}}} |U_k| \Big ) 
\leq
\prod_{k=1}^{n-1}\exp \Big ( \frac{\bar C_4 \lambda}{\Gamma_n k^{\frac{3}{2}}} \E\,[|U_1|] + \frac{1}{2} \big(\frac{\bar C_4 \lambda}{\Gamma_n k^{\frac{3}{2}}}\big)^2 \Big )
\\
&
=
&
\exp\Big( \frac{\bar C_4 \lambda \E\,[|U_1|] \Gamma_n^{(\frac{3}{2})}}{\Gamma_n}+ \frac{\bar C_4^2 \lambda^2 \Gamma_n^{(3)}}{2\Gamma_n^2}\Big).
\end{eqnarray*}
This completes the proof.\hfill $\square $


\subsection{Numerical Results}
\label{SEC_NUM}

We present in this section numerical results associated with the computation of the empirical measure $\nu_n $ illustrating our previous theorems.  

\subsubsection{Sub-Gaussian tails}\label{NUM_1}
We first consider $d=r=1$. Also, for simplicity, the innovations $(U_i)_ {i \geq 1}$ and $X_0$ are Bernoulli variables with $\P(U_1=-1)=\P(U_1=-1)=\frac{1}{2}$. 
We illustrate here Theorem~\ref{ineq_fin} taking  
$b(x)=-\frac{x}{2}$, and $\sigma(x)=\cos(x)$ in ~\eqref{eq_diff}.  This is a (weakly) hypoelliptic example. Indeed, setting for $x\in \R,\ X_1(x)=\cos(x)\partial_x $ and $X_0(x)=-\frac x2\partial_x $, we have ${\rm span}\{ X_1, [X_1,X_0]\}=\R $. 
We choose as well to compute $\nu_n({\mathcal A}\varphi) $ for $ \varphi(x) =x+\varepsilon \cos(x)$ for $\varepsilon=0.01 $, and $\varphi\textcolor{black}{(x)}= \cos(x)$. \textcolor{black}{The function $\varphi $ is here given}. 
The assumptions of Theorem~\ref{ineq_fin} follow from Theorem 18 in Rotschild and Stein~\cite{roth:stei:76} (up to the introduction of a suitable partition of unity).
From Theorem~\ref{ineq_fin}, for steps of the form $(\gamma_k)_{k\ge 1}=(k^{-\theta})_{k\ge 1}, \ \theta\in [1/3,1] $ (corresponding to $\beta=1 $ in Theorem~\ref{ineq_fin}), the function $a\in \R^+\mapsto g_{n,\theta}(a):=\begin{cases}
\log\big(\P[  |\sqrt{\Gamma_n} \nu_n(\mathcal{A}\varphi) | \geq a ]\big), \ \theta\in (1/3,1],\\
\log\big(\P[  | \sqrt{\Gamma_n}\nu_n(\mathcal{A}\varphi)+({\mathcal B}_{n,1}-E_n^1)| \geq a
]\big),\ \theta=1/3,
\end{cases}$ is such that for 
$a>a_n:=a_n(\theta) $ where for $\theta\in (1/3,1] $, $a_n(\theta)=0$ and
 for $\theta=1/3 $, $a_n(\theta)=\frac{ [\varphi^{(3)}]_{\beta}  \|\sigma\|_\infty^{(3+\beta)} \E\big[|U_1|^{3+\beta} \big] }{(1+\beta)(2+\beta)(3+\beta) }   \frac{\Gamma_n^{(\frac{3 +\beta}{2})}}{\sqrt{\Gamma_n}}$:
$$g_{n,\theta}(a)\le -c_n \frac{(a-a_n)^2}{2 \| \sigma \| _{\infty}^2 \| \nabla \varphi \|_{\infty}^2}+\log(2C_n).$$ 
We  plot in Figure~\ref{Nappe_hypo} the curves of $g_{n,\theta}$ for $\theta$ varying as $\theta_j= \frac{1}{3}+(1-\frac{1}{3})\frac{j}{5}$, for $j\in [\![1,5 ]\!],\ \varphi(x)=x+\varepsilon \cos(x)$ and in Figure~\ref{Nappe_hypo_Lp} the curve of $g_{n,\theta}$ for $\theta=\theta_0= \frac{1}{3}$ and $\varphi(x)=\cos(x) $.
The simulations have been performed for $n=5 \times 10^4$ in Figure~\ref{Nappe_hypo}, $n=5\times 10^6$ in Figure~\ref{Nappe_hypo_Lp},  and the probability estimated by Monte Carlo simulation for $MC=10^4$ realizations of the random variable $ |\sqrt{\Gamma_n} \nu_n(\mathcal{A}\varphi) | $ in the unbiased case and in the biased case of the random variable $|\sqrt{\Gamma_n} \nu_n(\mathcal{A}\varphi)+({\mathcal B}_{n,1}-E_n^1)^M |$, where $({\mathcal B}_{n,1}-E_n^1)^M$ is obtained from ${\mathcal B}_{n,1}-E_n^1 $ replacing the integral over [0,1], that needs to be evaluated at every time step, by a quantization of the uniform law on $[0,1] $ with $M=10$ points. We refer to \cite{graf:lush:00} or \cite{page:97} for details on quantization. We point out that this is one drawback that appears to obtain the fastest convergence rate, the bias needs to be estimated and therefore the function $\varphi $ in some sense known (since the approximation of the bias requires to compute its derivatives). The corresponding $95\% $ confidence intervals have size at most of order $0.016$. 
To compare with, we also introduce the functions 
$S_{n,\theta}(a):=-\frac{(a-a_n(\theta))^2}{2 \| \sigma \| _{\infty}^2 \| \nabla \varphi \|_{\infty}^2}$, $S_{n,\theta,c}(a):=-\frac{(a-a_n(\theta))^2}{2 \nu_{n_c}(\sigma^2 ) \| \nabla \varphi \|_{\infty}^2}$, $S_{n,\theta,A}(a):=-\frac{(a-a_n(\theta))^2}{2 \nu_{n_c}( | \sigma  \nabla \varphi |^2 )} $ and the optimal concentration $P(\lambda_{\min})(n,\theta,a,\rho)$, 
obtained in Remark~\ref{LA_RQ_QUI_CONCENTRE}, optimizing numerically in $\rho $. The quantities $\nu_{n_c}(\sigma^2), \nu_{n_c}( | \sigma  \nabla \varphi|^2) $ in the previous expressions actually correspond to the numerical estimation, for $n_c=10^4$ and 
$(\gamma_k^c)_{k\ge 1} = (k^{-\theta^c})_{k\ge 1}$ with $ \theta^c=\frac 13+10^{-3} $, of $\nu(\sigma^2), \nu(|\sigma \nabla \varphi|^2) $ appearing respectively in the sharper concentration bound of Theorem~\ref{THM_COBORD} when $\sigma^2 - \nu(\sigma^2)$ is a coboundary and in the asymptotic Theorem~\ref{theo}.  
In the unbiased case of Figure~\ref{Nappe_hypo}, 
we plot the maximum in $j $ of the $(S_{n,\theta_j})_{j\in \leftB 1,5\rightB},(S_{n, \theta_j,c})_{j\in \leftB 1,5\rightB}, (S_{n, \theta_j,A})_{j\in \leftB 1,5\rightB}, \big(P(\lambda_{\min})(n,\theta_j,a,\rho)\big)_{j\in \leftB 1,5\rightB}$ corresponding to $j=1$. The associated curves are denoted by $S_n,S_{n,c}, S_{n,A} $ and $P(\lambda_{\min})(n) $.

The Figures~\ref{Nappe_hypo} and~\ref{Nappe_hypo_Lp} correspond to \textcolor{black}{the} unbiased and biased cases respectively. In the \textcolor{black}{un}biased case, we observe that the curves almost overlay, the optimal deviation rate $P(\lambda_{\min}) $ is very close to the empirical data. It is also below the numerical estimation of the asymptotic threshold given by $S_{n,\theta,A} $ which is, for our considered example, almost indistinguishable from the coboundary $S_{n,\theta,c}$ (indeed, since $\varepsilon=0.01, \|\nabla \varphi\|_\infty^2\le 1+\varepsilon^2$ and  $\nu(\sigma^2)\|\nabla \varphi\|_\infty^2 \simeq \nu ( |\sigma \nabla \varphi|^2) $) and far below from the bounds of $S_{n,\theta} $. 
In the biased case, 
$P(\lambda_{\min}) $ stays very close to the theoretical asymptotic bound given by $S_{n,\theta,A} $ up to a certain deviation level $a$, namely for $a\in [0, 0.5]$. It then remains the best bound provided by our results. In this example, the improvement associated with $S_{n,\theta,c}$ is also notable. It is precisely because the source term has a more oscillating gradient that we have also considered a larger running time, corresponding to $n=10^6 $, for the empirical curves. For this choice, we see relatively good agreement w.r.t. to the asymptotic deviation bounds of $S_{n,\theta_0,A} $.

The figures below thus illustrate that the explicit \textit{optimal rate} of Remark~\ref{LA_RQ_QUI_CONCENTRE} seems rather appropriate to capture the deviations of the empirical random measures.

\begin{figure}[!h]
\begin{minipage}[b]{.46\linewidth}
\includegraphics[scale=.75]{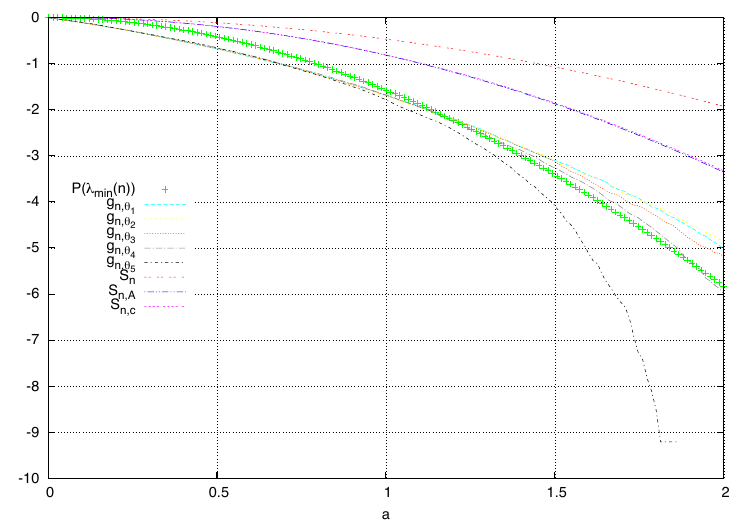}
\caption{Unbiased Case. Plot of $a \mapsto g_{n,\theta}(a)$, for $(\theta_k) 
_{k\in [\![1,5 ]\!]}$, with $\varphi(x)\!=\!\sigma(x) \!=\!x+\varepsilon \cos(x),\,\varepsilon=0.01$.}
\label{Nappe_hypo}
\end{minipage}\hfill
\begin{minipage}[b]{.46\linewidth}
\includegraphics[scale=.75]{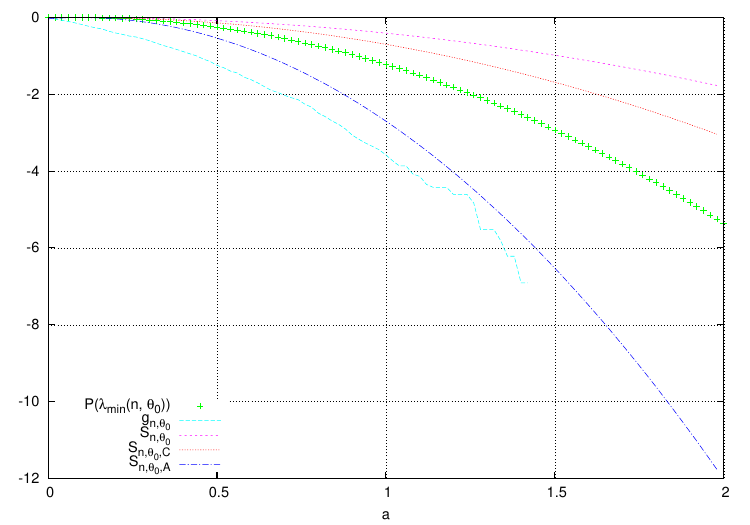}
\caption{Biased Case. Plot of $a \mapsto g_{n,\theta}(a)$, for $\theta_0= \frac{1}{3}$, with $\varphi(x)=\sigma(x) = \cos(x)$.}
\label{Nappe_hypo_Lp}
\end{minipage}
\end{figure}
We eventually plot below the deviation curves with source $\varphi(x)=\cos(x) $ adding a last curve obtained replacing in the formula for $P(\lambda_{\min}) $ of Remark~\ref{LA_RQ_QUI_CONCENTRE} the $\|\nabla \varphi\|_\infty^2\nu(\sigma^2) $ by $\nu(|\sigma \nabla \varphi|^2)  $. For practical purposes, this last quantity is again estimated numerically with the same previous parameters. 
\begin{figure}[!h]
\begin{minipage}[b]{.46\linewidth}
\includegraphics[scale=.75]{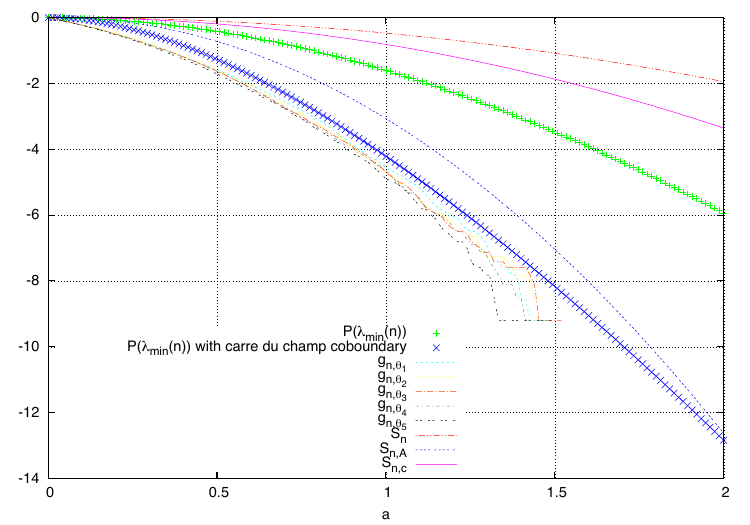}
\caption{Plot of $a \mapsto g_{n,\theta}(a)$, for $(\theta_k)_{k\in [\![1,5 ]\!]}$, with $\varphi(x)\!=\!\sigma(x) \!=\! \cos(x)$.}
\label{COURBE_FATALE}
\end{minipage}
\end{figure}
Even if the analysis of Theorem~\ref{THM_COBORD} cannot be extended to justify such a choice, the empirical evidence is rather striking. 

\subsubsection{Slutsky like result}
In this paragraph, we illustrate our results from Theorem~\ref{Slutsky_theorem}, which can be viewed as an extension of the usual  Slustky's Lemma to our current framework, for a multidimensional process, precisely for $r=d=2$ in the case $\beta\in (0,1) $.
In order to converge as fast as possible without bias, we take $\theta=\frac 1{2+\beta}+\frac 1{1000} $.
We also choose a  model which satisfies the assumptions of Theorem~\ref{Poisson_regular} \textcolor{black}{under \A{C${}_{\rm \mathbf{UE}}$}} and Lemma~\ref{PT_GR_BD}. We consider:
\begin{equation*}
f(x)=\frac{|x|^{1+\beta}}{1+|x|^\beta}, \
b(x)=
\begin{pmatrix}
-4x_1+6x_2 \\
-5x_1-5x_2
\end{pmatrix}, 
\ 
\sigma \sigma^*(x)=
\begin{pmatrix}
\frac{\cos(x_1+x_2)}2+1 & \frac{\sin(x_1)\sin(x_2)}4\\
\frac{\sin(x_1)\sin(x_2)}4 &1- \frac{\sin(x_2)}2
\end{pmatrix}.
\end{equation*}
Remark that the non-degeneracy condition \A{UE} is fulfilled by $\Sigma=\sigma \sigma^*$, 
as well as the condition set in Theorem~\ref{Poisson_regular} \textcolor{black}{under \A{C${}_{\rm \mathbf{UE}}$}}, $\Sigma_{i,j}(x)=\Sigma_{i,j}(x_i, \hdots, x_d)$, for all $1 \leq i \leq j \leq d$.
Furthermore, from the Cholesky decomposition, we write:
 \begin{equation*}
\sigma (x)=
\begin{pmatrix}
\sqrt{\frac{\cos(x_1+x_2)}2+1 }&0\\
\frac{\sin(x_1)\sin(x_2)}{4 \sqrt{\frac{\cos(x_1+x_2)}2+1 }} & \sqrt{- \frac{\sin(x_1)^2 \sin(x_2)^2}{16 (\frac{\cos(x_1+x_2)}2+1)}+ 1- \frac{\sin(x_2)}2}
\end{pmatrix}.
\end{equation*}
Let us check that \A{D${}_\alpha^p$} is satisfied. Firstly, remark that $\frac{Db+Db^*}2$ is a constant matrix whose  eigenvalues are $\{-\frac{\sqrt{2}+9}2, \frac{\sqrt{2}-9}2 \}$. Direct computations yield that,
for all $x\in \R^d $, $\xi\in \R^d $:
\begin{equation*}
\left \langle \frac {Db(x)+Db(x)^*} 2\xi,\xi \right\rangle +\frac 12 \sum_{j=1}^r |D\sigma_{\cdot j} (x) 
\xi|^2 \le 
-3.085 |\xi|^2.
\end{equation*}
It can be checked similarly that the condition $\|D\sigma\|_\infty^2 \le \frac{2\alpha}{2(1+\beta)-p} $ is satisfied for $\alpha=3.085$ and $ \beta=.5 $ which we consider below.
Also, the condition \A{R${}_{1,\beta}$} clearly holds.
In other words, all assumptions of Theorem~\ref{Slutsky_theorem} are in force.
We set for the following plot:
\begin{equation*}
g^\sigma_n(a)= \log \P[\sqrt{\Gamma_n} |\nu_nf) | \geq a ], \ S_\sigma(a)= - \frac{a^2 \alpha^2}{2 [f]_1^2},
\end{equation*}
with $\alpha=3.085$, and $[f]_1=1$.

Unlike in the previous simulations, we do not know here the value of $\nu(f)$. In fact, in paragraph~\ref{NUM_1} we had chosen to compute the deviation of $\mathcal A \varphi$ from $0= \nu(\mathcal A \varphi)$.
Here, we estimate from the ergodic theorem $\nu(f)$, taking $\beta=.5 $,  by $\nu_{n^c}(f) \approx 0.71308$ for $n^c=5\cdot 10^5$. Running $MC=10^2 $ samples, we find that the size of the associated $95\% $ confidence interval is $3.208 \cdot10^{-4}$. 
Finally, the simulations are performed for $n=5\times 10^4$, and the probability is calculated by Monte Carlo algorithm for $MC=10^3$ realizations.
The maximum size of the associated $95\% $ confidence interval is $4.75054\cdot 10^{-5}$.
The innovations are Gaussian random variables.
\begin{figure}[!h]
\centering
\begin{minipage}[b]{.46\linewidth}
\centerline{\includegraphics[scale=.65]{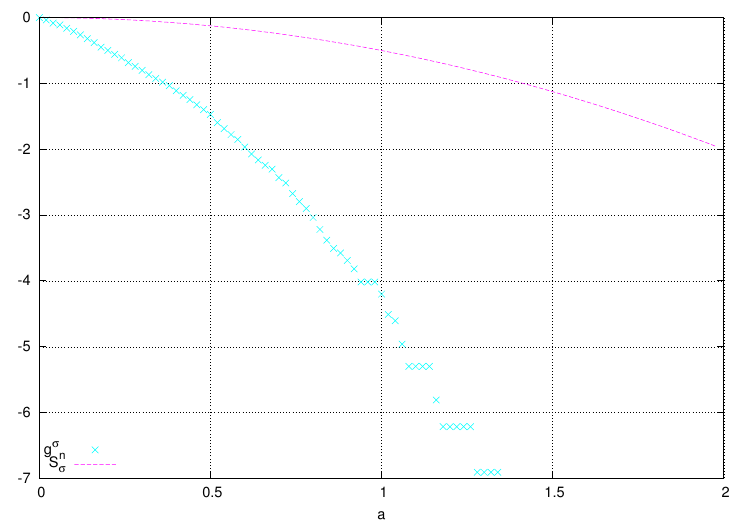}}
\caption{Plot of $a \mapsto g_{n}(a)$ with $f(x)=\frac{|x|^\beta}{1+|x|^\beta},\ \beta=.5$.}
\label{Nappe_Slutsky}
\end{minipage}
\end{figure}

In Figure~\ref{Nappe_Slutsky}, we observe that the curve $S_\sigma$ stays above $g^\sigma_n$ as proved in Theorem~\ref{Slutsky_theorem}.
However, remark that the graphs are quite spaced. This can be explained, among other things, by the difference between $\nu(\|\sigma\|^2) \frac{[f]_1}\alpha$ and the asymptotic variance $\nu(|\sigma^* \nabla \varphi|^2)$. Furthermore we have represented $S_\sigma$ which is a kind of asymptotic version of $P(\lambda_{\min}(n))$ in the previous plots.

\section*{Acknowledgments}
For the second author, the article was prepared within the framework of a subsidy granted to the HSE by the Government of the Russian Federation for the implementation of the Global Competitiveness Program.


\bibliographystyle{alpha}
\small
\bibliography{bibli}

\end{document}